\documentclass[a4paper,11pt]{article}
\usepackage[utf8]{inputenc} 
\usepackage[english]{babel}
\usepackage[T1]{fontenc}   
\usepackage{graphicx}
\usepackage{comment}
\usepackage{amsthm,amsmath,amsfonts,stmaryrd,mathrsfs,amssymb}
\usepackage{mathtools}
\usepackage{enumitem}
\usepackage{physics}
\usepackage{chemarrow}
\usepackage{tikz}
\usepackage[top=2.5cm, bottom=2.5cm, left=2cm, right=2cm]{geometry}
\usepackage{subfig}
\usepackage{multirow}
\usepackage{array}
\usepackage{bigints}
\usepackage{etoolbox}
\usepackage{mdframed}
\usepackage{color}
\usepackage[backend=bibtex,style=trad-plain,abbreviate=false,doi=false,isbn=false,url=false]{biblatex}
\DeclareSourcemap{
	\maps[datatype=bibtex]{
		\map[overwrite=true]{
			\step[fieldsource=fjournal]
			\step[fieldset=journal, origfieldval]
		}
	}
}

\renewbibmacro*{journal}{%
	\printfield{journaltitle}%
	\setunit{\subtitlepunct}%
	\printfield{journalsubtitle}}
\usepackage[unicode,colorlinks=true,linkcolor=blue,citecolor=red,pdfencoding=auto,psdextra]{hyperref}
\setenumerate[1]{label=(\roman*), font = \normalfont}
\makeatletter
\newcommand{\mylabel}[2]{#2\def\@currentlabel{#2}\label{#1}}
\makeatother

\newcommand{\ensemblenombre}[1]{\mathbb{#1}}
\newcommand{\N}{\ensemblenombre{N}}
\newcommand{\Z}{\ensemblenombre{Z}}
\newcommand{\Q}{\ensemblenombre{Q}}
\newcommand{\R}{\ensemblenombre{R}}

\renewcommand{\P}{\mathbb{P}}
\newcommand{\E}{\mathbb{E}}
\newcommand{\D}{\mathbb{D}}
\newcommand{\M}{\mathbb{M}}
\newcommand{\K}{\mathbb{K}}
\newcommand{\bU}{\mathbb{U}}
\newcommand{\Ec}[1]{\mathbb{E} \left[#1\right]}

\newcommand{\Pp}[1]{\mathbb{P} \left(#1\right)}

\newcommand{\Ecsq}[2]{\mathbb{E} \left[#1\mathrel{}\middle|\mathrel{}#2\right]}

\newcommand{\Ppsq}[2]{\mathbb{P} \left(#1\mathrel{}\middle|\mathrel{}#2\right)}

\newcommand{\Var}[1]{\mathrm{Var}\left(#1\right)}

\newcommand{\intervalle}[4]{\mathopen{#1}#2
\mathclose{}\mathpunct{},#3
\mathclose{#4}}
\newcommand{\intervalleff}[2]{\intervalle{[}{#1}{#2}{]}}
\newcommand{\intervalleof}[2]{\intervalle{(}{#1}{#2}{]}}
\newcommand{\intervallefo}[2]{\intervalle{[}{#1}{#2}{)}}
\newcommand{\intervalleoo}[2]{\intervalle{(}{#1}{#2}{)}}
\newcommand{\intervalleentier}[2]{\intervalle\llbracket{#1}{#2}
\rrbracket}
       
\newcommand{\petito}[1]{o\mathopen{}\left(#1\right)}

\newcommand{\grandO}[1]{O\mathopen{}\left(#1\right)}

\newcommand{\grandOdom}[2]{O_{#1}\mathopen{}\left(#2\right)}
\newcommand{\petitodom}[2]{o_{#1}\mathopen{}\left(#2\right)}

\newcommand{\enstq}[2]{\left\lbrace#1\mathrel{}\middle|\mathrel{}#2\right\rbrace}

\newcommand{\ind}[1]{\mathbf{1}_{\left\lbrace #1 \right\rbrace}}

\newcommand{\cG}{\mathcal{G}}

\newcommand{\bL}{\mathbb{L}}

\newcommand{\cF}{\mathcal{F}}

\newcommand{\ttT}{\mathtt{T}}
\newcommand{\ttP}{\mathtt{P}}
\newcommand{\ttG}{\mathtt{G}}

\newcommand{\Gam}[1]{\Gamma \left(#1\right)}

\DeclareMathOperator{\Log}{Log}
\DeclareMathOperator{\cst}{cst}
\DeclareMathOperator{\haut}{ht}

\DeclareMathOperator{\dist}{d}

\DeclareMathOperator{\wrt}{WRT}
\DeclareMathOperator{\pa}{PAT}

\DeclareMathOperator{\MLMC}{MLMC}

\newcommand{\bigabs}[1]{\biggr\lvert#1\biggr\rvert}

\newmdtheoremenv{theorem}{Theorem}

\newtheorem{proposition}[theorem]{Proposition}
\newtheorem{lemma}[theorem]{Lemma}

\newtheorem{corollary}[theorem]{Corollary}

\newtheorem{remark}[theorem]{Remark}

\begin{document}
\title{Geometry of weighted recursive and affine preferential attachment trees }
\author{Delphin Sénizergues}
\maketitle
\abstract{We study two models of growing recursive trees. For both models, initially the tree only contains one vertex $u_1$ and at each time $n\geq 2$ a new vertex $u_n$ is added to the tree and its parent is chosen randomly according to some rule.
In the \emph{weighted recursive tree}, we choose the parent $u_k$ of $u_n$ among $\{u_1,u_2,\dots, u_{n-1}\}$ with probability proportional to $w_k$, where $(w_n)_{n\geq1}$ is some deterministic sequence that we fix beforehand. 
In the \emph{affine preferential attachment tree with fitnesses}, the probability of choosing any $u_k$ is proportional to $a_k+\deg^{+}(u_k)$, where $\deg^{+}(u_k)$ denotes its current number of children, and the sequence of \emph{fitnesses} $(a_n)_{n\geq 1}$ is deterministic and chosen as a parameter of the model. 

We show that for any sequence $(a_n)_{n\geq 1}$, the corresponding preferential attachment tree has the same distribution as some weighted recursive tree with a \emph{random} sequence of weights (with some explicit distribution).
We then prove almost sure scaling limit convergences for some statistics associated with weighted recursive trees as time goes to infinity, such as degree sequence, height, profile and also the weak convergence of some measures carried on the tree.
Thanks to the connection between the two models, these results also apply to affine preferential attachment trees. 
}

\section{Introduction}

The uniform recursive tree has been introduced in the 70's as an example of random graphs constructed by addition of vertices: starting from a tree with a single vertex, the vertices arrive one by one and the $n$-th vertex picks its parent uniformly at random from the $n - 1$ already present vertices. Many properties of this tree were then investigated due to its particularly simple dynamics: number of leaves, profile, height, degrees, size of subtrees and others. 
We refer to the survey \cite{smythe_survey_1995} and the more recent book \cite[Section~6]{drmota_random_2009} for an overview of the results obtained for this model.

We consider a generalisation of the uniform recursive tree called the weighted recursive tree (WRT), which was introduced in \cite{borovkov_asymptotic_2006} in 2006. 
In this model, each vertex is assigned a non-negative weight, constant in time. 
When a newcomer randomly picks its parent, it does so with probability proportional to those weights. 
Although more general than the uniform recursive tree, WRT's have attracted far fewer contributions, see e.g.\ \cite{mailler_random_2018_online,hiesmayr_asymptotic_2017_online}. 
In \cite{mailler_random_2018_online} those trees are studied because of their connection to a model of random walk with preferential relocation (a.k.a. "monkey walk"). 
The authors prove some limiting results for the distribution of the weight of vertices at different heights in the tree, for different assumptions on the weight sequence which cover a wide range of behaviours.

In this paper, we prove asymptotic results for this model about the degree sequence, the height, the profile and the convergence of some probability measures carried on the tree, mainly under some assumptions that ensure that the sequence $(w_n)_{n\geq1}$ describing the weights of the vertices in order of creation behaves roughly as a power of $n$. 
Our deepest result is the one that concerns the asymptotic behaviour of the profile of the tree, which is the function that maps each integer $k$ to the number of vertices in the tree at height $k$. 
Both the statement and the proof of this result are inspired from the work carried out in the last 20 years for different models of logarithmic trees, see \cite{chauvin_profile_2001,chauvin_martingales_2005,sulzbach_functional_2008,schopp_functional_2010,kabluchko_general_2017}. 
They rely on the study of the Laplace transform of the profile using tools that ultimately date back to Biggins \cite{biggins_uniform_1992} in the context of the branching random walk, together with a Fourier inversion argument, which in our case is handled by a very precise theorem of \cite{kabluchko_general_2017}.
The rest of our results and proofs on WRT's are less involved and mostly rely on more elementary arguments, as well as a connection with Pemantle's time-dependent P\'olya urns, introduced in \cite{pemantle_time_1990}. 

We will also consider another model of trees which we call the affine preferential attachment tree (PAT) with fitnesses. 
In this one, every vertex has a fixed fitness, and the probability of picking any vertex to be the parent of a newcomer is proportional to its fitness plus its current number of children. 

The term "preferential attachment", coined by Barab\'asi and Albert in \cite{barabasi_emergence_1999}, refers to the property that a vertex in the graph that has a high degree tends to increase its degree even more over time, also referred to as a "rich-get-richer" effect.
Many different preferential attachment mechanisms have then been studied in the last two decades because the degree distribution that emerges from this type of construction shares some quantitative properties with real-world networks, see \cite{hofstad_random_2017, newman_structure_2003} for good overviews of the vast literature on this subject.

In our case, one of our motivations for studying those trees arises from the analysis of some growing random graphs, developed in the companion paper \cite{senizergues_growing_2020_online}. 
The class of models that we study there is designed to encompass Rémy's algorithm, described in \cite{remy_un_1985}, which creates a sequence of binary trees, and a lot of its natural generalizations, studied in \cite{ford_probabilities_2005_online,marchal_note_2008,chen_new_2009,haas_scaling_2015,haas_scaling_2019_online,ross_scaling_2018}.
In particular, we show that the sequences graphs obtained using these constructions, considered as metric spaces, almost surely converge in the so-called Gromov--Hausdorff--Prokhorov scaling limit towards a limiting random continuous metric space. 
This proof relies on a decomposition of our graphs along the structure of a tree, whose evolution is that of an affine preferential attachment tree with fitnesses. Notably, a crucial result that is needed in this argument is a uniform control over the degree of all the vertices the tree, which we prove in this paper.

Let us note that only a few contributions in the literature concern this particular model, where the fitness can depend on the vertex.
In the case where the fitnesses are i.i.d., the model is considered for the first time in \cite{erguen_growing_2002} and the first rigorous mathematical result can be found in \cite{bhamidi_universal_2007_online}. 
Very recently, still in the case of i.i.d. fitnesses, it has been studied in more detail in \cite{lodewijks_phase_2020_online} along with some other similar models. 
The authors study the asymptotic degree distribution and maximum degree in the tree and show that these can exhibit different behaviours according to the tail of the fitness distribution, which the authors classify as weak, strong and extreme disorder.
Let us also mention two models that do not fall in our setting but are somewhat related, studied in \cite{deijfen_preferential_2009} and \cite{bloem-reddy_preferential_2017_online}, in which the reinforcement is affine in the degree of the vertices but there is some inhomogeneity between vertices. 
Instead of coming from different fitnesses associated to vertices like in our model, this inhomogeneity is introduced using a random initial degree, or respectively a random time of creation.
 
Our approach for studying this model relies on the connection between the PAT and the WRT models (this was already known in the field in the case of constant fitnesses but stated in a slightly different form, see \cite{bloem-reddy_preferential_2017_online,berger_asymptotic_2014}).
Indeed we shall see that using a de Finetti-type argument, a PAT can be seen as a WRT with a random sequence of weights that almost surely decays like a power of $n$.   
This enables us to translate all of the results obtained for WRT's to corresponding results for PAT's, and hence prove asymptotics for degrees, height and profile of the tree. 
In particular, we prove the almost sure scaling limit convergence of the sequence of degrees of the vertices in the tree in an $\ell^p$ norm. 
For some regular sequences of fitnesses, we can explicitly describe the distribution of the limiting sequence using Beta, Gamma and Mittag-Leffler distributions. 
This relates in various ways to other results that can be found in the literature associated to preferential attachment trees or to urn models, contained in \cite{mori_maximum_2005,janson_limit_2006,pekoez_degree_2013,james_generalized_2015_online, pekoez_generalized_2016, pekoez_joint_2017, pekoez_polya_2019, banderier_periodic_2019_online}.

\subsection{Two related models of growing trees}\label{wrt:subsec:two models of growing trees}
\paragraph{Definitions.}For any sequence of non-negative real numbers $(w_n)_{n\geq 1}$ with $w_1>0$, we define the distribution $\wrt((w_n)_{n\geq 1})$ on sequences of growing rooted labelled trees\footnote{In fact, in the rest of the paper we will see them as plane trees, see Section~\ref{wrt:subsec:measures}.}, which is called the \emph{weighted recursive tree with weights $(w_n)_{n\geq 1}$}. 
We construct a sequence of rooted trees $(\ttT_n)_{n\geq 1}$ starting from $\ttT_1$ containing only one root-vertex $u_1$ with label $1$ and let it evolve in the following manner: the tree $\ttT_{n+1}$ is obtained from $\ttT_n$ by adding a vertex $u_{n+1}$ with label $n+1$. The parent of this new vertex is chosen to be the vertex with label $K_{n+1}$ with probability proportional to its weight, that is 
\begin{align}\label{wrt:eq:def Kn}
\forall k\in \{1,\dots,n\}, \qquad \Ppsq{K_{n+1}=k}{\ttT_1,\ttT_2,\dots \ttT_n}\propto w_k.
\end{align}
Remark that this conditional distribution does not depend on the evolution $\ttT_1,\ttT_2,\dots \ttT_n$ up to time $n$, which ensures in particular that the random variables $K_2,K_3,\dots$ are independent.  
In this definition, we also allow sequences of weights $(\mathsf{w}_n)_{n\geq 1}$ that are random and in this case the distribution $\wrt((\mathsf{w}_n)_{n\geq 1})$ denotes the law of the random tree obtained by the above process conditionally on $(\mathsf{w}_n)_{n\geq 1}$, so that the obtained distribution on growing trees is a mixture of WRT with deterministic sequences of weights.

Similarly, for any sequence $(a_n)_{n\geq 1}$ of real numbers, with $a_1>-1$ and $a_n\geq 0$ for $n\geq 2$, we define another model of growing tree.
The construction goes on as before: $\ttP_1$ contains only one root-vertex $u_1$ with label $1$ and $\ttP_{n+1}$ is obtained from $\ttP_n$ by adding a vertex $u_{n+1}$ with label $n+1$ and the parent of the newcomer is chosen to be the vertex with label $J_{n+1}$, where now
\begin{align}\label{wrt:eq:def Jn}
\forall k\in \{1,\dots,n\}, \qquad \Ppsq{J_{n+1}=k}{\ttP_1,\ttP_2,\dots,\ttP_n}\propto \deg^{+}_{\ttP_n}(u_k)+a_k,
\end{align}
where $\deg^{+}_{\ttP_n}(\cdot)$ denotes the number of children in the tree $\ttP_n$. In the particular case where $n=1$, the second vertex $u_2$ is always defined as a child of $u_1$, even in the case $-1<a_1\leq0$ for which the last display does not make sense. We call this sequence of tree an \emph{affine preferential attachment tree with fitnesses $(a_n)_{n\geq 1}$} and its law is denoted by $\pa((a_n)_{n\geq 1})$.

\paragraph{Notation.}Here and in the rest of the paper, whenever we have any sequence of real numbers $(x_n)_{n\geq 1}$, we write $\boldsymbol{x}=(x_n)_{n\geq 1}$ in a bold font as a shorthand for the sequence itself, and $(X_n)_{n\geq 1}$ with a capital letter to denote the sequence of partial sums defined for all $n\geq 1$ as $X_n:=\sum_{i=1}^nx_i$. In particular, we do so for sequences of fitnesses $(a_n)_{n\geq 1}$, for deterministic sequences of weights $(w_n)_{n\geq 1}$ and for random sequence of weights $(\mathsf w_n)_{n\geq 1}$.

\paragraph{Representation result.}The following result gives a connection between these two models of growing trees. It is an analogue of the so-called "P\'olya urn-representation" result described in \cite[Theorem~2.1]{berger_asymptotic_2014} or \cite[Section 1.2]{bloem-reddy_preferential_2017_online} for related models, which already cover the case of constant sequences $\mathbf{a}$. 

For $a,b>0$ the distribution $\mathrm{Beta}(a,b)$ has density $\frac{\Gamma(a+b)}{\Gamma(a)\Gamma(b)}\cdot x^{a-1}(1-x)^{b-1} \cdot \ind{0\leq x\leq 1}$ with respect to Lebesgue measure. 
If $b=0$ and $a>0$, we use the convention that the distribution $\mathrm{Beta}(a,b)$ is a Dirac mass at $1$.
\begin{theorem}[WRT-representation of PAT's]\label{wrt:thm:connection PA WRT}
For any sequence $\mathbf a$ of fitnesses, we define the associated random sequence $\boldsymbol{\mathsf w}^\mathbf a=(\mathsf{w}^\mathbf{a}_n)_{n\geq 1}$ as 
\begin{equation}
\mathsf{w}^\mathbf{a}_1=\mathsf{W}^\mathbf{a}_1=1 \qquad \text{and} \qquad \forall n\geq 2, \quad \mathsf{W}^\mathbf{a}_n=\prod_{k=1}^{n-1}\beta_k^{-1},
\end{equation}
where the $(\beta_k)_{k\geq 1}$ are independent with respective distribution $\mathrm{Beta}(A_k+k,a_{k+1})$. 
Then,
the distributions $\pa(\mathbf a)$ and $\wrt(\boldsymbol{\mathsf w}^\mathbf a)$ coincide.
\end{theorem}
The result of the theorem is obtained by studying the evolution of the degrees in the preferential attachment model $(\ttP_{n})_{n\geq 1}$. The key argument lies in the fact that we can describe the whole process $(\ttP_{n})_{n\geq 1}$ using a sequence of P\'olya urns, related to the degrees of the vertices. 
The connection of the evolution of the degrees to P\'olya urns in the context of preferential attachment models is well-know and was observed for the first time in \cite{mori_maximum_2005}. It  explains why Beta-distributed random variables appear in the limit. 
In our case, the theorem relies on applying the de Finetti theorem to this sequence of urns and on proving that those urns are jointly independent.

In fact, the result stated in the theorem can be made a bit more precise than an equality in distribution as soon as the sequence $\mathbf{a}$ is chosen in such way that almost surely the degree of the first vertex $\deg^{+}_{\ttP_n}(u_1)$ tends to infinity as $n\rightarrow\infty$. For example, it is easy to check that the condition $A_n=\grandO{n}$ is sufficient to ensure this behaviour, and in this case we can state the following corollary.
\begin{corollary}\label{wrt:cor: pa conditional distribution limiting degrees}
For a sequence $\mathbf{a}$ such that $A_n=\grandO{n}$, we can construct the sequence $\boldsymbol{\mathsf w}^\mathbf a$ from $(\ttP_n)_{n\geq 1}$ in such a way that for all $k\geq 1$:
\begin{align}\label{wrt:eq:wna as limit of ration of degrees}
	\mathsf w_k^\mathbf{a}=\lim_{n\rightarrow\infty} \frac{\deg^{+}_{\ttP_n}(u_k)}{\deg^{+}_{\ttP_n}(u_1)} \quad \text{almost surely.}
\end{align}
The obtained sequence has the distribution described in Theorem~\ref{wrt:thm:connection PA WRT} and conditionally on this sequence $(\ttP_n)_{n\geq 1}$ has distribution $\wrt(\boldsymbol{\mathsf w}^\mathbf a)$. 
\end{corollary}

In fact, and this is the content of Proposition~\ref{wrt:prop:assum pa implies assum wrt} below, if $A_n$ grows linearly as some $c\cdot n$ with some $c>0$ then the sequence $(\mathsf W^\mathbf{a}_n)_{n\geq 1}$ almost surely grows as some power of $n$ which depends on $c$.
This is done using moment computations under the explicit definition of $(\mathsf W^\mathbf{a}_n)_{n\geq 1}$ given by the theorem.  
In the rest of the paper, we investigate several properties of the WRT under this type of assumptions for the sequence of weights, such as convergence of height, profile and measures carried on the tree. Thanks to this connection, our results will then also hold for the PAT under the assumption that $A_n$ grows linearly.

\paragraph{Assumptions on the sequences.}
For two sequences $(x_n)$ and $(y_n)$ we say that 
\begin{align}
x_n\underset{n\rightarrow\infty}{\bowtie}y_n  \qquad \text{if and only if} \qquad \exists  \epsilon>0,\ x_n\underset{n\rightarrow\infty}=y_n\cdot (1+\grandO{n^{-\epsilon}}).
\end{align}
Our main assumption for sequences $\mathbf a=(a_n)_{n\geq 1}$ of fitnesses is the following \eqref{wrt:assum: pa An=cn+rn}, which is parametrised by some positive $c>0$ and ensures that the fitness of vertices is $c$ on average
\begin{align}\tag{$H_c$}
\label{wrt:assum: pa An=cn+rn}
A_n\underset{n\rightarrow\infty}{\bowtie} c\cdot n.
\end{align}
For sequences of weights $\boldsymbol{w}=(w_n)_{n\geq 1}$, we introduce the following hypothesis, which depends on a parameter $\gamma>0$
\begin{align}\tag{$\square_\gamma$}\label{wrt:assum: wrt alpha}
W_n\underset{n\rightarrow\infty}{\bowtie} \cst \cdot n^{\gamma},
\end{align}
where $\cst$ denotes a positive constant.
The following proposition ensures in particular that our assumption on sequences of fitnesses $\mathbf{a}$ translates to a power-law behaviour for the random sequence of cumulated weights $(\mathsf W^\mathbf{a}_n)_{n\geq 1}$ defined in Theorem~\ref{wrt:thm:connection PA WRT}.
\begin{proposition}\label{wrt:prop:assum pa implies assum wrt}
Suppose that there exists $c>0$ such that $\mathbf a$ satisfies \eqref{wrt:assum: pa An=cn+rn}, then the random sequence $(\mathsf{w}^\mathbf{a}_n)_{n\geq 1}$ defined in Theorem~\ref{wrt:thm:connection PA WRT} almost surely satisfies \eqref{wrt:assum: wrt alpha} with \[\gamma=\frac{c}{c+1}.\] 

If furthermore $\mathbf a$ is such that $a_n\leq (n+1)^{c'+\petito{1}}$ for some $c'\in\intervallefo{0}{1}$, then almost surely $\mathsf{w}^\mathbf{a}_n \leq (n+1)^{c'-\frac{1}{c+1}+o_\omega(1)}$, where $o_\omega(1)$ is a random function of $n$ which tends to $0$ when $n\rightarrow\infty$.
\end{proposition}

\paragraph{Convergence of degrees using the WRT representation.}
In the WRT with a deterministic sequence of weights $\boldsymbol{w}$ that satisfies
\begin{align}\label{wrt:eq:Wn asymptotique n puissance gamma}
W_n\underset{n\rightarrow\infty}{\sim} C\cdot n^{\gamma},
\end{align} 
for some $\gamma\in\intervalleoo{0}{1}$, the degree of a fixed vertex evolves as a sum of independent Bernoulli random variables and it is possible to handle it with elementary methods and obtain
\begin{equation}\label{wrt:eq:convergence degré fixé wrt}
\deg^{+}_{\ttT_n}(u_k)\underset{n\rightarrow\infty}{\sim}\frac{w_k}{C(1-\gamma)}\cdot n^{1-\gamma}.
\end{equation}
 Further calculations allow us to improve this statement to an almost sure convergence
\begin{equation}\label{wrt:eq:convergence degrees wrt}
n^{-(1-\gamma)}\cdot (\deg^+_{\ttT_n}(u_1),\deg^+_{\ttT_n}(u_2),\dots)\rightarrow \frac{1}{C(1-\gamma)}\cdot(w_1,w_2,\dots)
\end{equation} in the space $\ell^p$ of $p$-th power summable sequences, for weight sequences $\boldsymbol{w}$ that satisfy some additional control. A precise version of this statement is given in Proposition~\ref{wrt:prop:control wkn}.

Suppose that $\mathbf{a}$ satisfies \eqref{wrt:assum: pa An=cn+rn} and consider $(\ttP_{n})_{n\geq1}$ which has distribution $\pa(\mathbf{a})$.  
Then, according to Theorem~\ref{wrt:thm:connection PA WRT} and Corollary~\ref{wrt:cor: pa conditional distribution limiting degrees}, we know that conditionally on the sequence $(\mathsf w^\mathbf{a}_n)_{n\geq 1}$ obtained as in \eqref{wrt:eq:wna as limit of ration of degrees}, the sequence $(\ttP_{n})_{n\geq1}$ has distribution $\wrt((\mathsf{w}^\mathbf{a}_n)_{n\geq 1})$. 
Also, thanks to Proposition~\ref{wrt:prop:assum pa implies assum wrt}, we know that there exists some random variable $Z$ such that $\mathsf W^\mathbf{a}_n\underset{}{\sim} Z\cdot  n^{\gamma}$ almost surely as $n\rightarrow\infty$, with $\gamma=\frac{c}{c+1}$.
So let us introduce
 \begin{equation}\label{wrt:eq:def mna}
 (\mathsf m_n^\mathbf{a})_{n\geq 1}:= \frac{1}{Z(1-\gamma)}\cdot (\mathsf w^\mathbf{a}_n)_{n\geq 1}=\frac{c+1}{Z} \cdot (\mathsf w^\mathbf{a}_n)_{n\geq 1} \quad \text{a.s..}
 \end{equation} 

Applying the convergence \eqref{wrt:eq:convergence degrees wrt} conditionally on the sequence $(\mathsf w_n^\mathbf{a})_{n\geq 1}$ (or equivalently conditionally on $(\mathsf m_n^\mathbf{a})_{n\geq 1}$) yields an almost sure convergence in the product topology on sequences, which can be improved to an $\ell^p$ convergence if $\mathbf{a}$ satisfies some additional control, thanks Proposition~\ref{wrt:prop:control wkn}.
This is stated below as a theorem.
\begin{theorem}\label{wrt:thm:convergence degrees pa}
	Suppose that $\mathbf{a}$ satisfies \eqref{wrt:assum: pa An=cn+rn}. Then for a sequence $(\ttP_n)_{n\geq 1}\sim \pa(\mathbf{a})$ we obtain the following almost sure convergence in the product topology
	\begin{equation}\label{wrt:eq:convergence degrees pa}
	n^{-\frac{1}{c+1}}\cdot (\deg^{+}_{\ttP_n}(u_1),\deg^{+}_{\ttP_n}(u_2),\dots )\underset{n\rightarrow\infty}{\longrightarrow}(\mathsf m_1^\mathbf{a},\mathsf m_2^\mathbf{a},\dots ).
	\end{equation}
	Furthermore, if $a_n\leq (n+1)^{c'+\petito{1}},$ for some $0\leq c'<\frac{1}{c+1}$, the previous convergence also takes place in the space $\ell^{p}$ of $p$-th power summable sequences, for all $p>\frac{c+1}{1-(c+1)c'}$.
\end{theorem}
Let us emphasize that the function $\mathrm{max}:\ell^p\rightarrow\R$ that outputs the maximum of a  sequence is a continuous function, so that the scaling limit of the maximal degree in the tree $\ttP_n$ is ensured by the theorem whenever the appropriate condition on the sequence $\mathbf{a}$ is satisfied.
Convergence of the rescaled degree of fixed vertices in preferential attachment trees is a well-know phenomenon in the case a preferential attachment trees with constant fitnesses, as is the convergence of the maximum of that sequence, see \cite{mori_maximum_2005}.
However, to the best of the author's knowledge, Theorem~\ref{wrt:thm:convergence degrees pa} is the first result that ensures an almost sure convergence of the rescaled degrees as a sequence in such a topology. 
This improves the $\ell^p$ convergence proved in distribution in \cite{pekoez_joint_2017} for a related model, which we treat in Proposition~\ref{wrt:cor:convergence degree sequence m,delta preferential attachment}.

The distribution of the limiting sequence $(\mathsf m_n^\mathbf{a})_{n\geq 1}$ can be characterized, and even has a reasonable description for certain regular sequences of fitnesses $\mathbf{a}$, as it is explained in the following paragraph. 
This result is actually related to the study of some urn models like the P\'olya urns with immigration of \cite{pekoez_generalized_2016} or the periodic P\'olya urns of \cite{banderier_periodic_2019_online} and allows us to provide some alternative proofs and complete some of the known results about those processes. 
This is developed in Section~\ref{wrt:subsec:application to polya urn with immigration}.

\paragraph{Distribution of the limiting chain.}
Let us say a word on the properties of the non-decreasing sequence $(\mathsf M^\mathbf{a}_n)_{n\geq 1}$ that corresponds using our notation to the sequence $(\mathsf m_n^\mathbf{a})_{n\geq 1}$ defined in \eqref{wrt:eq:def mna}.
Of course, using the random variables $(\beta_n)_{n\geq 1}$ defined in Theorem~\ref{wrt:thm:connection PA WRT}, we can write for any $n\geq1$,
\begin{align}\label{wrt:eq:Mna as iterated product}
\mathsf M^\mathbf{a}_n=\frac{c+1}{Z}\cdot \prod_{k=1}^{n-1}\beta_{k}^{-1},
\end{align}
but then because the random variable $Z$ depends on the whole sequence $(\beta_n)_{n\geq 1}$, the sequence $(\mathsf M^\mathbf{a}_n)_{n\geq 1}$ is not just an iterated product of independent random variables, as it was the case for $(\mathsf{W}_n^\mathbf{a})_{n\geq 1}$.
Nevertheless, the sequence still has the nice property of being a time-inhomogeneous Markov chain with a simple \emph{backward} transition, characterised by the equality
\begin{equation*}
	\mathsf M^\mathbf{a}_n=\beta_n\cdot \mathsf M^\mathbf{a}_{n+1},
\end{equation*}
where $\beta_n$ is independent of $\mathsf M^\mathbf{a}_{n+1}$ and has distribution $\mathrm{Beta}(A_n+n,a_{n+1})$. This is the content of Proposition~\ref{wrt:prop:Mk markov chain}.

For some specific choices of sequences $\mathbf{a}$, the distribution of the chain $(\mathsf M_n^\mathbf{a})_{n\geq1}$ is explicit. Whenever $\mathbf{a}$ is of the form \[\mathbf{a}=a,b,b,b,\dots \qquad \text{with $a>-1$ and $b>0$,}\] 
we retrieve Goldschmidt and Haas' Mittag-Leffler Markov chain family, introduced in \cite{goldschmidt_line_2015} and also studied by James \cite{james_generalized_2015_online}.

The other case where the chain is explicit is when $\mathbf{a}$ is periodic starting from the second term, of the form
\[\mathbf{a}=a,\underbrace{b_1,b_2,\dots,b_l},\underbrace{b_1,b_2,\dots,b_\ell},b_1,b_2\dots \qquad\text{with $a>-1$ and $b_1,b_2,\dots,b_\ell\in\N$.}\]
Then the sequence $(\mathsf M_n^\mathbf{a})_{n\geq1}$ has an explicit distribution defined using products of Gamma-distributed random variables. We define it in Section~\ref{wrt:subsec:product generalised gamma}.

\subsection{Other geometric properties of weighted random trees}
Let us now state the convergence for other statistics of weighted random trees, namely profile, height and probability measures. Here we let $(\ttT_n)_{n\geq 1}$ be a sequence of trees evolving according to the distribution $\wrt(\boldsymbol{w})$ for some deterministic sequence $\boldsymbol{w}$ and state our results in this setting. 
Our results will also apply to random sequences of weights $\boldsymbol{\mathsf w}$ that satisfy the assumptions of the theorems almost surely, they will hence apply to PAT with appropriate sequences of fitnesses, thanks to Theorem~\ref{wrt:thm:connection PA WRT} and Proposition~\ref{wrt:prop:assum pa implies assum wrt}. 
\subsubsection{Height and profile of WRT}\label{wrt:subsec:height and profile}
Let
\begin{align*}
\mathbb L_n(k):=\#\enstq{1\leq i\leq n}{\haut(u_i)=k}
\end{align*}
be the number of vertices of $\ttT_n$ at height $k$. The function $k\mapsto \bL_n(k)$ is called the \emph{profile} of the tree $\ttT_n$.
The height of the tree is the maximal distance of a vertex to the root, which we can also express as $\haut(\ttT_n):=\max\enstq{k\geq 0}{\bL_n(k)>0}$. We are interested in the asymptotic behaviour of $\bL_n$ and $\haut(\ttT_n)$ as $n\rightarrow\infty$.

In order to express our results, we need to introduce some quantities. For $\gamma>0$, we define the function $f_\gamma:\R\rightarrow\R$ as
\begin{align*}%
f_\gamma:z\mapsto f_\gamma(z):=1+\gamma\left(e^z-1-ze^z\right).
\end{align*}
 This function is increasing on $\intervalleof{-\infty}{0}$ and decreasing on $\intervallefo{0}{\infty}$  with $f_\gamma(-\infty)=1-\gamma $ and $f_\gamma(0)=1$ and $f_\gamma(\infty)=-\infty$.
We define $z_{+}$ and $z_-$ as
\begin{equation}\label{wrt:eq:definition z+ and z-}
	z_+:=\sup\enstq{z\in \R}{f_\gamma(z)>0} \quad \text{and} \quad
z_{-}:=\begin{cases}
-\infty &\text{if }\gamma\leq 1,\\
\log((\gamma-1)/\gamma)&\text{if }\gamma> 1.
\end{cases}
\end{equation}

We are going to assume that we work with a sequence $\boldsymbol{w}$ which satisfies the following assumption \eqref{wrt:assum: wrt alpha p} for some $\gamma>0$ and $p\in\intervalleof{1}{2}$, 
\begin{align}\tag{$\square_\gamma^p$}\label{wrt:assum: wrt alpha p}
W_n\underset{n\rightarrow\infty}{\bowtie}\cst\cdot n^\gamma \qquad \text{and} \qquad \sum_{i=n}^{2n}w_i^p\leq n^{1+(\gamma-1) p + \petito{1}}.
\end{align}
Thanks to Proposition~\ref{wrt:prop:assum pa implies assum wrt}, this property is almost surely satisfied for $\gamma=\frac{c}{c+1}$ by the random sequence $\boldsymbol{\mathsf w}^\mathbf{a}$ for any sequence $\mathbf{a}$ of fitnesses satisfying $A_n\underset{n\rightarrow\infty}{\bowtie} c\cdot n$ and $a_n\leq (n+1)^{\petito{1}}$. 
\begin{theorem}\label{wrt:thm:profile and height of wrt}	
	Suppose that there exists $\gamma>0$  and $p\in\intervalleof{1}{2}$ such that the sequence $\boldsymbol{w}$ satisfies \eqref{wrt:assum: wrt alpha p}. Then, for a sequence of random trees $(\ttT_n)_{n\geq 1}\sim\wrt(\boldsymbol{w})$, we have the almost sure asymptotics for the profile 
\begin{align}\label{wrt:eq:convergence uniforme profil}
\bL_n(k)\underset{n\rightarrow\infty}{=}\frac{n}{\sqrt{2\pi \log n}}\exp\left\lbrace-\frac{1}{2}\cdot \left(\frac{k-\gamma\log n}{\sqrt{\gamma\log n}}\right)^2\right\rbrace+\grandO{\frac{n}{\log n}},
\end{align}
where the error term is uniform in $k\geq 0$.
Also for any compact $K\subset \intervalleoo{z_-}{z_+}$ we have almost surely for all $z\in K$
\begin{align}\label{wrt:eq:radish asymptotics}
\bL_n\left(\lfloor \gamma e^z \log n\rfloor\right)=n^{f_\gamma(z)-\frac{1}{2}\frac{\log\log n}{\log n}+\grandO{\frac{1}{\log n}}},
\end{align}
where the error term is uniform in $z\in K$. 
Moreover, we have the almost sure convergence
\begin{align}\label{wrt:eq:convergence height gamma p}
\frac{\haut(\ttT_n)}{\log n}\underset{n\rightarrow\infty}{\longrightarrow}\gamma\cdot e^{z_{+}}.
\end{align}
\end{theorem}
The proof of this result follows the path used for many similar results for trees with logarithmic growth (see \cite{chauvin_profile_2001,chauvin_martingales_2005,katona_width_2005}): we study the Laplace transform of the profile $z\mapsto \sum_{k=0}^ne^{zk}\bL_n(k)$ on a domain of the complex plane and prove its convergence to some random analytic function when appropriately rescaled. Then, we apply \cite[Theorem 2.1]{kabluchko_general_2017}, which consists of a fine Fourier inversion argument and hence allows to obtain precise asymptotics for $\bL_n$. The application of the theorem in its full generality proves a so-called Edgeworth expansion for $\bL_n$, which we express here in a weaker form by equations \eqref{wrt:eq:convergence uniforme profil} and \eqref{wrt:eq:radish asymptotics}. The convergence \eqref{wrt:eq:convergence uniforme profil} expresses that the profile is asymptotically close to a Gaussian shape centred around $\gamma \log n$ and with variance $\gamma \log n$, so that a majority of vertices have a height of order $\gamma \log n$. The second equation \eqref{wrt:eq:radish asymptotics} provides the behaviour of the number of vertices at a given height, for heights that are not necessarily close to $\gamma \log n$ (for which the preceding result ensure that there are of order $\frac{n}{\sqrt{\log n}}$ vertices per level). According to this result, at height $\lfloor \gamma e^z \log n\rfloor$ for any $z\in\intervalleoo{z_-}{z_+}$ there are of order $\frac{n^{f_\gamma(z)}}{\sqrt{\log n}}$ vertices.

Remark that the exponent $f_\gamma(z)$ is continuous in $z$ and tends to $0$ when $z\rightarrow z_+$. Although this does not directly prove the convergence \eqref{wrt:eq:convergence height gamma p}, it already provides a lower-bound for $\haut(\ttT_n)$ since it ensures that asymptotically there always exist vertices at height $\lfloor\gamma e^{(z_+-\epsilon)} \log n\rfloor$, for any small $\epsilon>0$. 
The convergence of the height \eqref{wrt:eq:convergence height gamma p} can then be obtained by proving a corresponding upper-bound, which can be done using quite rough estimates.

This result includes the well-known asymptotics $\haut(\ttT_n)\sim e \log n$ as $n\rightarrow\infty$ for the uniform random tree, proved for example in \cite{devroye_branching_1987,pittel_note_1994}. 
Using the connection of preferential attachment trees to weighted recursive trees given by Theorem~\ref{wrt:thm:connection PA WRT}, it also includes the case of preferential attachment trees with constant fitnesses, for which similar results were proved, in \cite{pittel_note_1994} for the height and in \cite{katona_width_2005,kabluchko_general_2017} for the asymptotic behaviour of the profile \eqref{wrt:eq:convergence uniforme profil}. 

\begin{remark}
One can also notice that, in the case $\gamma>1$, the function $f_\gamma$ tends to a negative value $1-\gamma$ as $z\rightarrow-\infty$ so that it is understandable that the asymptotics \eqref{wrt:eq:radish asymptotics} couldn't be valid for values of $z$ below a certain threshold.
	Nevertheless, one can check that $f_\gamma(z_-)>0$ and wonder what happens for values of $z$ that are slightly below $z_-$.
	In fact, in the proof, we first study the \emph{weighted profile} of the tree, which corresponds to the total weight of vertices at every height instead their number. In this case, the study of the corresponding Laplace transform is easier and would lead to a statement similar to Theorem~\ref{wrt:thm:profile and height of wrt} for the asymptotics of the weighted profile that would hold for any $z$ such that $f_\gamma(z)>0$.
	A subsequent part of the proof then consists in transferring this result to the Laplace transform of the "true" profile of the tree, and this is the part that breaks down if $z$ is chosen smaller than $z_-$.
\end{remark}

As a complement to this result, let us mention that there is another case where we can compute the asymptotic height of the tree, which corresponds to sequences $\boldsymbol{w}$ that grow fast to infinity.
For any sequence of weights $\boldsymbol{w}$, a quantity of interest is $\sum_{i=2}^{n}\frac{w_i}{W_i}$, which is the expected height of a vertex taken with probability proportional to its weight in $\ttT_n$. When this quantity grows faster than logarithmically, we have the almost sure convergence (see Proposition~\ref{wrt:prop:upper-bound height} in Section~\ref{wrt:subsec:upper-bounds on height})
\begin{align*}
\lim_{n\rightarrow\infty} \frac{\haut(\ttT_n)}{\sum_{i=2}^{n}\frac{w_i}{W_i}}=\lim_{n\rightarrow\infty} \frac{\haut(u_n)}{\sum_{i=2}^{n}\frac{w_i}{W_i}}=1,
\end{align*}
which indicates that all the action takes place at the very tip of the tree.
\subsubsection{Convergence of the weight measure} \label{wrt:subsec:measures}\label{wrt:subsubsec:convergence of the weight measure intro}
We also study the convergence of some natural probability measures defined on the trees $(\ttT_n)_{n\geq 1}$. This will prove useful for the applications developed in the companion paper \cite{senizergues_growing_2020_online}.     
 
\paragraph{Plane-tree framework.}
For this result it will be easier to work with plane trees. 
We introduce the Ulam-Harris tree 
$\bU=\bigcup_{n=0}^\infty \N^n$, where $\N:=\{1,2,\dots\}$, with the convention that $\N^0=\{\emptyset\}$. 
Classically, a plane tree $\tau$ is defined as a non-empty subset of $\bU$ such that
\begin{enumerate}
	\item if $v\in\tau$ and $v=ui$ for some $i\in\N$, then $u\in \tau$,  
	\item for all $u\in\tau$, there exists $\deg^{+}_\tau(u)\in\N\cup \{0\}$ such that for all $i\in \N$,  $ui\in\tau$ iff $i\leq \deg^{+}_\tau(u)$.
\end{enumerate} 

We choose to construct our sequence $(\ttT_n)_{n\geq 1}$ of weighted recursive trees as plane trees by considering that each time a vertex is added, it becomes the right-most child of its parent. In this way the vertices $(u_1,u_2\dots)$ of the trees $(\ttT_n)_{n\geq 1}$, listed in order of arrival, form a sequence of elements of $\bU$. In fact, from now on, we will always assume that we use this particular embedded construction, both for the WRT and the PAT. 
Note that with this representation as unlabelled subsets of $\bU$, the tree $\ttT_n$ itself, for any $n\geq1$, \emph{does not} contain information relative to the labelling (hence the weight) of its vertices, but this piece of information can be read from the sequence $(\ttT_1,\ttT_2,\dots,\ttT_n)$. 

We also denote $\partial \bU=\N^\N$, which we can be interpreted as the set of infinite paths from the root to infinity, and write $\overline{\bU}=\bU\cup \partial\bU$. We classically endow this set with the distance 
\begin{align}\label{wrt:eq:def dexp}
\dist^{\mathrm{exp}}(u,v)=\begin{cases}
0, \qquad &\text{if } u=v,\\
\exp(-\haut(u\wedge v)), \qquad &\text{otherwise,}
\end{cases}
\end{align}
 where $u\wedge v$ denotes the most recent common ancestor of $u$ and $v$, and the height $\haut(u)$ of a vertex $u\in\bU$ is defined as the only $n$ such that $u\in \N^n$. Note that even when  $u,v\in\partial\bU$, their most recent common ancestor $u\wedge v$ belongs to $\bU$, as long as $u\neq v$.
 Endowed with this distance, $\overline{\bU}$ is then a complete separable metric space. 
 
In the paper, except when proving results related to the weak convergence of measures, for which we use the topology generated by $\dist^{\mathrm{exp}}$, we consider $\bU$ as a graph and that we compute distances between vertices using the corresponding graph distance, which we denote $\dist$. In particular, the height $\haut(u)$ of a vertex $u$ is always its graph distance $\dist(\emptyset,u)$ to the root $\emptyset$. 

\paragraph{Convergence of measures.}
For every $n\geq 1$, we define the measure $\mu_n$ on $\bU$, which only charges the set $\{u_1,\dots,u_n\}$ of vertices of $\ttT_n$, with for any $1\leq k \leq n$,
\begin{equation}
\mu_n(\{u_k\})=\frac{w_k}{W_n}.
\end{equation}
We refer to $\mu_n$ as the \emph{natural weight measure} on $\ttT_n$.  
The following theorem classifies the possible behaviours of $(\mu_n)$ for any weight sequence.

\begin{theorem}\label{wrt:thm:weak convergence of measures}
The sequence $(\mu_n)_{n\geq 1}$ converges almost surely weakly towards a limiting probability measure $\mu$ on $\overline{\bU}$.
There are three possible behaviours for $\mu$:
\begin{enumerate}
	\item If $\sum_{i=1}^{\infty}w_i<\infty$, then $\mu$ is carried on $\bU$.
	\item\label{wrt:it:condition mesure diffuse} If $\sum_{i=1}^{\infty}w_i=\infty$ and $\sum_{i=1}^{\infty}\left(\frac{w_i}{W_i}\right)^2<\infty$, then $\mu$ is diffuse and supported on $\partial \bU$.
	\item If $\sum_{i=1}^{\infty}\left(\frac{w_i}{W_i}\right)^2=\infty$ then $\mu$ is concentrated on one point of $\partial \bU$.
\end{enumerate} 
\end{theorem}
This convergence can be extended to other natural measures on the tree, such as the uniform measure on $\ttT_n$, or some "preferential attachment measure" which charges each vertex proportionally to some affine function of its degree. This is the content of Proposition~\ref{wrt:prop:convergence other measures}.
Note that in our case of interest, when the sequence $\boldsymbol{w}$ satisfies the assumption \eqref{wrt:assum: wrt alpha} for some $\gamma>0$, we are in case \ref{wrt:it:condition mesure diffuse} of the theorem. 

In the specific case of $\wrt$ (resp. $\pa$) with a sequence of weights (resp. of fitnesses) that is constant starting from the second term, the measure $\mu$ has an explicit description: if any $u\in \bU$, writing $T(u):=\enstq{uv}{v\in\overline\bU}$ for the sub-tree descending from $u$, the sequences
\begin{align*}
	\left(\frac{\mu(T(ui))}{\mu(T(u))}\right)_{i\geq 1}, \quad \text{for} \ u\in \bU,
\end{align*} 
are independent and have an explicit $\mathrm{GEM}$ distribution (see \cite{pitman_combinatorial_2006} for a definition).
Furthermore, the corresponding sequence of trees, conditionally on $\mu$, can be described as a \emph{split tree}. 
This result, along with other other properties of these families of growing trees can be found in \cite{janson_random_2019}.

\subsection{Organisation of the paper}
The paper is organised as follows.

We first investigate some properties of weighted random trees $(\ttT_n)_{n\geq 1}$ with deterministic weight sequence $\boldsymbol{w}$. 
In Section~\ref{wrt:subsec:convergence of the degree sequence} we first prove Proposition~\ref{wrt:prop:control wkn} which states the convergence of the degree sequence using elementary methods. Then in Section~\ref{wrt:subsec:convergence of measure}, we prove the weak convergence of the weight measure $\mu_n$ to some limit $\mu$ and describe three regimes for its behaviour. We also study other natural measures related to the sequence of trees $(\ttT_n)$ and prove that they also converge towards $\mu$. For all these measures, our main tool consists in introducing martingales related to the mass of a subtree descending from a fixed vertex. This is the content of Theorem~\ref{wrt:thm:weak convergence of measures} and Proposition~\ref{wrt:prop:convergence other measures}.  
In Section~\ref{wrt:sec:height and profile of WRT}, we prove Theorem~\ref{wrt:thm:profile and height of wrt} about the convergence of the height and the profile of WRT. 
This is achieved by  first proving the uniform convergence of a rescaled version of the Laplace transform of the profile on a complex domain, which is the content of Proposition~\ref{wrt:prop:assumptions kabluchko}. 
This ensures that we can use \cite[Theorem~2.1]{kabluchko_general_2017} for the convergence of the profile. This convergence provides a lower-bound for the height of the tree; we then prove a matching upper-bound to obtain asymptotics for the height. 
We also prove Proposition~\ref{wrt:prop:upper-bound height}, which identify the asymptotic behaviour of the height of the tree in the case where the weights increase very fast.

Then we switch to studying a sequence $(\ttP_n)_{n\geq 1}$ of preferential attachment trees with sequence of fitnesses $\mathbf{a}$. 
In Section~\ref{wrt:sec:exchangeability}, we present a proof of Theorem~\ref{wrt:thm:connection PA WRT} and Corollary~\ref{wrt:cor: pa conditional distribution limiting degrees} using a coupling between the preferential attachment process with a sequence of P\'olya urn processes and this establishes that $(\ttP_n)_{n\geq 1}$ can also be described as having distribution $\wrt(\boldsymbol{\mathsf w}^\mathbf{a})$ for a random sequence $\boldsymbol{\mathsf w}^\mathbf{a}$; we then prove Proposition~\ref{wrt:prop:assum pa implies assum wrt} which relates the properties of $\boldsymbol{\mathsf w}^\mathbf{a}$ to the ones of $\mathbf{a}$. We finish the section by stating and proving Proposition~\ref{wrt:prop:Mk markov chain} in which we prove that the sequence $(\mathsf M^\mathbf{a}_n)$ defined above as some random multiple of $(\mathsf W^\mathbf a_n)$ is a Markov chain.
In Section~\ref{wrt:sec:examples and applications}, we identify in Proposition~\ref{wrt:prop:distribution limiting chains} the distribution of the chain $(\mathsf{M}^\mathbf a_n)$ for particular sequences $\mathbf{a}$ using moment identifications. 
We then present two applications of this result, one concerning a model of P\'olya urn with immigration and the other concerning another model of preferential attachment graphs, in Proposition~\ref{wrt:cor:convergence degree sequence m,delta preferential attachment}.

Some technical results can be found in Appendix~\ref{wrt:app:computations}.

\subsection*{Acknowledgements}
The author would like to thank the anonymous referees for their numerous comments and suggestions that helped improve the presentation of this paper. 
He would also like to thank Philippe Marchal whose remarks led to an improvement in the generality of one of the results.
\section{Measures and degrees in weighted random trees}\label{wrt:sec:measures and degrees in the wrt}
In this section, we work with a sequence of trees $(\ttT_n)_{n\geq 1}$ that has distribution $\wrt\left(\boldsymbol{w}\right)$ for a deterministic sequence $\boldsymbol{w}$. We start with two statistics of the tree that are quite easy to analyse, namely the sequence of degrees of the vertices of the tree and also some natural measures defined on the tree.
\subsection{Convergence of the degree sequence}\label{wrt:subsec:convergence of the degree sequence}
We start the section by proving convergence for the sequence of degrees of the vertices in their order of creation under the $\wrt$ model.
We suppose here that the sequence of weights $\boldsymbol{w}$ is such that there exists constants $C>0$ and $0<\gamma<1$ for which
\begin{equation}\label{wrt:eq:Wk equivalent kgamma}
W_k \underset{k\rightarrow\infty }{\sim} C\cdot k^{\gamma}.
\end{equation}
We write $\deg^{+}_{\ttT_n}(u_k)$ for the out-degree of the vertex $u_k$ in $\ttT_n$. For a fixed $k\geq 1$ remark that, as a sequence of random variables indexed by $n\geq 1$, we have the equality in distribution
\begin{equation}\label{wrt:eq:equality distribution degree}
\left(\deg^{+}_{\ttT_n}(u_k)\right)_{n\geq 1}\overset{\mathrm{(d)}}{=}\left(\sum_{i=k}^{n-1}\ind{U_i\leq \frac{w_k}{W_i}}\right)_{n\geq 1},
\end{equation}
with $(U_i)_{i\geq 1}$ a sequence of independent uniform variables in $\intervalleoo{0}{1}$.
With this description of the distribution of the degrees of fixed vertices, only using some law of large numbers for the convergence and Chernoff bounds for the fluctuations we obtain the following result.
\begin{proposition}\label{wrt:prop:control wkn}For a sequence of weights $\boldsymbol{w}$ satisfying \eqref{wrt:eq:Wk equivalent kgamma}, the following holds.
	\begin{enumerate}
			\item\label{wrt:it:convergence pointwise} We have the almost sure pointwise convergence
		\begin{equation}\label{wrt:eq:convergence degres pointwise}
		n^{-(1-\gamma)}\cdot (\deg^{+}_{\ttT_n}(u_1),\deg^{+}_{\ttT_n}(u_2),\dots )\underset{n\rightarrow\infty}{\longrightarrow}\frac{1}{(1-\gamma)C} \cdot (w_1,w_2,\dots ).
		\end{equation}
		\item\label{wrt:it:convergence lp} If the sequence furthermore satisfies $w_k\leq (k+1)^{\gamma-1+c'+\petito{1}}$ for some constant $0\leq c'<1-\gamma$, then there exists a function of $k$ which goes to $0$ as $k\rightarrow\infty$, also denoted $o(1)$, such that all $n$ large enough, we have for all $k\geq1$
		\begin{align}\label{wrt:eq:control deg in lp}
		\deg^{+}_{\ttT_n}(u_k)\leq n^{1-\gamma}\cdot (k+1)^{\gamma-1+c'+o(1)},
		\end{align}
and the convergence \eqref{wrt:eq:convergence degres pointwise} holds almost surely in the space $\ell^{p}$ for all $p>\frac{1}{1-\gamma -c'}$. 
\end{enumerate}
\end{proposition}
\begin{proof}
		To prove \ref{wrt:it:convergence pointwise}, first remark that for any $k\geq 1$ such that $w_k\neq 0$, thanks to \eqref{wrt:eq:Wk equivalent kgamma}, we have \[\sum_{i=k}^{n-1}\frac{w_k}{W_i}\underset{n\rightarrow\infty}{\sim} w_k\cdot \frac{n^{1-\gamma}}{C(1-\gamma)}.\] 
	Using a version of the Borel-Cantelli lemma (see Lemma~\ref{wrt:lem:sum bernoulli} in the appendix), we get that almost surely
	\begin{equation*}
	\deg^{+}_{\ttT_n}(u_k)=\sum_{i=k}^{n-1}\ind{U_i\leq \frac{w_k}{W_i}}\underset{n\rightarrow\infty}{\sim}\sum_{i=k}^{n-1}\frac{w_k}{W_i}\underset{n\rightarrow\infty}{\sim}w_k\cdot \frac{n^{1-\gamma}}{(1-\gamma)C},
	\end{equation*}
	and hence $n^{-(1-\gamma)}\cdot \deg^{+}_{\ttT_n}(u_k)\rightarrow \frac{w_k}{(1-\gamma)C}$.
	For the indices $k$ for which $w_k=0$, we of course have $\deg^{+}_{\ttT_n}(u_k)=0$ almost surely for all $n\geq 1$, and so the convergence also holds. This finishes the proof of \ref{wrt:it:convergence pointwise}.

For the second part of the statement, let us first compute
	\begin{align*}
	\Ec{\exp\left( \deg^{+}_{\ttT_n}(u_k)\right)}=\Ec{\exp\left(\sum_{i=k}^{n-1}\ind{U_i\leq \frac{w_k}{W_i}}\right)}
	&=\prod_{i=k}^{n-1}\left(1+(e-1)\frac{w_k}{W_i}\right)\\
	&\leq \exp\left((e-1)w_k\sum_{i=k}^{n-1}\frac{1}{W_i}\right),
	\end{align*}
	where we have used the inequality $1+x\leq e^x$.
	Now let $C'$ be a constant such that for all $n\geq 1$, we have $\sum_{i=1}^{n-1}\frac{1}{W_i}\leq C'\cdot n^{1-\gamma}$ (such a constant exists because of the assumption \eqref{wrt:eq:Wk equivalent kgamma}). For all $k\geq 1$, we introduce the following 
	\begin{equation*}
\xi_k:= \max\left(2C'(e-1)w_k, k^{\gamma-1}\log^2(k+a)\right),
	\end{equation*}
	where the real number $a>0$ is chosen in such a way that the function $x\mapsto x^{\gamma-1}\log(x+a)$ is decreasing on $\intervalleoo{0}{\infty}$.
	Using Markov's inequality, we get for any integers $k$ and $n$ such that $n\geq k$
	\begin{align*}
	\Pp{\deg^{+}_{\ttT_n}(u_k)\geq \xi_k\cdot n^{1-\gamma}}
	&\leq \exp\left(-\xi_k\cdot n^{1-\gamma} +(e-1)w_k\sum_{i=k}^{n-1}\frac{1}{W_i}\right)\\
	&\leq \exp\left(-\frac12\cdot \xi_k\cdot n^{1-\gamma}\right).
	\end{align*}
	Using a union bound, the fact that $\deg^{+}_{\ttT_n}(u_k)=0$ for any $k>n$, and the definition of $\xi_k$, we get that for all $n\geq 1$
	\begin{align*}
	\Pp{\exists k\geq 1, \quad \deg^{+}_{\ttT_n}(u_k)\geq  \xi_k\cdot n^{1-\gamma}} &\leq \sum_{k=1}^n\exp\left(-\frac12\cdot \xi_k \cdot n^{1-\gamma}\right)\\
	&\leq n \cdot \exp\left(-\frac{1}{2}\cdot \log^2(n+a)\right).
	\end{align*}
	The last display is summable over all $n\geq 1$ and hence using the Borel-Cantelli lemma, we almost surely have for $n$ large enough 
	\[\forall k\geq 1, \quad \deg^{+}_{\ttT_n}(u_k)\leq n^{1-\gamma}\cdot \xi_k.\]
 We can conclude by noting that under our assumptions we have $\xi_k\leq (k+1)^{\gamma-1+c'+\petito{1}}$.
	The convergence in $\ell^p$ for $p>\frac{1}{1-\gamma-c'}$ is just obtained by dominated convergence using the componentwise convergence \eqref{wrt:eq:convergence degres pointwise} and the $\ell^p$ domination \eqref{wrt:eq:control deg in lp}.
\end{proof}
\subsection{Convergence of measures}\label{wrt:subsec:convergence of measure}
The goal of this section is to prove Theorem~\ref{wrt:thm:weak convergence of measures}, which concerns the convergence of the sequence of weight measures $(\mu_n)$ seen as measures on $\overline{\bU}$. One of the key arguments is the fact that the weight of the subtree descending from a fixed vertex can be described using a generalised P\'olya urn scheme, as studied by Pemantle \cite{pemantle_time_1990}. We also prove Proposition~\ref{wrt:prop:convergence other measures}, which states the weak convergence of other measures.

\paragraph{Time-dependent P\'olya urn scheme.} Let us start by describing an urn process, following Pemantle \cite{pemantle_time_1990}. 
Let $a,b$ be two non-negative real numbers, with $a+b>0$, and $k\geq1$ be an integer and $(s_n)_{n\geq k+1}$ be a sequence of non-negative real numbers. We refer to the following process as a time-dependent P\'olya urn starting at time $k$  with $a$ red balls and $b$ black balls and weight sequence $(s_n)_{n\geq k+1}$: 
\begin{itemize}
\item At time $k$, the urn contains $a$ red balls and $b$ black balls\footnote{Those numbers of balls are not required to be integers.}.
\item Then at every time $n\geq k+1$, a ball is drawn at random and replaced in the urn, along with $s_n$ additional balls of the same colour. 
\end{itemize}
For any $n\geq k$ we call $R_n$ the proportion of red balls in the urn at time $n$. We can easily check that $(R_n)_{n\geq k}$ is a martingale in its own filtration, with values in $\intervalleff{0}{1}$. As a result, it converges as $n\rightarrow\infty$ a.s. and in $L^1$ towards some random variable $R_\infty$.

\paragraph{Characterization of the convergence of probability measures over $\overline{\bU}$.}Recall from the introduction the definition of the Ulam-Harris tree $\bU=\bigcup_{n=0}^\infty \N^n$ and its completed version $\overline{\bU}=\bU\cup \partial \bU$, which is endowed with the distance $\dist^{\mathrm{exp}}$ defined in \eqref{wrt:eq:def dexp}. 
Recall that $(\overline{\bU},\dist^{\mathrm{exp}})$ is a separable and complete metric space.

For any $u\in \bU$, we write $T(u):=\enstq{uv}{v\in\overline\bU}$ the subtree descending from $u$. 
In $\overline\bU$ there is an easy characterisation of the weak convergence of sequences of probability measures defined on the Borel-$\sigma$-field associated to $\dist^{\mathrm{exp}}$, which a direct consequence of the Portmanteau theorem (see e.g.\ \cite[Theorem~2.1]{billingsley_convergence_1999}):
\begin{lemma}\label{wrt:lem:convergence measure}
Let $(\pi_n)_{n\geq 1}$ be a sequence of Borel probability measures on $\overline{\bU}$. Then $(\pi_n)_{n\geq 1}$ converges weakly to a probability measure $\pi$ if and only if for any $u\in\bU$, 
\[\pi_n(\{u\})\rightarrow \pi(\{u\}) \quad \text{and}\quad \pi_n(T(u))\rightarrow \pi(T(u)) \quad \text{as }  n\rightarrow\infty.\]
\end{lemma}
Let us provide a proof of this lemma for completeness.
\begin{proof}
We can check that the sets of the form $\{u\}$ for $u\in \bU$, or $T(u)$ for $u\in \bU$, are clopen for the topology generated by $\dist^{\mathrm{exp}}$, so by the Portmanteau theorem this already proves the "only if" part of the lemma.
Now reciprocally, suppose that the condition on $(\pi_n)_{n\geq 1}$ is satisfied. We can check that every open set $\mathcal{O}$ can be written as a countable disjoint union of these clopen sets (see for example \cite[Lemma~1.2]{rousselin_marches_2018} for a similar statement for the topology of $\partial \bU$), which we write $\mathcal{O}=\bigsqcup_{k\geq 1}\mathcal{O}_k$.
Then, using Fatou's lemma and the $\sigma$-additivity of measures we get
\begin{align*}\liminf_{n\rightarrow\infty}\pi_n\left(\mathcal{O}\right)=\liminf_{n\rightarrow\infty}\sum_{k=1}^{\infty} \pi_n(\mathcal{O}_k)\geq \sum_{k=1}^{\infty}\liminf_{n\rightarrow\infty} \pi_n(\mathcal{O}_k)
= \sum_{k=1}^{\infty}\pi(\mathcal{O}_k)=\pi(\mathcal{O})
\end{align*}
We conclude using the Portmanteau theorem again.
\end{proof}

\subsubsection{Proof of Theorem~\ref{wrt:thm:weak convergence of measures}.}
We are going to apply this criterion to our sequence $(\mu_n)_{n\geq 1}$, which, we recall, is defined in such a way that for all $n\geq 1$, the measure $\mu_n$ charges only the vertices $\{u_1,u_2,\dots,u_n\}$ of the tree $\ttT_n$, and such that for any $1\leq k \leq n$,
\begin{equation}
\mu_n(\{u_k\})=\frac{w_k}{W_n}.
\end{equation}
\begin{proof}[Proof of Theorem~\ref{wrt:thm:weak convergence of measures}]
We can already see that if $(W_n)_{n\geq 1}$ is bounded and hence converges to some $W_\infty$ we have $\mu_n(\{u_k\})\underset{}{\rightarrow}\frac{w_k}{W_\infty}$ as $n\rightarrow\infty$. 
In this case it is easy to verify that $(\mu_n)$ weakly converges to the measure $\mu$ which is such that $\mu(\{u_k\})=\frac{w_k}{W_\infty}$. 
In this case $\mu(\bU)=1$ and so $\mu$ is carried on $\bU$. 

Let us now assume that $W_n\underset{}{\rightarrow}\infty$ as $n\rightarrow\infty$ and show that in this case, $\mu_n$ converges weakly to some limit $\mu$ that is carried on $\partial \bU$.
In this case we have $\mu_n(\{u_k\})=\frac{w_k}{W_n}\underset{}{\rightarrow}0$ as $n\rightarrow\infty$. 
Let us denote for every integers $n, k\geq 1$,
\[M_n^{(k)}:=\mu_n(T(u_k)),\]
the proportion of the total mass above vertex $u_k$ at time $n$. Remark that this quantity evolves as the proportion of red balls in a time-dependent P\'olya urn scheme with weights $(w_i)_{i\geq k+1}$, starting at time $k$ with $W_{k-1}$ black balls and $w_k$ red balls.
Hence for all $k\geq 1$, the sequence $(M_{n}^{(k)})_{n\geq k}$ almost surely converges to a limit $M_{\infty}^{(k)}$.
Also, for any $u\in\bU$ that does not receive a label in the process, the sequence $(\mu_n(T(u)))_{n\geq 1}$ (and also $(\mu_n(\{u\}))_{n\geq 1}$) is identically equal to zero. 
Hence we have convergence of $(\mu_n(\{u\}))_{n\geq 1}$ and $(\mu_n(T(u)))_{n\geq 1}$ for all $u\in\bU$. 

The last step in order to prove the weak convergence of $(\mu_n)_{n\geq 1}$ is to prove that the quantities that we obtain in the limit indeed define a probability measure on $\overline{\bU}$. 
If for all $u\in\bU$ we have
\begin{equation}\label{wrt:eq:no loss of mass}
	\lim_{n\rightarrow\infty}\mu_n(T(u))= \sum_{i=1}^{\infty}\lim_{n\rightarrow\infty}\mu_n(T(ui)),
\end{equation}
then it entails that $\mu_n \underset{n\rightarrow\infty}{\rightarrow}\mu$, where $\mu$ is the unique probability measure on $\overline{\bU}$ such that for all $u\in\bU$, \begin{equation*}
\mu(\{u\})=0 \quad \text{and} \quad  \mu(T(u))=\lim_{n\rightarrow\infty}\mu_n(T(u)). 
\end{equation*}
The existence of such a measure $\mu$ is ensured by the Kolmogorov extension theorem on the product space $\partial \bU=\N^\N$. 

For any $u\notin\{u_1,u_2,\dots\}$, the equality \eqref{wrt:eq:no loss of mass} is immediate, so let us prove it for all $u_k$ for $k\geq 1$.
For any $n,k,i\geq 1$, let 
\[M_n^{(k,i)}:=\mu_n(\{u_k\})+\sum_{j=i+1}^{\infty}\mu_n(T(u_kj))=\mu_n\left(T(u_k)\right)-\sum_{j=1}^{i}\mu_n\left(T(u_kj)\right).\]
Using what we just proved, we know that for any $k,i$, the quantity $M_n^{(k,i)}$ almost surely converges as $n\rightarrow\infty$ to some limit $M_\infty^{(k,i)}$.
Proving \eqref{wrt:eq:no loss of mass} reduces to proving that for any $k\geq 1$, we almost surely have $M_\infty^{(k,i)}\underset{i\rightarrow\infty}{\rightarrow}0$. By construction, the sequence $(M_\infty^{(k,i)})_{i\geq 1}$ is non-negative and non-increasing, hence it converges almost surely, so it suffices to prove that its almost sure limit is $0$.

We define $\tau^{(k,i)}:=\inf\enstq{n\geq 1}{u_n=u_ki}$, the time when the vertex $u_k$ receives its $i$-th child in the growth procedure. 
Conditionally on the event $\{\tau^{(k,i)}=t\}$, the process $(M_n^{(k,i)})_{n\geq t}$ evolves as the proportion of red balls in a time-dependent P\'olya urn scheme, starting at time $t$ with $w_k$ red balls (that correspond to the weight of $u_k$) and $(W_{t}-w_k)$ blacks balls (that correspond to the total weight of other vertices in the tree), and weights $(w_m)_{m\geq t+1}$. Hence we have
\begin{align*}
	\Ecsq{M_\infty^{(k,i)}}{\tau^{(k,i)}=t}=\Ecsq{M_t^{(k,i)}}{\tau^{(k,i)}=t}=\frac{w_k}{W_t}.
\end{align*}

On the event $\{\tau^{(k,i)}=\infty\}$, we have $M_n^{(k,i)}=0$ for $n<k$ and $M_n^{(k,i)}=\mu_n(\{u_k\})$ for $n\geq k$, which decreases almost surely to $0$, so $M_\infty^{(k,i)}=0$ a.s. on that event.

Using the crude bound $\tau^{(k,i)}\geq i$, which entails that $W_{\tau^{(k,i)}}\geq W_{i}$ almost surely, we get 
\begin{align*}
\Ec{M_\infty^{(k,i)}}=\Ec{M_\infty^{(k,i)}\ind{\tau^{(k,i)}<\infty}}=\Ec{M_{\tau^{(k,i)}}^{(k,i)}\ind{\tau^{(k,i)}<\infty}}
&= \Ec{\frac{w_k}{W_{\tau^{(k,i)}}}\ind{\tau^{(k,i)}<\infty}}\\
&\leq \frac{w_k}{W_i}\underset{i\rightarrow\infty}{\rightarrow} 0,
\end{align*}
 hence $M_\infty^{(k,i)}\underset{i\rightarrow\infty}{\rightarrow} 0$ in $L^1$, so its almost sure limit is also $0$. 
In the end, by Lemma~\ref{wrt:lem:convergence measure}, the sequence of measures $(\mu_n)$ almost surely converges weakly to a limit $\mu$, and this measure only charges the set $\partial \bU$.

We just finished proving that, for any sequence of weight $\boldsymbol{w}$, the sequence $(\mu_n)_{n\geq 1}$ almost surely converges weakly to a probability measure $\mu$. 
Furthermore, we proved that $\mu$ is carried on $\bU$ when $(W_n)$ is bounded and carried on $\partial\bU$ when $W_n\rightarrow \infty$. 
The proof is then finished by applying the lemma stated below. 
\end{proof}
\begin{lemma}
Suppose that $\sum_{n=1}^{\infty}w_n=\infty$ so that $\mu$ is carried on $\partial\bU$. Then either $\sum_{n=1}^\infty \left(\frac{w_n}{W_n}\right)^2<\infty$ and then $\mu$ is almost surely diffuse, or $\sum_{n=1}^\infty \left(\frac{w_n}{W_n}\right)^2=\infty$ and then $\mu$ is carried on one point of $\partial\bU$.
\end{lemma}
\begin{proof}
For any $k\geq 1$ the process $(\mu_n(T(u_k))_{n\geq k}$ evolves as the proportion of red balls in a time-dependent P\'olya urn started at time $k$ with $w_k$ red balls, $W_{k-1}$ black balls and a weight sequence $(w_n)_{n\geq k+1}$. 
By the work of Pemantle in \cite{pemantle_survey_2007}, if we assume $\sum_{n=1}^\infty \left(\frac{w_n}{W_n}\right)^2=\infty$ then the limiting proportion $\mu(T(u_k))$ almost surely belongs to the set $\{0,1\}$. This translates into the fact that $\mu(T(u))\in\{0,1\}$ almost surely for any $u\in\bU$, which entails that $\mu$ is almost surely carried on one leaf of $\partial\bU$.

On the contrary, let us suppose that $\sum_{n=1}^\infty \left(\frac{w_n}{W_n}\right)^2<\infty$ and prove that this entails that the limiting measure $\mu$ is diffuse almost surely. Consider the function $(\cdot \wedge \cdot): \overline{\bU}\times \overline{\bU} \rightarrow\overline{\bU}$ which associates to each couple $(u,v)$ their most recent common ancestor $u\wedge v$ in the completed tree $\overline{\bU}$. This function is continuous with respect to the distance $\dist^{\mathrm{exp}}$. Then, since $\mu_n\rightarrow\mu$ almost surely weakly, we also get the following almost sure weak convergence of the push-forward of the product measure $\mu_n\otimes \mu_n$ on $\bU\times\bU$ by the function $(\cdot\wedge \cdot)$:
\begin{align}\label{wrt:eq:convergence pushforward measures}
(\cdot \wedge \cdot)_*(\mu_n\otimes\mu_n)\rightarrow (\cdot \wedge \cdot)_*(\mu\otimes\mu).
\end{align}
Let us fix $n\geq 1$ and conditionally on $(\ttT_1,\ttT_2,\dots, \ttT_n)$, let $D_n$ and $D_n'$ be two independent vertices taken under $\mu_n$.
Then, an argument taken from the proof of \cite[Lemma~3.8]{curien_random_2017} in a slightly different setting ensures that 
\begin{align*}\Pp{D_n\wedge D_n'=u_k}=\left(\frac{w_k}{W_k}\right)^2 \cdot \prod_{i=k+1}^{n}\left(1-\left(\frac{w_i}{W_i}\right)^2\right)\underset{k\rightarrow\infty}{\longrightarrow} p_k:=\left(\frac{w_k}{W_k}\right)^2 \cdot \prod_{i=k+1}^{\infty}\left(1-\left(\frac{w_i}{W_i}\right)^2\right).
\end{align*}
The argument goes as follows: 
\begin{itemize}
\item with probability $\left(\frac{w_n}{W_n}\right)^2$, we have $D_n=D_n'=D_n\wedge D_n'=u_n$,
\item with probability $1-\left(\frac{w_n}{W_n}\right)^2$, it is not the case, and we can check that conditionally on this event, the vertices $\left[D_n\right]_{n-1}$ and $\left[D_n'\right]_{n-1}$ defined as the most recent ancestor in $\ttT_{n-1}$ of respectively $D_n$ and $D_n'$, are independent and taken under the measure $\mu_{n-1}$, and that $D_n\wedge D_n'=\left[D_n\right]_{n-1}\wedge\left[D_n'\right]_{n-1}$.
\end{itemize}
It suffices to then apply this in cascade to get the last display.

Note that thanks to the summability condition, the infinite product $\prod_{i=2}^{\infty}\left(1-\left(\frac{w_i}{W_i}\right)^2\right)$ is non-zero, and this suffices to ensure that the obtained sequence $(p_k)_{k\geq 1}$ is a probability distribution.
Thanks to the weak convergence \eqref{wrt:eq:convergence pushforward measures}, it corresponds to the (annealed) distribution $p_k=\Pp{D_\infty\wedge D_\infty'=u_k}$, where $D_\infty$ and $D_\infty'$ are two independent points taken under the measure $\mu$, conditionally on $(\ttT_n)_{n\geq 1}$. Now we can write
\begin{align*}
\Pp{\dist^{\mathrm{exp}}(D_\infty,D_\infty')\leq e^{-k}} = \Pp{\haut(D_\infty \wedge D_\infty ')\geq k}\leq \sum_{i=k+1}^{\infty} p_i,
\end{align*}
where the inequality is due to the fact that the vertices $u_1, u_2, \dots,u_k$ have a height smaller than $k$. Hence $\Pp{\dist^{\mathrm{exp}}(D_\infty,D_\infty')=0}\leq \lim_{k\rightarrow \infty } \Pp{\dist^{\mathrm{exp}}(D_\infty,D_\infty')\leq e^{-k}} =0$. So, almost surely, two points taken independently under $\mu$ are different, and this ensures that $\mu$ is diffuse. 
\end{proof}

\subsubsection{Convergence of other sequences of measures.}
We also study two other sequences of measures $(\eta_n)$ and $(\nu_n)$ carried on the Ulam tree $\bU$. For every $n\geq 2$, these measures only charge the vertices $\{u_1,u_2,\dots,u_n\}$ in such a way that for any $1\leq k \leq n$,
\begin{equation*}%
\eta_n(\{u_k\})=\frac{b_k+\deg^{+}_{\ttT_n}(u_k)}{B_n+n-1} \qquad \text{and} \qquad \nu_n(\{u_k\})=\frac{1}{n},
\end{equation*}
where $(b_n)_{n\geq 1}$ is a sequence of real numbers such that $b_1>-1$ and $b_n\geq 0$ for all $n\geq 2$. We write $B_n:= \sum_{k=1}^{n}b_k$. We suppose that $B_n=\grandO{n}$ and that there exists $\epsilon>0$ such that $b_n=\grandO{n^{1-\epsilon}}$. The assumptions on the sequence $(b_n)_{n\geq 1}$ are chosen such that they are satisfied by a sequence $(a_n)_{n\geq 1}$ of fitnesses that satisfies \eqref{wrt:assum: pa An=cn+rn} for some $c>0$.
\begin{proposition}\label{wrt:prop:convergence other measures}
Under the assumptions $\sum_{n=1}^\infty w_n=\infty$ and $\sum_{n=1}^\infty \left(\frac{w_n}{W_n}\right)^2<\infty$, the sequences $(\eta_n)_{n\geq 1}$ and $(\nu_n)_{n\geq 1}$ converge almost surely weakly towards the limiting measure $\mu$ on $\partial \bU$ that is defined in Theorem~\ref{wrt:thm:weak convergence of measures}.
\end{proposition}
The rest of this section is devoted to proving Proposition~\ref{wrt:prop:convergence other measures}. We treat the two sequences of measures separately. 

\paragraph{The degree measure.}
Consider the sequence $(\eta_n)_{n\geq 1}$ on $\overline{\bU}$. 
Since the sequence $(W_n)_{n\geq1}$ tends to infinity, we have $\eta_n(\{u\})\rightarrow 0$ for every $u\in\bU$. 
Indeed, using the equality in distribution \eqref{wrt:eq:equality distribution degree} and Lemma~\ref{wrt:lem:sum bernoulli} in the appendix, it is easy to see that either $\sum_{i=1}^\infty W_i^{-1}<\infty$ in which case the degrees $\deg^{+}_{\ttT_n}(u_k)$ are eventually constant as $n\rightarrow\infty$; 
or $\sum_{i=1}^\infty W_i^{-1}=\infty$, in which case we have the almost sure asymptotic behaviour $\deg^{+}_{\ttT_n}(u_k)\sim w_k\cdot \sum_{i=k}^n W_i^{-1}$. 
In both cases, for all $k\geq 1$, we have $n^{-1}\deg^{+}_{\ttT_n}(u_k)\rightarrow 0$ almost surely as $n\rightarrow\infty$. 

For all $k\geq n$, we keep the notation $M_n^{(k)}=\mu_n(T(u_k))$ introduced in the proof of Theorem~\ref{wrt:thm:weak convergence of measures} and let \[N_n^{(k)}:=\eta_n(T(u_k)).\]
We can check that
\begin{align*} 
N_{n+1}^{(k)}= \frac{1}{B_{n+1}+n}\cdot\left((B_n+n-1)\cdot N_n^{(k)}+(b_{n+1}+1)\cdot \ind{u_{n+1}\in T(u_k)}\right).
\end{align*}
Now, using that $\Ppsq{u_{n+1}\in T(u_k)}{\cF_n}=M_n^{(k)}$ and that $\Ecsq{M_{n+1}^{(k)}}{\cF_n}=M_n^{(k)}$, we get
\begin{align*}
\Ecsq{N_{n+1}^{(k)}-M_{n+1}^{(k)}}{\mathcal{F}_n}&=\frac{B_n+n-1}{B_{n+1}+n}\cdot N_{n}^{(k)} +\frac{b_{n+1}+1}{B_{n+1}+n}\cdot M_{n}^{(k)}-M_{n}^{(k)}\\
&=\frac{B_n+n-1}{B_{n+1}+n}\cdot\left(N_{n}^{(k)}-M_{n}^{(k)}\right).
\end{align*}
Hence, if we denote $X_n^{(k)}:= (B_{n}+n-1)\cdot \left(N_{n}^{(k)}-M_{n}^{(k)}\right)$, then the last computation shows that $\left(X_n^{(k)}\right)_{n\geq k}$ is a martingale for the filtration generated by $(\ttT_n)_{n\geq 1}$.
More precisely we can write
\begin{equation*}
X_{n+1}^{(k)}-X_n^{(k)}=\underset{c_n}{\underbrace{\left((1+b_{n+1})-\frac{w_{n+1}}{W_{n+1}}\left(B_{n+1}+n\right)\right)}}\cdot \left(\ind{u_{n+1}\in T(u_k)}-M_n^{(k)}\right),
\end{equation*}
hence we have
\begin{equation*}
\Ecsq{ X_{n+1}^{(k)}-X_n^{(k)}}{\cF_n}=0 \qquad \text{and} \qquad  \Ec{\left(X_{n+1}^{(k)}-X_n^{(k)}\right)^2}\leq c_n^2.
\end{equation*}
Then, using \cite[Chapter~VII.9, Theorem~3]{feller_introduction_1971}, we get that if
\begin{equation}\label{wrt:eq:somme second moment finie}
\sum_{n=k}^{\infty}n^{-2}c_n^2<\infty
\end{equation}
then $\frac{X_n^{(k)}}{n}\rightarrow 0$ a.s. as $n\rightarrow \infty$, which would prove that $	N_n^{(k)}\underset{}{\longrightarrow}M_{\infty}^{(k)}$ as $n\rightarrow\infty$. In our case, we can verify that \eqref{wrt:eq:somme second moment finie} holds. Indeed, using the fact that we assumed that $B_n=\grandO{n}$ and $b_{n+1}=\grandO{n^{1-\epsilon}}$, we have
\begin{align*}
n^{-2}c_n^2&=n^{-2}\left((1+b_{n+1})-\frac{w_{n+1}}{W_{n+1}}\left(B_{n+1}+n\right)\right)^2\\
&\leq n^{-2} \cdot 3\left(1+b_{n+1}^2+\left(\frac{w_{n+1}}{W_{n+1}}\left(B_{n+1}+n\right)\right)^2\right)\\
&\leq 3 n^{-2}+ 3 b_{n+1}^2 n^{-2} + \cst\cdot \left(\frac{w_{n+1}}{W_{n+1}}\right)^2,
\end{align*}
which is summable under our assumptions.
In the end, using Lemma~\ref{wrt:lem:convergence measure}, we have the almost sure convergence
\[\eta_n\overset{}{\rightarrow}\mu \qquad \text{weakly}.\]

\paragraph{The uniform measure on the vertices of $\ttT_n$.}Consider the sequence $(\nu_n)$ on $\overline{\bU}$.
Fix $k\geq 1$. For any $n\geq k$ we can write $\nu_n(T(u_k))=\frac{1}{n}\sum_{i=k}^{n}\ind{u_i\in T(u_k)}$. For any $i\geq k+1$, we have $\mathsf p_i:=\Ppsq{u_i\in T(u_k)}{\cF_{i-1}}=\mu_{i-1}(T(u_k))$, which tends a.s.\ to some limit $\mu(T(u_k))$ as $i\rightarrow\infty$. Using Lemma~\ref{wrt:lem:sum bernoulli} in the appendix, we have 
\begin{align*}
	\frac{\sum_{i=k+1}^{n}\ind{u_i\in T(u_k)}}{\sum_{i=k+1}^{n}\mathsf p_i}\overset{}{\underset{n\rightarrow\infty}{\longrightarrow}} 1 \quad \text{a.s.\ on the event} \quad \left\lbrace\sum_{i=k+1}^{\infty}\mathsf p_i=\infty\right\rbrace,
\end{align*}
and also
\begin{align*}
	\sum_{i=k+1}^{n}\ind{u_i\in T(u_k)} \quad \text{converges a.s.\ on the event} \quad   \left\lbrace\sum_{i=k+1}^{\infty}\mathsf p_i<\infty\right\rbrace.
\end{align*}
In both cases we get $\nu_n(T(u_k))\underset{n\rightarrow\infty}{\rightarrow} \lim_{i\rightarrow\infty}\mathsf{p}_i=\mu(T(u_k))$ almost surely. %
We also have for any $k\geq 1$, \[\nu_n(\{u_k\})=\frac{1}{n} \underset{n\rightarrow\infty}{\rightarrow}0 \qquad \text{and of course} \qquad \forall u \notin \{u_1,u_2,\dots\}, \forall n\geq 1,\ \nu_n(\{u\})=\nu_n(T(u))=0,\] 
so we can conclude using Lemma~\ref{wrt:lem:convergence measure} that almost surely $\nu_n \underset{n\rightarrow\infty}{\rightarrow} \mu$ weakly.

\section{Height and profile of WRT}\label{wrt:sec:height and profile of WRT}
The main goal of this section is to prove Theorem~\ref{wrt:thm:profile and height of wrt} which gives asymptotics for the profile and height of the tree. Recall that we denote
\begin{align*}
	\mathbb L_n(k):=\#\enstq{1\leq i\leq n}{\haut(u_i)=k},
\end{align*}
the number of vertices at height $k$ in the tree $\ttT_n$.
In order to get information on the sequence of functions $(k\mapsto \bL_n(k))_{n\geq 1}$ we study their Laplace transform
\begin{equation}\label{wrt:eq:def laplace tranform true profile}
z\mapsto  \sum_{k=0}^{\infty}\bL_n(k)e^{kz}=\sum_{i=1}^ne^{z \haut(u_i)}=n\cdot \int_{\bU}e^{z\haut(u)}\dd\nu_n(u),
\end{equation}
where the last expression is given using an integral against the probability measure $\nu_n$ defined in Section~\ref{wrt:subsec:convergence of measure} as the uniform measure on the vertices of $\ttT_n$. The key result in our approach is to prove the convergence of this sequence of analytic functions when appropriately rescaled, uniformly in $z$ on an open neighbourhood of $0$ in the complex plane. It then allows us to use \cite[Theorem~2.1]{kabluchko_general_2017} and hence derive a convergence result for the profile.
We actually start in Section~\ref{wrt:subsec:laplace transform weighted profile} by studying the convergence of the similarly defined sequence of functions
\begin{equation}\label{wrt:eq:def laplace transform weighted profile}
z\mapsto\int_{\bU}e^{z\haut(u)}\dd\mu_n(u)=\sum_{i=1}^{n}\frac{w_i}{W_n}e^{z\haut(u_i)},
\end{equation}
where we integrate with respect to the weight measure $\mu_n$ instead of the uniform measure $\nu_n$ as before. This one is easier to study because for every fixed $z\in\mathbb{C}$, it defines a martingale as $n$ grows, up to some deterministic scaling. Then in Section~\ref{wrt:subsec:from the weighted to unweighted sum}, we make use of this first convergence and show that up to some deterministic multiplicative constant, the two sequences of integrals appearing in \eqref{wrt:eq:def laplace tranform true profile} and \eqref{wrt:eq:def laplace transform weighted profile} are almost surely equivalent when $n$ tends to infinity.

\textbf{Let us fix some $\gamma>0$ for this whole section. 
Throughout this section, we always work under the assumption that \eqref{wrt:assum: wrt alpha} holds for the sequence $\boldsymbol{w}$.}
For some results, we will assume that their exists  $p\in\intervalleof{1}{2}$ such that the stronger condition \eqref{wrt:assum: wrt alpha p}  holds, i.e.
\begin{align*}
W_n\underset{n\rightarrow\infty}{\bowtie} \cst \cdot n^{\gamma} \qquad \text{and} \qquad \sum_{i=n}^{2n}w_n^p\leq n^{1+(\gamma-1) p + \petito{1}}.
\end{align*}

We let $\phi: z\mapsto \gamma(e^z-1)$ be a function of a complex parameter $z$ and let $z\mapsto N_n(z)$ be the following rescaled version of the Laplace transform of the profile 
\begin{align}\label{wrt:eq:def Nnz}
	N_n(z):=n^{-(1+\phi(z))}\sum_{k=0}^\infty\mathbb L_n(k)e^{z k}.
\end{align}
The proposition below ensures that the sequence $(z\mapsto N_n(z))_{n\geq 1}$ converges uniformly on all compact subsets of some domain $\mathscr D\subset \mathbb C$ to some limiting function $z\mapsto N_\infty(z)$ which does not vanish anywhere on the set $\mathscr D\cap \R$, along with some more technical statements.
\begin{proposition}\label{wrt:prop:assumptions kabluchko}
Suppose that the weight sequence $\boldsymbol{w}$ satisfies \eqref{wrt:assum: wrt alpha p} for some $\gamma>0$ and some $p\in \intervalleof{1}{2}$.
Then there exists a domain $\mathscr D\subset \mathbb C$ such that $\mathscr D\cap \R=\intervalleoo{z_{-}}{z_{+}}$ 
where $z_-<0$ and $z_+>0$ are defined as in \eqref{wrt:eq:definition z+ and z-},
 such that the following properties are satisfied.
\begin{enumerate}
	\item\label{wrt:it:convergence Nn sur tout compact} With probability $1$, the sequence of random analytic functions $(z\mapsto N_n(z))_{n\geq 1}$ converges uniformly on all compact subsets of $\mathscr{D}$, as $n\rightarrow\infty$, to some random
	analytic function $z\mapsto N_\infty(z)$ which satisfies $\Pp{N_\infty(z) \neq  0 \text{ for all } z \in (z_{-} ,z_{+})} = 1$.
	\item\label{wrt:it:convergence Nn polynomiale} For every compact subset $K\subset \mathscr{D}$ and $r\in\N$, we can find an
	a.s. finite random variable $C_{K,r}$ such that for all $n\in\N$,
\[\sup_{z\in K}\abs{N_n(z)-N_\infty(z)}< C_{K,r} (\log n)^{-r} .\]
\item\label{wrt:it:convergence Im>0} For every compact subset $K\subset \intervalleoo{z_{-}}{z_{+}}$, every $0<a<\pi$ and $r\in \N$,
\begin{align*}
	\sup_{z\in K }\left[e^{-(1+\phi(z))\log n}\int_{a}^{\pi}\abs{\sum_{k=0}^{\infty}\bL_n(k)e^{z+iu}} \dd u\right]=\petito{(\log n)^{-r}} \quad \text{a.s. as } n\rightarrow \infty.
\end{align*}
\end{enumerate}
\end{proposition}
Under the results of Proposition~\ref{wrt:prop:assumptions kabluchko} we can apply  \cite[Theorem~2.1]{kabluchko_general_2017} whose conclusions for the sequence $(k\mapsto \bL(k))_{n\geq 1}$ are the following.
For any $k\geq 0,\ n\geq 1$ and $z\in\intervalleoo{z_-}{z_+}$, we denote \[x_n(k;z)=\frac{k-\gamma e^z \log n}{\sqrt{\gamma e^z \log n}}.\] Then, for every integer $r\geq 0$ and every compact subset $K\subset \intervalleoo{z_-}{z_+}$, we have the convergence
\begin{align}\label{wrt:eq:general Edgeworth expansion}
(\log n)^{\frac{r+1}{2}}\cdot \sup_{k\in \N}\sup_{z\in K}\abs{e^{z k-(1+\phi(z))\log n}\bL_n(k) - \frac{N_\infty(z)e^{-\frac{1}{2}x_n(k;z)^2}}{\sqrt{2\pi \log n}} \sum_{j=0}^{r}\frac{G_j(x_n(k);z)}{(\log n)^{j/2}}}\overset{\text{a.s.}}{\underset{n\rightarrow\infty}{\longrightarrow}} 0,
\end{align}
where for all $j\geq 0$, the (random) functions $G_j(x,z)$ are polynomials of degree at most $3$ in $x$ and are entirely determined from $\phi$ and $N_\infty$, with $G_0=1$, see \cite[Equation (16)]{kabluchko_general_2017} for their complete definition. 
The asymptotics \eqref{wrt:eq:convergence uniforme profil} and \eqref{wrt:eq:radish asymptotics} stated in Theorem~\ref{wrt:thm:profile and height of wrt} follow from the last display. Indeed, \eqref{wrt:eq:convergence uniforme profil} is obtained by letting $r=0$ and $z=0$ and using the fact that $N_\infty(0)=1$ almost surely. For \eqref{wrt:eq:radish asymptotics}, we let $r=0$, and use $k=\left\lfloor \gamma e^z \log n\right\rfloor$.

In Section~\ref{wrt:subsec:upper-bounds on height}, we complete the proof of Theorem~\ref{wrt:thm:profile and height of wrt} by computing the asymptotic behaviour of the height of the tree. Since the convergence of the profile already ensures that there almost surely are vertices at height $\gamma e^{(z_+-\epsilon)}\log n$ for $\epsilon>0$ small enough and all $n$ large enough, it suffices to prove a corresponding upper-bound in order to finish proving the convergence \eqref{wrt:eq:convergence height gamma p} in Theorem~\ref{wrt:thm:profile and height of wrt}.
\subsection{Study of the Laplace transform of the weighted profile}\label{wrt:subsec:laplace transform weighted profile}
We study the sequence $\left(z\mapsto \sum_{i=1}^{n}\frac{w_i}{W_{n}}e^{z\haut(u_i)}\right)_{n\geq 1}$. The following lemma is the starting point of our analysis. 
In this section, we will use the notation $\cF_n=\sigma(\ttT_1,\ttT_2,\dots,\ttT_n)$.
\begin{lemma}\label{wrt:lem:Mn martingale}
For all $z\in\mathbb C$ and all $n\geq1$, we have 
\begin{align*}
\Ecsq{\sum_{i=1}^{n+1}\frac{w_i}{W_{n+1}}e^{z\haut(u_i)}}{\cF_n}
&=\left(1+(e^z-1)\frac{w_{n+1}}{W_{n+1}}\right) \cdot \sum_{i=1}^{n}\frac{w_i}{W_{n}}e^{z\haut(u_i)}.
\end{align*}
\end{lemma}
\begin{proof}
Recall that conditionally on $\cF_n$, the $n+1$-st vertex $u_{n+1}$ of $\ttT_{n+1}$ is a child of the vertex $u_{K_{n+1}}$, where $\Ppsq{K_{n+1}=k}{\cF_n}=\frac{w_k}{W_n}$. We compute 
\begin{align*}
\sum_{i=1}^{n+1}\frac{w_i}{W_{n+1}}e^{z\haut(u_i)}&=\frac{W_n}{W_{n+1}}\sum_{i=1}^{n}\frac{w_i}{W_n}e^{z\haut(u_i)}+ \frac{w_{n+1}}{W_{n+1}}\cdot e^z \cdot e^{z\haut(u_{K_{n+1}})}.
\end{align*}
Taking conditional expectation with respect to $\cF_n$ yields:
\begin{align*}
\Ecsq{\sum_{i=1}^{n+1}\frac{w_i}{W_{n+1}}e^{z\haut(u_i)}}{\cF_n}&=\frac{W_n}{W_{n+1}}\cdot \sum_{i=1}^{n}\frac{w_i}{W_{n}}e^{z\haut(u_i)}+\frac{w_{n+1}}{W_{n+1}}\cdot e^z\cdot \sum_{i=1}^{n}\frac{w_i}{W_n}e^{z\haut(u_i)}\\
&=\left(1+(e^z-1)\frac{w_{n+1}}{W_{n+1}}\right) \cdot \sum_{i=1}^{n}\frac{w_i}{W_{n}}e^{z\haut(u_i)}.
\end{align*}
This concludes the proof.
\end{proof}
Let $J$ be an integer that we are going to fix later on. 
The last result ensures that if $z\in \mathbb C$ is such that $\forall i\geq J,\ 1+(e^z-1)\frac{w_i}{W_i}\neq 0$, then we can define for all $n\geq J$
\begin{align*}
C_n(z):=\prod_{i=J}^n\left(1+(e^z-1)\frac{w_i}{W_i}\right) \qquad \text{and}\qquad	M_n(z):=\frac{1}{C_n(z)}\sum_{i=1}^{n}\frac{w_i}{W_n}e^{z\haut(u_i)},
\end{align*}
and the sequence $(M_n(z))_{n\geq J}$ is a martingale. 
We want to prove results about the asymptotic behaviour of $(z\mapsto M_n(z))_{n\geq J}$, uniformly in $z$ on an appropriate set. If $J$ is fixed, then there exist parameters $z$ with $\Im(z)=\pi \mod 2\pi$ for which the sequence $(C_n(z))_{n\geq J}$ takes the value $0$. 
Under the assumption \eqref{wrt:assum: wrt alpha} on the sequence $\boldsymbol{w}$, we know that $\frac{w_n}{W_n}\underset{}{\rightarrow}0$ as $n\rightarrow\infty$.  
If we restrict ourselves to a set of the form $\enstq{z\in \mathbb C}{\Re(z)<x}$ for some $x>0$, then 
\[\abs{1+(e^z-1)\frac{w_n}{W_n}}\geq 1-\abs{e^z-1}\cdot \frac{w_n}{W_n}\geq 1-(e^x+1)\cdot \frac{w_n}{W_n} \underset{n\rightarrow\infty}{\rightarrow} 1>0,\]
hence it suffices to take $J$ large enough in order for the sequence $(C_n(z))_{n\geq J}$ to only take non-zero values for all $z\in \enstq{\xi\in \mathbb C}{\Re(\xi)<x}$ and all $n\geq J$.
In what follows we work on the set 
\begin{align*}
\mathscr{E}=\enstq{z\in \mathbb C}{\Re z <z_{+}},
\end{align*}
where $z_{+}$ is as defined in Proposition~\ref{wrt:prop:assumptions kabluchko}. 
For technical reasons, we also sometimes consider the larger set
\begin{align*}
\mathscr{E}'=\enstq{z\in \mathbb C}{\Re z <2z_{+}}.
\end{align*}
Using the preceding discussion, we fix $J\geq 1$ such that the sequence $z\mapsto (C_n(z))_{n\geq J}$ does not have any zero on $\mathscr E'$, so that $z\mapsto (M_n(z))_{n\geq J}$ is well-defined for all $z\in \mathscr E'$.

We introduce the following notation. Let $F(z,n)$ and $G(z,n)$ be two functions of a complex parameter $z$ and an integer $n\in \N$. For $D\subset \mathbb C$ a set of the complex plane we write
\begin{align}\label{wrt:eq:big and small o notation}
F(n,z)=\grandOdom{D}{G(n,z)} \qquad (\text{resp.} \qquad F(n,z)=\petitodom{D}{G(n,z)})
\end{align}
to express the fact that $F(n,z)$ is a big (resp. small) o of $G(n,z)$ as $n\rightarrow\infty$, uniformly on every compact $K\subset D$. 
Note that later in the paper, we will also use this notation for random functions of $z$ and $n$ when such a comparison holds almost surely. 

Now, let us derive some information on the asymptotic behaviour of $C_n(z)$.
\begin{lemma}\label{wrt:lem:behaviour Cn}
Suppose that $\boldsymbol{w}$ satisfies \eqref{wrt:assum: wrt alpha}. Then there exists $\epsilon>0$ and an analytic function $z\mapsto c(z)$ on $\mathscr E'$ such that 
\begin{align*}
	C_n(z)=\exp(\phi(z)\log n +c(z)+\grandOdom{\mathscr E}{n^{-\epsilon}}).
\end{align*}

\end{lemma}
Remark that the lemma implies that for any $z\in \mathscr E'$, we have \[\abs{C_n(z)}\underset{}{\sim} e^{\Re( c(z))} \cdot n^{\Re \phi(z)}\]
as $n\rightarrow\infty$.  It is also immediate that $\Ec{\sum_{k=1}^{n}\frac{w_k}{W_n}e^{z\haut(u_k)}}=\Ec{M_J(z)}\cdot C_n(z)$ satisfies the same asymptotics up to a constant, as soon as $z$ is such that $\Ec{M_J(z)}\neq 0$.

Before proving the lemma, we state the following result which follows from elementary calculus. Its proof can be found in the appendix. 
\begin{lemma}\label{wrt:lem:behaviour sum w/W}
	Suppose that $\boldsymbol{w}$ satisfies \eqref{wrt:assum: wrt alpha}. Then there exists $\epsilon>0$ such that 
	\begin{align*}
	\sum_{i=n}^{+\infty}\left(\frac{w_i}{W_i}\right)^2=\grandO{n^{-\epsilon}} \quad \text{and also} \quad 
	\sum_{i=1}^n\frac{w_i}{W_i}=\gamma\log n+\cst + \grandO{n^{-\epsilon}}.
	\end{align*}
\end{lemma}
\begin{proof}[Proof of Lemma~\ref{wrt:lem:behaviour Cn}]
We write $\Log$ for the principal value of the complex logarithm. For $z\in \mathbb{C}$ such that $\abs{z}<1$ we have $\Log(1+z)=\sum_{i=1}^{\infty}\frac{(-1)^{n-1}}{n}z^n$. 
If for every $i\geq J$ and $z\in\mathscr E'$, we let
\begin{align*}
h(i,z)=\Log\left(1+(e^z-1)\frac{w_i}{W_i}\right)-(e^z-1)\frac{w_i}{W_i},
\end{align*} 
which is well-defined thanks to our choice of $J$, then $\abs{h(i,z)}=\grandOdom{\mathscr E'}{\left(\frac{w_i}{W_i}\right)^2}$ is summable in $i$ and the rest of the series is 
\begin{align}\label{wrt:eq:remainder series h}
	\abs{\sum_{i=n}^{\infty}h(i,z)}\leq \sum_{i=n}^{\infty}\abs{h(i,z)}=
	\grandOdom{\mathscr E'}{\sum_{i=n}^{\infty}\left(\frac{w_i}{W_i}\right)^2}= \grandOdom{\mathscr E'}{n^{-\epsilon}},
\end{align}
for some $\epsilon>0$, thanks to Lemma~\ref{wrt:lem:behaviour sum w/W}.
Then we write
\begin{align*}
	C_n(z)=\prod_{i=J}^{n}\left(1+(e^z-1)\frac{w_i}{W_i}\right)=\exp\left((e^z-1)\sum_{i=J}^{n}\frac{w_i}{W_i}+\sum_{i=J}^n h(i,z)\right)
\end{align*}
which yields using \eqref{wrt:eq:remainder series h} and Lemma~\ref{wrt:lem:behaviour sum w/W}
\begin{align*}
C_n(z)	&= \exp\left((e^z-1)(\gamma\log n + \cst + \grandOdom{\mathscr E'}{n^{-\epsilon}})+\sum_{i=J}^\infty h(i,z)-\sum_{i=n+1}^\infty h(i,z)\right)\\
	&=\exp(\phi(z)\log n +\underset{c(z)}{\underbrace{(e^z-1)\cdot \cst +\sum_{i=J}^\infty h(i,z)}} +\grandOdom{\mathscr E'}{n^{-\epsilon}}),
\end{align*}
and $c(z)$ is an analytic function of $z$, which finishes the proof.
\end{proof}

\paragraph{Convergence of the martingales $(M_n(z))_{n\geq 1}$.}
When the parameter $z$ is a positive real number, the sequence $(M_n(z))_{n\geq 1}$ is a positive martingale and so it converges almost surely to some limit. We want to prove that these martingales converge almost surely and in $L^1$ for the largest possible range of parameters $z$. 
\textbf{For the rest of Section~\ref{wrt:subsec:laplace transform weighted profile} and also in the subsequent Section~\ref{wrt:subsec:from the weighted to unweighted sum}, we assume that the weight sequence $\boldsymbol{w}$ satisfies \eqref{wrt:assum: wrt alpha p} for some fixed parameters $\gamma>0$ and $p\in \intervalleof{1}{2}$.}

We align our notation with the one used in \cite[Theorem~2.2]{chauvin_martingales_2005} which states something similar to our forthcoming Proposition~\ref{wrt:prop:convergence Mnbeta} for another model, the binary search tree.

\begin{figure}
\centering
\begin{tabular}{ccc}
	\subfloat[$\gamma=2$]{\includegraphics[height=5cm]{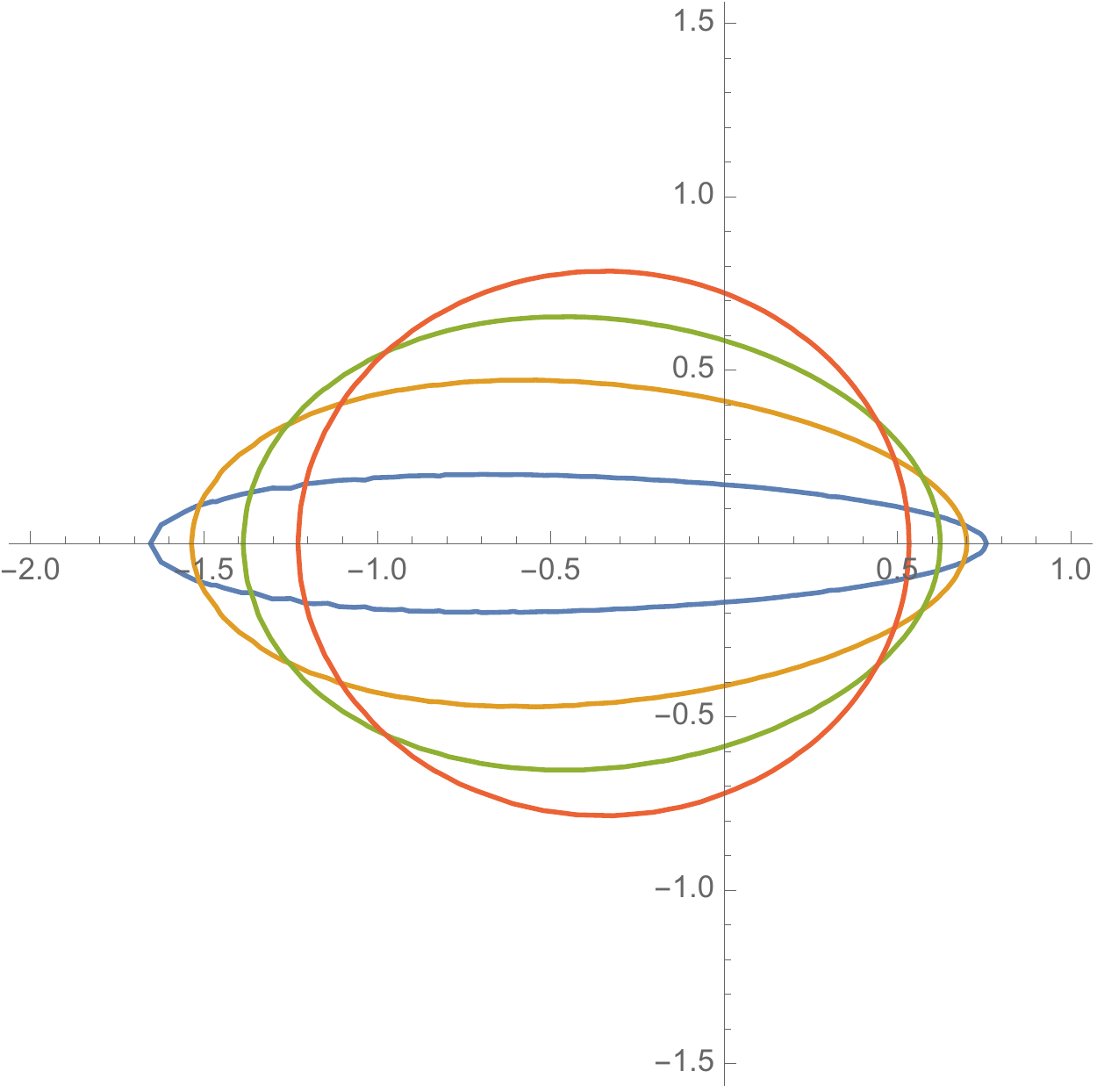}} &  &
	\subfloat[$\gamma=\frac12$]{\includegraphics[height=5cm]{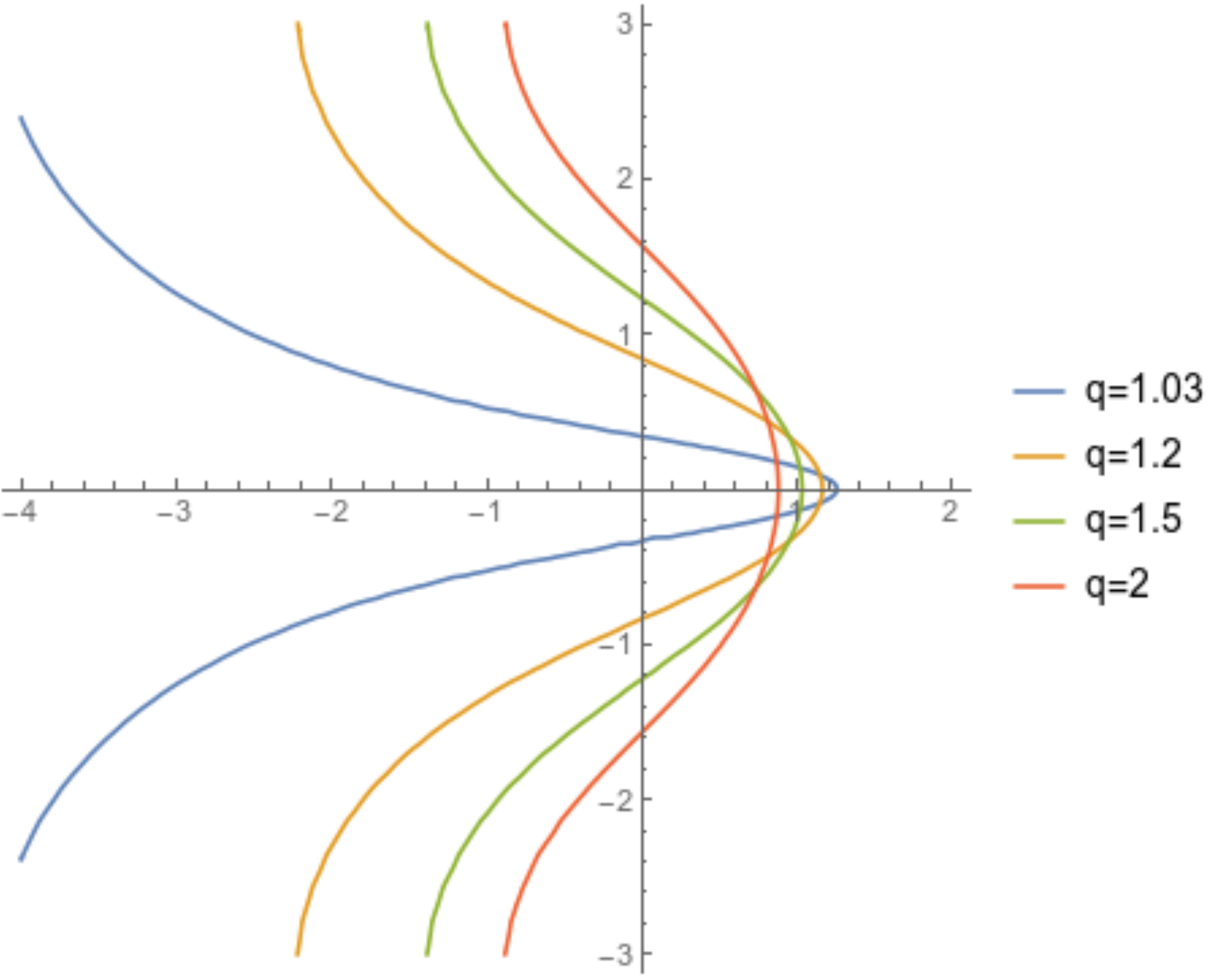}}
\end{tabular}
	\caption{The boundary of the connected component of $\mathscr V_q$ that contains $0$, for some values of $q\in\intervalleof{1}{2}$, plotted for $\gamma=2$ and $\gamma=\frac{1}{2}$.}
\end{figure}

For any $z\in \mathscr E$ and $q\in \intervalleof{1}{p}$, we let
\begin{align}\label{wrt:eq:definition g(z,q)}
g(z,q):=\phi(q\Re z) - q \Re( \phi(z)) -q+1=\gamma (e^{q \Re z} -1 -q \Re(e^z)+q)-q+1.
\end{align}
	For any $q\in \intervalleof{1}{p}$, let $\mathscr V_q=\enstq{z\in \mathscr E}{g(z,q)<0}$, and denote
	\begin{align}\label{wrt:eq:def domain V}
	\mathscr V=\bigcup_{1<q\leq p} \mathscr V_q.
	\end{align}
\begin{lemma}\label{wrt:lem:V open domain and contains open interval}
	The set $\mathscr V$ is open and contains the open interval of real numbers $I_\gamma:=\enstq{x\in\R}{\gamma(x e^x - e^x+1)-1<0}$ which contains $0$.  
\end{lemma}
\begin{proof}
Of course $\mathscr V$ is open as a union of open sets. For any real $x$ we have $g(x,1)=0$. So, if $\frac{\partial g}{\partial q}(x,1)<0$ then there exists $q>1$ for which $g(x,q)<0$. Since $\frac{\partial g}{\partial q}(x,1)=\gamma(x e^x - e^x+1)-1$, the set $\mathscr V$ contains the interval $I_\gamma$ defined above. Since $\frac{\partial g}{\partial q}(0,1)=-1<0$, we have $0\in I_\gamma$.
\end{proof}
	\begin{proposition}\label{wrt:prop:convergence Mnbeta}
	The sequence of functions $(z\mapsto M_n(z))_{n\geq J}$ converges uniformly almost surely and in $L^1$ towards an analytic function $z \mapsto M_\infty(z)$ on every compact subset of $\mathscr V$. Furthermore, for any compact subset $K\subset \mathscr V$, there exists a real $\epsilon(K)>0$ such that almost surely
	\begin{align*}
	\abs{M_n(z)-M_\infty(z)}=\grandOdom{K}{ n^{-\epsilon(K)}}.
	\end{align*}
\end{proposition}
The proof of the proposition will follow from the next lemma, together with Lemma~\ref{wrt:lem:convergence analytic martingale}, stated in the appendix. 
\begin{lemma}\label{wrt:lem:moments Mn}
	For any $q\in \intervalleof{1}{p}$ and $z\in \mathscr E$ we have
	\begin{align}\label{wrt:eq:martingales q-moment}
	\Ec{\abs{M_n(z)}^q}=\grandOdom{\mathscr E}{n^{0\vee g(z,q)+\petitodom{\mathscr E}{1}}}.
	\end{align}
	and also
	\begin{align}\label{wrt:eq:M2n-Mn^q is big o of}
	\Ec{\abs{M_{2n}(z)-M_{n}(z)}^q}=\grandOdom{\mathscr E}{n^{(1-q)\vee g(z,q)+\petitodom{\mathscr E}{1}}}.
	\end{align}
\end{lemma}
\begin{proof}
For any $q\in \intervalleof{1}{p}$ and $n\geq J$, we write
\begin{align*}
	M_{n+1}(z)-M_n(z)=M_n(z)\left(\frac{C_n(z)}{C_{n+1}(z)}-1\right)+\frac{1}{C_{n+1}(z)}\cdot \frac{w_{n+1}}{W_{n+1}}\cdot e^{z\haut(u_{n+1})}.
\end{align*}
Taking the $q$-th power of the modulus on both sides and using the inequality $\abs{a+b}^q\leq 2^q \cdot (\abs{a}^q+\abs{b}^q)$, we get 
\begin{align}\label{wrt:eq:moment increment Mn power q}
	&\Ec{\abs{M_{n+1}(z)-M_n(z)}^q}\notag \\
	&\leq \Ec{\abs{M_n(z)}^q}\cdot 2^q\abs{\frac{C_n(z)}{C_{n+1}(z)}-1}^q + 2^q \frac{1}{\abs{C_{n+1}(z)}^q}\left(\frac{w_{n+1}}{W_{n+1}}\right)^q \cdot \Ec{\abs{e^z}^{q\haut(u_{n+1})}}.
\end{align}
Using Lemma~\ref{wrt:lem:biggins lemma} in the appendix, we have for any $n\geq J$, 
\begin{align*}
\Ec{\abs{M_{n+1}(z)}^q}\leq \Ec{\abs{M_n(z)}^q}+2^q\cdot \Ec{\abs{M_{n+1}(z)-M_n(z)}^q}.
\end{align*}
Using the last display and equation~\eqref{wrt:eq:moment increment Mn power q}, we get a recurrence inequality of the form
\begin{align}\label{wrt:eq:upper bound increments Mn q}
\Ec{\abs{M_{n+1}(z)}^q}\leq (1+ a_n(z))\cdot \Ec{\abs{M_n(z)}^q}+ b_n(z),
\end{align}
where 
\[a_n(z)=2^{2q}\abs{\frac{C_n(z)}{C_{n+1}(z)}-1}^q \quad \text{and}\quad  b_n(z)=2^{2q} \frac{1}{\abs{C_{n+1}(z)}^q}\left(\frac{w_{n+1}}{W_{n+1}}\right)^q \cdot \Ec{\abs{e^z}^{q\haut(u_{n+1})}}.\]
Applying \eqref{wrt:eq:upper bound increments Mn q} in cascade we get
\begin{align}\label{wrt:eq:expectation M2n-Mn^q}
\Ec{\abs{M_{n}(z)}^q}\leq \prod_{i=J}^{n-1}(1+a_i(z)) \cdot \left(\Ec{\abs{M_J(z)}^q}+ \sum_{i=J}^{n-1}b_i(z)\right).
\end{align}
Now notice that from our assumption on the sequence $(w_n)_{n\geq 1}$ we have 
\begin{align}\label{wrt:eq:an big o}
	a_n(z)=2^{2q}\abs{\frac{C_n(z)}{C_{n+1}(z)}-1}^q=2^{2q}\abs{\frac{1}{1+(e^z-1)\frac{w_{n+1}}{W_{n+1}}}-1}^q=\grandOdom{\mathscr E}{\left(\frac{w_{n+1}}{W_{n+1}}\right)^q}.
\end{align}
On the other hand, since $z\in \mathscr E$ then $q \Re z\in \mathscr E'$, so we can use  Lemma~\ref{wrt:lem:behaviour Cn} to get
\begin{align}\label{wrt:eq:sum bn big o}
b_n(z)&=\cst \cdot \left(\frac{w_{n+1}}{W_{n+1}}\right)^q \cdot \abs{C_{n+1}(z)}^{-q}\cdot e^{q\Re z}\cdot \Ec{\sum_{k=1}^{n}\frac{w_k}{W_n}e^{(q\Re z)\haut(u_k)}} \notag\\
&=\left(\frac{w_{n+1}}{W_{n+1}}\right)^q \cdot \grandOdom{\mathscr E}{n^{-q \Re(\phi(z))}} \cdot  \grandOdom{\mathscr E}{n^{\phi(q \Re z)}} \notag\\
&=\left(\frac{w_{n+1}}{W_{n+1}}\right)^q \cdot \grandOdom{\mathscr E}{n^{g(z,q)-1+q}}.
\end{align}
We conclude using the following lemma which is an application of Hölder's inequality using the assumption \eqref{wrt:assum: wrt alpha p}.
\begin{lemma}\label{wrt:fact:app of holder inequality}
For any $q\in\intervalleof{1}{p}$ we have $\displaystyle	\sum_{i=n}^{2n}\left(\frac{w_i}{W_i}\right)^q \leq n^{1-q+\petito{1}}$.
\end{lemma}
Together with \eqref{wrt:eq:an big o}, this proves that $(a_n(z))_{n\geq 1}$ is summable and so $\prod_{i=J}^{\infty }(1+a_i(z))=\grandOdom{\mathscr E}{1}$. Also
\begin{align*}
\sum_{i=n}^{2n} b_i(z)=\grandOdom{\mathscr E}{n^{g(z,q)+\petitodom{\mathscr E}{1}}},
\end{align*}
and so $\sum_{i=J}^{n}b_i(z)=\grandOdom{\mathscr E}{n^{0\vee g(z,q)+\petitodom{\mathscr E}{1}}}$. Replacing this in \eqref{wrt:eq:expectation M2n-Mn^q} finishes to prove \eqref{wrt:eq:martingales q-moment}. 
In order to prove \eqref{wrt:eq:M2n-Mn^q is big o of}, we use Lemma~\ref{wrt:lem:biggins lemma} again and write
\begin{align*}
\Ec{\abs{M_{2n}(z)-M_{n}(z)}^q}&\leq 2^q\cdot \sum_{i=n}^{2n-1}\Ec{\abs{M_{i+1}(z)-M_{i}(z)}^q}\\
&\underset{\text{\eqref{wrt:eq:martingales q-moment},\eqref{wrt:eq:moment increment Mn power q}}}{\leq} \sum_{i=n}^{2n-1} \left(a_i(z) \cdot \grandOdom{\mathscr E}{n^{0\vee g(z,q)+\petitodom{\mathscr E}{1}}}+ b_i(z)\right)\\
&\underset{\text{\eqref{wrt:eq:an big o},\eqref{wrt:eq:sum bn big o}}}{\leq}\sum_{i=n}^{2n-1}\left(\frac{w_{i+1}}{W_{i+1}}\right)^q  \left(\grandOdom{\mathscr E}{n^{0\vee g(z,q)+\petitodom{\mathscr E}{1}}}+\grandOdom{\mathscr E}{n^{g(z,q)-1+q}}\right).
\end{align*}
Using Lemma~\ref{wrt:fact:app of holder inequality} we get $\Ec{\abs{M_{2n}(z)-M_{n}(z)}^q}=\grandOdom{\mathscr E}{n^{(1-q)\vee g(z,q)+\petitodom{\mathscr E}{1}}}$ which finishes the proof of the lemma.
\end{proof}

\begin{proof}[Proof of Proposition~\ref{wrt:prop:convergence Mnbeta}]
Any compact subset $K\subset \mathscr V_q$ can be covered by a finite number of $\mathscr V_q$. The convergence result is then an application of Lemma~\ref{wrt:lem:convergence analytic martingale}, on the set $\mathscr V_q$ with $\alpha(z)=0$ and, say $\delta(z)=-\frac{1}{2}g(z,q)>0$.  
The limiting function is analytic as a uniform limit of analytic functions. 
\end{proof}

\paragraph{Zeros of the limit.} Now that we have proved that their exists a limiting function $z\mapsto M_\infty(z)$ defined on the set $\mathscr V$, we are interested in the possible location of the zeros of this random function. In fact, the function $z\mapsto M_\infty(z)$ is related to the function $z\mapsto N_\infty(z)$ of Proposition~\ref{wrt:prop:assumptions kabluchko}, for which we aim to prove that it has almost surely no zero on some real interval $\intervalleoo{z_-}{z_+}$ which contains $0$. 
We will prove a similar result for $z\mapsto M_\infty(z)$ in Lemma~\ref{wrt:lem:Minfty almost surely no zero}, and we start by proving the following weaker statement.
Recall the definition of the interval $I_\gamma$ in Lemma~\ref{wrt:lem:V open domain and contains open interval}.
\begin{lemma}\label{wrt:lem:Minfty almost surely not zero}
For all $z\in I_\gamma$, we have almost surely $M_\infty(z)>0$. As a consequence, the number of zeros of the map $(z\mapsto M_\infty(z))$ on the interval $I_\gamma$ is almost surely at most countable. 
\end{lemma}
Let us recall from \eqref{wrt:eq:def Kn} the definition of the sequence of independent random variables $(K_2,K_3,\dots)$ that is used to construct the trees $(\ttT_n)_{n\geq 1}$. 
\begin{proof}[Proof of Lemma~\ref{wrt:lem:Minfty almost surely not zero}]This follows from an application of Kolmogorov's $0-1$ law. Indeed, fix $N\geq J$ and $z\in I_\gamma$ and for all $n\geq N$, let
\begin{align*}
	M_n^{(N)}(z)=\frac{1}{C_n(z)}\sum_{i=1}^{n}\frac{w_i}{W_n}e^{z\dist(u_i,\ttT_N)},
\end{align*} 
where $\dist$ denotes the graph distance in $\bU$, and the distance between a vertex and a subset of vertices is defined the usual way. 
The idea behind $(M_n^{(N)}(z))_{n\geq N}$ is that, up to a positive multiplicative constant (i.e. a deterministic constant that depends on $N$ and $z$ but not on $n$), it has the same distribution as the sequence $M_n(z)$ associated to the growth of the tree that one informally obtains by contracting all the vertices $\{u_1,u_2,\dots,u_N\}$ into one. 
Note that the growth of such a tree can be described as that of a weighted recursive tree with weight sequence $(W_N,w_{N+1},w_{N+2},\dots)$, which shares the same asymptotic property \eqref{wrt:assum: wrt alpha p} as the original sequence $(w_n)_{n\geq 1}$. 

We claim the following:
\begin{enumerate}
\item Due to the above remarks, $(M_n^{(N)}(z))_{n\geq N}$ is a positive martingale which satisfies the same assumptions as $M_n(z)$ so it converges a.s. and in $L^1$ towards a non-negative limit, $M_\infty^{(N)}(z)$, thanks to Proposition~\ref{wrt:prop:convergence Mnbeta}, which we can apply here because $z\in I_\gamma\subset \mathscr{V}$.
\item We have $(1\wedge e^z)^N M_n^{(N)}(z)\leq M_n(z)\leq (1\vee e^z)^N M_n^{(N)}(z)$.
\item The sequence $(M_n^{(N)}(z))_{n\geq N}$, hence its limit $M_\infty^{(N)}(z)$, is independent of the $N$ first steps of the construction, and is hence a measurable function of the sequence $(K_{n})_{n\geq N+1}$. 
\end{enumerate}
Using all these observations we deduce that for any $N\geq J$ we have the equality of events $\{M_\infty(z)>0\}=\{M^{(N)}_\infty(z)>0\}$. This proves that $\{M_\infty(z)>0\}$ is measurable with respect to the tail $\sigma$-algebra generated by the sequence $(K_n)_{n\geq 2}$, which is a sequence of jointly independent random variables. Kolmogorov's $0$-$1$ law then ensures that this event has probability $0$ or $1$. By $L^1$ convergence we have $\Ec{M_\infty(z)}=\Ec{M_J(z)}>0$ and this proves our claim.
It follows immediately that the limit $z\mapsto M_\infty(z)$ can only have finitely many zeros in any compact subset of $I_\gamma$ almost surely, because otherwise, by analyticity of $z \mapsto M_\infty(z)$ on the connected component of $\mathscr{V}$ that contains $I_\gamma$, the function would be identically $0$ on $I_\gamma$ with positive probability. 
This ensures that the total number of zeros in $I_\gamma$ is at most countable and finishes the proof. 
\end{proof} 

\begin{lemma}\label{wrt:lem:Minfty almost surely no zero}
The function $M_\infty(z)$ has almost surely no zero on $I_\gamma$.
\end{lemma}
In order to prove this lemma, we use an argument of self-similarity: essentially, if we take two vertices $u_i$ and $u_j$ in the tree, then conditionally on the sequences of vertices that are grafted above $u_i$ or above $u_j$, the subtrees above $u_i$ and $u_j$ evolve as two \emph{independent} weighted recursive trees. 
Using Proposition~\ref{wrt:prop:convergence Mnbeta} and Lemma~\ref{wrt:lem:Minfty almost surely not zero}, the normalized Laplace transform of the weighted profile of each of those two subtrees should converge almost surely to some random analytic function on $\mathscr{V}$ which is non-negative on $I_\gamma$ and has at most countably many zeros on this interval. Since the two are independent, their zeros should not overlap and hence the sum of their contribution should result in a function that is positive on $I_\gamma$.
\begin{proof}
Let us formalise this line of reasoning. Using Theorem~\ref{wrt:thm:weak convergence of measures}, we know that the measure $\mu$ on $\partial \bU$ is almost surely diffuse, hence we can define
\begin{align*}
 I^{(1)}:= \inf\enstq{i\geq 1}{\mu(T(u_i))\in\intervalleoo{0}{1}}  \quad \text{and} \quad I^{(2)}:= \inf\enstq{i\geq I_1}{u_i\notin T(u_{I_1}) \text{ and }\mu(T(u_i))\in\intervalleoo{0}{1}},
\end{align*}
and they are almost surely finite. 

Let us consider the sequences $\left(\ind{u_n\in T(u_{I^{(j)}})}\right)_{n\geq 1}$ for $j\in\{1,2\}$, which record the times when a vertex is added to $T(u_{I^{(1)}})$ or $T(u_{I^{(2)}})$, and work conditionally on them for the rest of the proof. 
We let
\begin{align*}
\forall n \geq 1, \quad N_n^{(j)}:= \sum_{i=1}^{n}\ind{u_i\in T(u_{I^{(j)}})} \qquad \text{and} \qquad \forall k\geq 1, \quad \tau^{(j)}_k:=\inf\enstq{n\geq 1}{N_n^{(j)}\geq k},
\end{align*}
which record respectively the number of vertices among $\{u_1,u_2,\dots,u_n\}$ that are in $T(u_{I^{(j)}})$ and conversely, the $k$-th time where a vertex is added to $T(u_{I^{(j)}})$ in the construction of $(\ttT_n)_{n\geq 1}$.
We let $w^{(j)}_k:=w_{\tau^{(j)}(k)}$ and $W^{(j)}_k:=\sum_{i=1}^{k}w^{(j)}_k$, and also $u_k^{(j)}:=u_{\tau_k^{(j)}}$. 
We also define for $j\in\{1,2\}$ and $k\geq 1$
\[\ttT_k^{(j)}:=\enstq{u\in\bU}{u_{I^{(j)}}u\in \ttT_{\tau_k^{(j)}}},\]
the subtree hanging above $u_{I^{(j)}}$ at the time where it contains exactly $k$ vertices (translated to the origin in order to be considered as a plane tree).

Let us state the following intermediate result, which we will prove at the end of the section. Note that the random sequences $(N_n^{(j)})_{n\geq 1}$, $(\tau^{(j)}_k)_{k\geq 1}$ and 
	$(w^{(j)}_k)_{k\geq 1}$ for $j\in\{1,2\}$ can be read from $\left(\ind{u_n\in T (u_{I^{(j)}}) }\right)_{n\geq 1}$ for $j\in\{1,2\}$.
\begin{lemma}\label{wrt:lem:autosimilarity wrt}
The following holds.
\begin{enumerate}
\item\label{wrt:it:number of nodes in subtree} For $j\in\{1,2\}$, we almost surely  have
$N_n^{(j)}\underset{n\rightarrow\infty}{\sim} \mu(T(u_{I^{(j)}}))\cdot n$.
\item\label{wrt:it:weight subsequences satisfy square gamma p}  For $j\in\{1,2\}$, the sequence $(w^{(j)}_k)_{k\geq 1}$ satisfies \eqref{wrt:assum: wrt alpha p} almost surely.
\item\label{wrt:it:wrt conditionally on weight sequences}  Conditionally on the two sequences $\left(\ind{u_n\in T(u_{I^{(1)}})}\right)_{n\geq 1}$ and $\left(\ind{u_n\in T(u_{I^{(2)}})}\right)_{n\geq 1}$, the sequences of trees $(\ttT^{(1)}_k)_{k\geq1}$ and $(\ttT^{(2)}_k)_{k\geq1}$ are independent and have respective distributions $\wrt((w^{(1)}_k)_{k\geq 1})$ and $\wrt((w^{(2)}_k)_{k\geq 1})$.
\end{enumerate}  
\end{lemma}

Recall the discussion before Lemma~\ref{wrt:lem:behaviour Cn}. 
For $j\in\{1,2\}$, let $J^{(j)}\geq 1$ be the smallest integer such that for all $k\geq J^{(j)}$ and for all $z\in\mathscr E'$ we have $1+(e^z-1)\frac{w^{(j)}_k}{W^{(j)}_k}\neq 0$. 
Then we can define for $k\geq J^{(j)}$,
\begin{align*}
M^{(j)}_k(z):=\frac{1}{C_k^{(j)}(z)}\sum_{i=1}^{k}\frac{w_i^{(j)}}{W_k^{(j)}}e^{z\dist (u_{I^{(j)}},u_i^{(j)})} \quad \text{with} \quad C_k^{(j)}(z):=\prod_{i=J^{(j)}}^{k}\left(1+(e^z-1)\frac{w^{(j)}_i}{W^{(j)}_i}\right).
\end{align*}
These processes are the martingales associated to the weighted profile of the trees $(\ttT_k^{(j)})_{k\geq 1}$ for $j\in\{1,2\}$.
Thanks to Lemma~\ref{wrt:lem:autosimilarity wrt}\ref{wrt:it:wrt conditionally on weight sequences} those trees have respective distribution $\wrt((w^{(j)}_k)_{k\geq 1})$, for $j\in\{1,2\}$ and thanks to Lemma~\ref{wrt:lem:autosimilarity wrt}\ref{wrt:it:weight subsequences satisfy square gamma p}, those weight sequences satisfy \eqref{wrt:assum: wrt alpha p} almost surely. 
This allows us to apply Proposition~\ref{wrt:prop:convergence Mnbeta}, which entails that for $j\in\{1,2\}$, the sequence of functions $(z\mapsto M^{(j)}_k(z))_{k\geq J^{(j)}}$ converges almost surely to an analytic limit $z\mapsto M_\infty^{(j)}(z)$ on the set $\mathscr V$.
Now we can write, for $n$ sufficiently large
\begin{align}\label{wrt:eq:minoration Mn par Mn1 Mn2}
M_n(z)&=\frac{1}{C_n(z)}\sum_{i=1}^{n}\frac{w_i}{W_n}e^{z \haut(u_i)} \notag\\
&\geq \frac{C_{N_n^{(1)}}^{(1)}(z) \cdot W_{N_n^{(1)}}^{(1)}}{C_n(z)\cdot W_n} \cdot e^{z\haut(u_{I^{(1)}})} \cdot M_{N_n^{(1)}}^{(1)}(z)+ \frac{C_{N_n^{(2)}}^{(2)}(z) \cdot W_{N_n^{(2)}}^{(2)}}{C_n(z)\cdot W_n}\cdot e^{z\haut(u_{I^{(2)}})}\cdot  M_{N_n^{(2)}}^{(2)}(z).
\end{align}
Using Lemma~\ref{wrt:lem:behaviour Cn}, we have almost surely for $j\in\{1,2\}$, 	\[C_k^{(j)}(z)=\exp(\phi(z)\log k +c^{(j)}(z)+\petitodom{\mathscr E}{1}).\]
Using the asymptotics $N_n^{(j)}\underset{n\rightarrow\infty}{=}\mu(T(u_{I^{(j)}}))\cdot n\cdot (1+\petito{1})$ from Lemma~\ref{wrt:lem:autosimilarity wrt}\ref{wrt:it:number of nodes in subtree} we get
\begin{align*}
C_{N^{(j)}_n}^{(j)}(z)
&=\exp\left(\phi(z)\log n+\phi(z)\log (\mu(T(u_{I^{(j)}})))+c^{(j)}(z)+\petitodom{\mathscr E}{1}\right).
\end{align*}
From the a.s.\ convergence of the sequence of measures $(\mu_n)_{n\geq 1}$, see Theorem~\ref{wrt:thm:weak convergence of measures}, we also get \[\frac{W_{N_n^{(j)}}^{(j)}}{W_n}=\mu_n(T(u_{I^{(j)}}))\underset{n\rightarrow\infty}{\rightarrow}\mu(T(u_{I^{(j)}}))>0,\]
which entails that for $j\in\{1,2\}$, uniformly on all compact subsets of $\mathscr E$, we have the a.s.\ convergence
\[\frac{W_{N^{(j)}_n}^{(j)}}{W_n}\cdot\frac{C_{N^{(j)}_n}^{(j)}(z)}{C_n(z)}\underset{n\rightarrow\infty}{\rightarrow}A_j(z):=\mu(T(u_{I^{(j)}}))\cdot \exp\left(\phi(z)\log (\mu(T(u_{I^{(j)}})))+c^{(j)}(z)-c(z)\right) ,\] where the limiting function $z\mapsto A_j(z)$ is analytic and only takes positive values on $\mathscr E \cap \R$. 
Then, for any $z\in \mathscr V\cap \R$, taking the limit $n\rightarrow\infty$ in \eqref{wrt:eq:minoration Mn par Mn1 Mn2} yields
\begin{align*}
	M_\infty(z)\geq e^{z \haut(u_{I^{(1)}})}\cdot A_1(z)\cdot M^{(1)}_\infty(z) + e^{z\haut(u_{I^{(2)}})}\cdot A_2(z) \cdot M^{(2)}_\infty(z).
\end{align*}
Now, thanks to Lemma~\ref{wrt:lem:Minfty almost surely not zero}, the function $z\mapsto M^{(1)}_\infty(z)$ can only have at most countably many zeros on $I_\gamma\subset \mathscr V\cap \R$ and for all $z\in I_\gamma$, we have $M^{(2)}_\infty(z)>0$ almost surely.
Then if we condition on the location of the zeros $z_1,z_2\dots$ of $M^{(1)}_\infty$ on $I_\gamma$, since $M^{(2)}_\infty$ is independent of $z_1,z_2\dots$, we have $M^{(2)}_\infty(z_i)>0$ for all $i\geq 1$ almost surely. 
Hence $M_\infty$ has almost surely no zeros on $I_\gamma$.
\end{proof}

Now let us prove Lemma~\ref{wrt:lem:autosimilarity wrt} which we used in the preceding proof.
\begin{proof}[Proof of Lemma~\ref{wrt:lem:autosimilarity wrt}]
Point \ref{wrt:it:number of nodes in subtree} follows just from Theorem~\ref{wrt:thm:weak convergence of measures} and Proposition~\ref{wrt:prop:convergence other measures} and the fact that for $j\in\{1,2\}$ we have $N^{(j)}_n=n\nu_n(T(u_{I^{(j)}}))$.

Let us prove \ref{wrt:it:weight subsequences satisfy square gamma p}. In order to do that we are going to prove that for $j\in\{1,2\}$, we have
\begin{equation}\label{wrt:eq:measures converge nicely}
\mu_n(T(u_{I^{(j)}}))\underset{n\rightarrow\infty}{\bowtie}\mu(T(u_{I^{(j)}}))\qquad \text{and} \qquad \tau_n^{(j)}\underset{n\rightarrow\infty}{\bowtie} \mu(T(u_{I^{(j)}}))^{-1} \cdot n.
\end{equation}
Let us conclude from here: using the fact that $\boldsymbol{w}$ satisfies \eqref{wrt:assum: wrt alpha}, we get 
\begin{align*}
W_n^{(j)}=W_{\tau_n^{(j)}}\cdot \mu_{\tau_n^{(j)}}(T(u_{I^{(j)}})) \underset{n\rightarrow\infty}{\bowtie} \cst \cdot (\tau_n^{(j)})^\gamma \cdot \mu(T(u_{I^{(j)}})) \underset{n\rightarrow\infty}{\bowtie} \cst \cdot n^{\gamma},
\end{align*}
with a positive constant. We also have
\begin{align*}
\sum_{k=n}^{2n}(w_k^{(j)})^p=\sum_{k=n}^{2n}(w_{\tau_k^{(j)}})^p\leq \sum_{i=\tau_n^{(j)}}^{\tau_{2n}^{(j)}}w_i^p.
\end{align*}
Because of \eqref{wrt:eq:measures converge nicely}, we have the following almost sure convergence $\frac{\tau_{2n}^{(j)}}{\tau_{n}^{(j)}}\rightarrow 2$ as $n\rightarrow\infty$, hence almost surely for $n$ large enough we have $\tau_{2n}^{(j)}\leq 4 \tau_{n}^{(j)}$, so
\begin{align*}
	\sum_{i=\tau_n^{(j)}}^{\tau_{2n}^{(j)}}w_i^p \leq \sum_{i=\tau_n^{(j)}}^{4\tau_n^{(j)}}w_i^p \leq \sum_{i=\tau_n^{(j)}}^{2\tau_n^{(j)}}w_i^p + \sum_{i=2\tau_n^{(j)}}^{4\tau_n^{(j)}}w_i^p &\leq (\tau_n^{(j)})^{1+(\gamma-1) p+\petito{1}} +(2\tau_n^{(j)})^{1+(\gamma-1) p+\petito{1}} \\
	&\leq n^{1+(\gamma-1) p+\petito{1}},
\end{align*} 
where in the two last inequalities we used the fact that $\boldsymbol{w}$ satisfies \eqref{wrt:assum: wrt alpha p} and the almost sure linear growth of $\tau_n^{(j)}$ ensured by \eqref{wrt:eq:measures converge nicely}.

So it remains only to prove \eqref{wrt:eq:measures converge nicely}. 
Recall the proof of Theorem~\ref{wrt:thm:weak convergence of measures}. 
For all $k\geq 1$ the process $(\mu_n(T(u_k)))_{n\geq k}$ is a martingale and almost surely we have \[\abs{\mu_{n+1}(T(u_k))-\mu_n(T(u_k))}=\frac{w_{n+1}}{W_{n+1}}\cdot \abs{\ind{u_{n+1}\in T(u_k)}- \mu_n(T(u_k))}\leq \frac{w_{n+1}}{W_{n+1}}.\] 
Using successively Lemma~\ref{wrt:lem:biggins lemma} and then Lemma~\ref{wrt:fact:app of holder inequality}, which applies because $\boldsymbol{w}$ satisfies \eqref{wrt:assum: wrt alpha p},
\begin{align*}
\Ec{\abs{\mu_{2n}(T(u_k))-\mu_n(T(u_k))}^p}\leq 2^p\cdot \sum_{i=n+1}^{2n}\left(\frac{w_i}{W_i}\right)^p = \grandO{n^{1-p+\petito{1}}}.
\end{align*}
Using then Lemma~\ref{wrt:lem:convergence analytic martingale} with $q=p$ and $\alpha=0$ and $\delta=(p-1)/2$, we get that $\abs{\mu_n(T(u_k))-\mu(T(u_k))}=\grandO{n^{-\epsilon}}$ almost surely for some $\epsilon>0$. Since this is true almost surely for all $k\geq 1$, we use it with $k\in\{I^{(1)},I^{(2)}\}$. As by definition for $j\in\{1,2\}$ we have $\mu(T(u_{I^{(j)}}))>0$, we conclude that $\mu_n(T(u_{I^{(j)}}))\underset{n\rightarrow\infty}{\bowtie}\mu(T(u_{I^{(j)}}))$.

Then, for any $k\geq 1$, consider the process $\left(n\nu_n(T(u_k))-\sum_{i=k+1}^n\mu_i(T(u_k))\right)_{n\geq k}$. It is easy to verify that this process is a martingale in its own filtration and that its increments are bounded by $1$. Using again Lemma~\ref{wrt:lem:convergence analytic martingale} with $q=2$ and $\alpha=1$ and $\delta=1$, we get that
$n^{-1}\abs{n\nu_n(T(u_k))-\sum_{i=k+1}^n\mu_i(T(u_k))}=\grandO{n^{-\epsilon}}$ for some $\epsilon>0$. 
Using again that for $j\in\{1,2\}$ the limit $\mu(T(u_{I^{(j)}}))$ is almost surely positive, we can write 
$N_n^{(j)}=n\nu_n(T(u_{I^{(j)}}))\underset{n\rightarrow\infty}{\bowtie} \mu(T(u_{I^{(j)}}))\cdot n$. 
Using the definition of $\tau_n^{(j)}$, we can check that this entails that $\tau_n^{(j)}\underset{n\rightarrow\infty}{\bowtie}\mu(T(u_{I^{(j)}}))^{-1}\cdot n$ almost surely. This concludes the proof of \eqref{wrt:eq:measures converge nicely} and so, \ref{wrt:it:weight subsequences satisfy square gamma p} is proved.

Let us now prove \ref{wrt:it:wrt conditionally on weight sequences}. 
For any $k\geq1$, we consider the sequence $\left(\ind{u_n\in T(u_k)}\right)_{n\geq 1}$ that encodes the labels of the vertices above $u_k$. 
Note that the limiting mass $\mu(T(u_k))$ can be computed from that sequence.  
Now, let us sequentially reveal $\left(\ind{u_n\in T(u_1)}\right)_{n\geq 1},\left(\ind{u_n\in T(u_2)}\right)_{n\geq 1},\dots,\left(\ind{u_n\in T(u_k)}\right)_{n\geq 1},\dots$ until we get to a $k$ for which $\mu(T(u_k))\in \intervalleoo{0}{1}$. By definition, the first index for which it happens is $I^{(1)}$.

Then we continue revealing the sequences $\left(\ind{u_n\in T(u_k)}\right)_{n\geq 1}$ for $k>I^{(1)}$ but only for the $k$'s such that $u_k\notin T(u_{I^{(1)}})$ until we get to a $k$ for which $\mu(T(u_k))\in \intervalleoo{0}{1}$. By definition, this second index is $I^{(2)}$. 
Remark, and this is the key in this argument, that after determining $I^{(1)}$ and $I^{(2)}$ in this way, the only information that we have about $T(u_{I^{(1)}})$ and $T(u_{I^{(2)}})$ is the list of labels of the vertices that belong each of them (and the position of $u_{I^{(1)}}$ and $u_{I^{(2)}}$). 

Now, conditionally on all this information, it is straightforward to see from the attachment dynamics that for any $j\in\{1,2\}$, when the $i+1$-st vertex attaches above $u_{I^{(j)}}$ at time $\tau_{i+1}^{(j)}$, the label $K_{\tau_{i+1}^{(j)}}$ of the vertex to which it attaches is chosen among $\tau_1^{(j)},\tau_2^{(j)},\dots \tau_i^{(j)}$ with probability proportional to their respective weight $w_{\tau_1^{(j)}},w_{\tau_2^{(j)}},\dots w_{\tau_i^{(j)}}$, independently for different choices of $i\geq 1$ and $j\in\{1,2\}$. This finishes the proof of \ref{wrt:it:wrt conditionally on weight sequences} and hence that of the lemma.
\end{proof}

\subsection{From the weighted to the unweighted sum.}\label{wrt:subsec:from the weighted to unweighted sum}
Now we want to transfer these results of convergence to the Laplace transform of the real profile.  
Recall from \eqref{wrt:eq:def Nnz} the definition of the sequence of functions $(z\mapsto N_n(z))_{n\geq 1}$. We still assume until the end of Section~\ref{wrt:subsec:from the weighted to unweighted sum} that $\boldsymbol{w}$ satisfies \eqref{wrt:assum: wrt alpha p} for some $\gamma>0$ and $p\in \intervalleof{1}{2}$. 

We introduce the following quantity, for $n\geq J$, 
\begin{align*}
	X_n(z)&:=n^{1+\phi(z)}\cdot N_n(z)-e^z \sum_{k=J}^{n-1}C_k(z)M_k(z)\\
	&=\sum_{i=1}^{n}e^{z \haut(u_i)}-e^z \sum_{k=J}^{n-1}\left(\sum_{i=1}^{k}\frac{w_i}{W_k}e^{z\haut(u_i)}\right)
\end{align*}
The goal of this subsection is to show that the quantity $X_n(z)$ is negligible as $n\rightarrow\infty$ compared to any of the two terms in the difference, for $z$ contained in some subset of the complex plane. 
This way we will transfer the asymptotics that we have proved for $M_n(z)$ and $C_n(z)$ in the last section to asymptotics for $N_n(z)$, which is the quantity that we want to study in the end.
Let us start by proving a lemma.

\begin{lemma}\label{wrt:lem:Xn martingale}
	The process $\left(X_n(z)\right)_{n\geq J}$ is a martingale with respect to $(\cF_n)_{n\geq 1}$. Furthermore, for all $q\in\intervalleof{1}{p}$, \[\Ec{\abs{X_{2n}(z)-X_n(z)}^q}=\grandOdom{\mathscr E}{n^{1+(q\Re(\phi(z)) \vee \phi(q\Re z))+\petitodom{\mathscr E}{1}} }.\]
\end{lemma}
\begin{proof}
This process is of course $(\cF_n)$-adapted and integrable. For the martingale property we compute
\begin{align*}
	\Ecsq{X_{n+1}(z)}{\cF_n}&=\Ecsq{X_n(z)-e^z C_n(z)M_n(z)+e^{z\haut(u_{n+1})}}{\cF_n}\\
	&=X_n(z)-e^z C_n(z)M_n(z)+e^z\sum_{i=1}^{n}\frac{w_i}{W_n}e^{z\haut(u_i)}=X_n(z).
\end{align*}
For $z\in \mathscr E$ and $q\in \intervalleof{1}{p}$, we make the following computation, using Lemma~\ref{wrt:lem:behaviour Cn} and Lemma~\ref{wrt:lem:moments Mn},
\begin{align*}
	\Ec{\abs{X_{n+1}(z)-X_{n}(z)}^q}&=\Ec{\abs{-e^z C_n(z)M_n(z)+e^{z\haut(u_{n+1})}}^q}\\
	&\leq 2^q\cdot \left(e^{q\Re z}\abs{C_n(z)}^q \Ec{\abs{M_n(z)}^q}+e^{q\Re z} \Ec{\sum_{i=1}^{n}\frac{w_i}{W_n}e^{\haut(u_i)q\Re z}}\right),\\
	&= \grandOdom{\mathscr E}{n^{q\Re \phi(z)+0\vee g(z,q)+\petitodom{\mathscr E}{1}}}+\grandOdom{\mathscr E}{n^{\phi(q\Re z)}}\\
	&=\grandOdom{\mathscr E}{n^{q\Re \phi(z)\vee (q\Re \phi(z)+g(z,q))\vee \phi(q\Re z)+\petitodom{\mathscr E}{1} }}
	\end{align*}
	and the last exponent reduces to $q\Re \phi(z)\vee \phi(q\Re z)$ because $(q\Re \phi(z)+g(z,q))= \phi(q\Re z)+1-q<\phi(q\Re z)$. 
	Hence, using Lemma~\ref{wrt:lem:biggins lemma}, we get
	\begin{align*}
	\Ec{\abs{X_{2n}(z)-X_n(z)}^q}&\leq 2^q\sum_{i=n}^{2n}\Ec{\abs{X_{i+1}(z)-X_{i}(z)}^q}=\grandOdom{\mathscr E}{n^{1+(q\Re(\phi(z)) \vee \phi(q\Re z))+\petitodom{\mathscr E}{1}}},
	\end{align*}
which finishes the proof of the lemma.
\end{proof}
Recall the definition of $z_+$ and $z_-$ in \eqref{wrt:eq:definition z+ and z-}. Let us define the domain $\mathscr D$ to which we refer in the statement of Proposition~\ref{wrt:prop:assumptions kabluchko} as the connected component of
\begin{align*}
\mathscr V\cap \enstq{z \in \mathbb C}{1+\Re(\phi (z))>0}
\end{align*} 
that contains $0$, where $\mathscr{V}$ is defined in \eqref{wrt:eq:def domain V}.
In this way, $\mathscr D$ is a domain of $\mathbb C$ and $\mathscr D \cap \R=\intervalleoo{z_{-}}{z_{+}}$. 
Indeed, first, $\mathscr D$ is open and connected by definition. Then recall from Lemma~\ref{wrt:lem:V open domain and contains open interval} that $\mathscr V \cap \R$ contains $I_\gamma=\enstq{x\in \R}{1+\gamma(e^x-1-xe^x)>0}$ an open interval which contains $0$ and has $z_+$ as its right endpoint. 
Now just check that $\enstq{z \in \mathbb R}{1+\Re(\phi (z))>0}=\intervalleoo{z_-}{\infty}$ and that $z_-\in I_\gamma$.

For technical reasons, we also introduce the following subset of $\mathbb{C}$, here identified as $\R\times \R$,
\begin{align*}
\mathscr D':= \intervalleoo{z_{-}}{z_{+}}\times\intervalleoo{0}{2\pi},
\end{align*} 
on which the process $(z\mapsto M_n(z))_{n\geq J}$, and hence also $(z\mapsto X_n(z))_{n\geq J}$, are well-defined. 
Let us further decompose $\mathscr D'$ into a union of open sets
\begin{align*}
\mathscr D' = \bigcup_{1<q\leq p} \mathscr D'_q \qquad \text{where} \qquad \mathscr D'_q=\enstq{z\in \mathscr D'}{g(\Re z ,q)<0}.
\end{align*}
\begin{lemma}\label{wrt:lem:convergences linked to Xn}The following holds.
\begin{enumerate}
	\item\label{wrt:it:convergence Xn 1} For all compact $K\subset \mathscr D$ there exists $\epsilon(K)>0$ such that almost surely 
	\begin{align*}
	n^{-(1+\Re \phi(z))}\cdot\abs{X_n(z)}=\grandOdom{K}{n^{-\epsilon(K)}}.
	\end{align*}
	\item \label{wrt:it:convergence Xn 2} For all compact $K\subset \mathscr D'$, there exists $\epsilon(K)>0$ such that
	\begin{align*}
	n^{-(1+\phi(\Re z))}\cdot \abs{\sum_{i=J}^{n-1}C_i(z)M_i(z)}=\grandOdom{K}{n^{-\epsilon(K)}}
	\end{align*}
	\item \label{wrt:it:convergence Xn 3} For all compact $K\subset \mathscr D'$, there exists $\epsilon(K)>0$ such that almost surely
	\begin{align*}
	n^{-(1+\phi(\Re z))}\cdot\abs{X_n(z)}=\grandOdom{K}{n^{-\epsilon(K)}}.
	\end{align*}
\end{enumerate}
\end{lemma}
\begin{proof}
For the first one, for any $q\in\intervalleof{1}{p}$ we can apply Lemma~\ref{wrt:lem:convergence analytic martingale} on the open set $\mathscr V_q\cap \enstq{z\in \mathbb C}{1+\Re(\phi(z))>0}$ with $\alpha(z)=1+\Re(\phi(z))>0$ and $\delta(z)=\min(q-1, -g(z,q))>0$, thanks to Lemma~\ref{wrt:lem:Xn martingale}. 
Then using the compactness of $K$, \ref{wrt:it:convergence Xn 1} is true for every compact $K\subset \mathscr V \cap \enstq{z\in \mathbb C}{1+\Re(\phi(z))>0}$, hence for any compact $K\subset \mathscr D$. 

Let us prove point \ref{wrt:it:convergence Xn 2}. For any $q\in\intervalleof{1}{p}$, thanks to Lemma~\ref{wrt:lem:moments Mn}, on the open set $\mathscr D'_q\subset \mathscr E$ we have $\Ec{\abs{M_{2n}(z)-M_{n}(z)}^q}=\grandOdom{\mathscr D'_q}{n^{(1-q)\vee g(z,q)+\petitodom{\mathscr D'_q}{1}}}$ and 
\[g(z,q)=q\underset{>0}{\underbrace{(\phi(\Re z)-\Re \phi(z))}}+\underset{<0}{\underbrace{g(\Re z,q)}}.\]
Applying Lemma~\ref{wrt:lem:convergence analytic martingale} for the martingale $(z\mapsto M_n(z))_{n\geq J}$ on any compact subset $K\subset \mathscr D'_q$ with $\alpha(z)=\phi(\Re z)-\Re \phi(z)>0$ and $\delta(z)=\min(-1+q+q(\phi(\Re z)-\Re(\phi(z))), -g(\Re z,q))>0$ yields:
\begin{align*}
n^{-\phi(\Re z)+\Re \phi(z)}\cdot M_n(z)=\grandOdom{K}{n^{-\epsilon(K)}}.
\end{align*}
Using the estimates of Lemma~\ref{wrt:lem:behaviour Cn}, we have $\abs{C_n(z)}=\grandOdom{K}{n^{\Re\phi(z)}}$, and so $\abs{C_i(z) M_i(z)}=\grandOdom{K}{i^{\phi(\Re z)-\epsilon(K)}}$. Hence $\sum_{i=J}^{n-1}\abs{C_i(z)M_i(z)}=\grandOdom{K}{n^{0\vee (1+\phi(\Re z)-\epsilon(K))}}$ which finishes the proof of \ref{wrt:it:convergence Xn 2}.

Last, in order to prove \ref{wrt:it:convergence Xn 3}, we use Lemma~\ref{wrt:lem:convergence analytic martingale} on $\mathscr D'_q$ for the martingale $(z\mapsto X_n(z))_{n\geq J}$ with $\alpha(z)=1+\phi(\Re z)>0$ and $\delta(z)=\min(-1+q+q(\phi(\Re z)-\Re(\phi(z))), -g(\Re z,q))$. 
\end{proof}
In order to conclude, we will also need the following lemma, which is a direct consequence of Lemma~\ref{wrt:lem:behaviour Cn}.
\begin{lemma}\label{wrt:lem:sum Ciz}
For any compact $K\subset \mathscr E\cap \enstq{z\in \mathbb C}{1+\Re(\phi(z))>0}$, there exists $\epsilon(K)$ such that 
\begin{align*}
\abs{n^{-(1+\phi(z))}\cdot \sum_{i=J}^{n-1}C_i(z)-\frac{e^{c(z)}}{1+\phi(z)}}=\grandOdom{K}{n^{-\epsilon(K)}}
\end{align*}
\begin{proof}
	On any compact $K\subset \mathscr E\cap \enstq{z\in \mathbb C}{1+\Re(\phi(z))>0}$, using Lemma~\ref{wrt:lem:behaviour Cn} we write \[C_n(z)=e^{c(z)}\cdot n^{\phi(z)}\cdot (1+\grandOdom{K}{n^{-\epsilon}}),\] 
	so that
	\begin{align*}
	\sum_{i=J}^{n-1}C_i(z)&= e^{c(z)} \cdot \sum_{i=J}^{n-1}i^{\phi(z)}+ e^{c(z)} \cdot \sum_{i=J}^{n-1}i^{\phi(z)}\cdot \grandOdom{K}{i^{-\epsilon}}\\
	&=\frac{e^{c(z)}n^{1+\phi(z)}}{1+\phi(z)}\cdot(1+\grandOdom{K}{n^{-1}}) + \grandOdom{K}{n^{1+\phi(z)-\epsilon(K)}},
	\end{align*}
	where in the second line, we use the fact that $\inf_{z\in K}(1+\Re\phi(z))>0$, and we define $\epsilon(K):=\epsilon\wedge \inf_{z\in K}(1+\Re\phi(z))$. This proves the lemma.
\end{proof}
\end{lemma}

We can now prove Proposition~\ref{wrt:prop:assumptions kabluchko}. 
\begin{proof}[Proof of Proposition~\ref{wrt:prop:assumptions kabluchko}]
Let us start by proving simultaneously that for any $z\in\mathscr D$, we almost surely have
\begin{align}\label{wrt:eq:Ninfty as Minfty}
N_\infty(z):=\lim_{n\rightarrow\infty}N_n(z)=\frac{e^{z+c(z)}}{1+\phi(z)}\cdot M_\infty(z),
\end{align}
and also that both points \ref{wrt:it:convergence Nn sur tout compact} and \ref{wrt:it:convergence Nn polynomiale} of the proposition hold. For any compact $K\subset \mathscr D$  and $z\in K$, we write
\begin{align*}
\abs{N_n(z)-\frac{e^{z+c(z)}}{1+\phi(z)}M_\infty(z)}
	&\leq \abs{n^{-(1+\phi(z))}X_n(z)}+\abs{n^{-(1+\phi(z))} e^z\sum_{i=J}^n C_i(z)M_i(z)-\frac{e^{z+c(z)}}{1+\phi(z)}M_\infty(z)}
\end{align*}
The first term is $\grandOdom{K}{n^{-\epsilon(K)}}$ thanks to Lemma~\ref{wrt:lem:convergences linked to Xn}\ref{wrt:it:convergence Xn 1}. We bound the second one by the following quantity
\begin{align*}
	\abs{M_\infty(z)}\cdot \abs{e^z}\cdot \underset{\grandOdom{K}{n^{-\epsilon(K)}}}{ \underbrace{\abs{n^{-(1+\phi(z))}\cdot \sum_{i=J}^{n}C_i(z)-\frac{e^{c(z)}}{(1+\phi(z))}}}}
	+n^{-(1+\Re \phi(z))}\cdot \abs{e^z}\cdot \sum_{i=J}^{n-1}\underset{\grandOdom{K}{i^{\Re \phi(z)-\epsilon(K)}}}{\underbrace{\abs{C_i(z)} \cdot \abs{M_i(z)-M_\infty(z)}}}.
\end{align*} 
In the above display, we used Lemma~\ref{wrt:lem:sum Ciz} and then Lemma~\ref{wrt:lem:behaviour Cn} together with Proposition~\ref{wrt:prop:convergence Mnbeta} on respectively the first and the second term. In the end, the whole expression is $\grandOdom{K}{n^{-\epsilon(K)}}$. 
From \eqref{wrt:eq:Ninfty as Minfty}, it is clear that the limiting function $z\mapsto N_\infty(z)$ is analytic and has almost surely no zero on $\intervalleoo{z_{-}}{z_{+}}$ because of Lemma~\ref{wrt:lem:Minfty almost surely no zero}. 
For \ref{wrt:it:convergence Im>0}, let us prove the stronger statement: for any compact subset $K \subset \intervalleoo{z_{-}}{z_{+}}$ and $0<a<\pi$, there exists $\epsilon(K,a)>0$ such that almost surely, 
\begin{align*}
\sup_{x\in K}\sup_{a\leq \eta \leq \pi}n^{-(1+\phi(x))}\abs{\sum_{i=1}^n e^{(x+i\eta)\haut(u_i)}}=\grandO{n^{-\epsilon(K,a
		)}}.
\end{align*}
For this, we write \[n^{-(1+\phi(x))}\abs{\sum_{i=1}^n e^{(x+i\eta)\haut(u_i)}}\leq n^{-(1+\phi(x))}\abs{X_n(x+i\eta)}+ n^{-(1+\phi(x))}\abs{\sum_{i=J}^{n-1}C_i(x+i\eta)M_i(x+i\eta)}.\]
We apply points \ref{wrt:it:convergence Xn 2} and \ref{wrt:it:convergence Xn 3} of Lemma~\ref{wrt:lem:convergences linked to Xn} to the compact $K\times \intervalleff{a}{\pi}$ and get the desired bound.
\end{proof}

\subsection{Height of the tree}\label{wrt:subsec:upper-bounds on height}
In this section, we study the behaviour of the height $\haut(\ttT_n)$ of the tree $\ttT_n$, which is defined as the maximal height of the vertices of $\ttT_n$, i.e.
\begin{align*}
\haut(\ttT_n)=\max_{1\leq k\leq n}\haut(u_k).
\end{align*}
We start by showing that under the assumption \eqref{wrt:assum: wrt alpha p} we have the convergence \eqref{wrt:eq:convergence height gamma p}. Then, for the sake of completeness, we also study the simpler case where $\log n = \petito{\sum_{i=1}^{n}\frac{w_i}{W_i}}$. 

One key argument in our proofs is the following equality for the \emph{annealed} moment generating function of the height of $u_k$, for any fixed $k\geq 2$, which can be seen as a corollary of Lemma~\ref{wrt:lem:Mn martingale}
\begin{equation}\label{wrt:eq:distribution of the height of unif vertex}
\Ec{e^{z\haut(u_k)}}=e^z\cdot \prod_{j=2}^{k-1}\left(1 + (e^z-1)\frac{w_j}{W_j}\right).
\end{equation} 
Some elementary computations using the Chernoff bound and the last display yield the following lemma. 
\begin{lemma}\label{wrt:lem:upperbound height in log}
	Suppose that the sequence of weights $\boldsymbol{w}$ satisfies
	\begin{align*}
		\limsup_{n\rightarrow\infty} \frac{1}{\log n}\sum_{i=2}^{n}\frac{w_i}{W_i} \leq u \in \R_+^*.
	\end{align*}
Then almost surely we have
\begin{align*}
	\limsup_{n\rightarrow\infty} \frac{\haut(\ttT_n)}{\log n}\leq u e^{z_{+}(u)},
\end{align*}
where $z_+(u)$ is the unique positive root of $u (z e^z - e^z +1) -1=0$.  
\end{lemma}
\begin{proof}
Using the expression \eqref{wrt:eq:distribution of the height of unif vertex} for the moment generating function of $\haut(u_n)$ we get, for any $z>0$
\begin{align*}
\Ec{e^{z \haut(u_n)}}=e^z \cdot \prod_{j=2}^{n-1}\left(1 + (e^z-1)\frac{w_j}{W_j}\right)&\leq \exp\left(1+(e^z-1)\sum_{j=2}^{n-1}\frac{w_j}{W_j}\right)\\
&\leq \exp\left((\log n)\cdot(u(e^z-1)+\petito{1})\right),
\end{align*}
where we use the inequality $(1+x)\leq e^x$ and the assumption on $\boldsymbol{w}$.
Then, for  any $z>0$ and $n\geq 1$,
\begin{align*}
\Pp{\haut(u_n)\geq u e^z\log n}\leq e^{-uz e^z \log n}\Ec{e^{z \haut(u_n)}} \leq \exp\left(-u\log n (z e^z -e^z +1+\petito{1})\right)
\end{align*}
If we take $z>0$ such that $u (z e^z - e^z+1)>1$ then the right-hand-side is summable and hence using the Borel-Cantelli lemma shows that for all $n$ large enough, we have $\haut(u_n)\leq ue^z \log n$. Letting $z\searrow z^+(u)$, we get the result.
\end{proof}

Let us prove the last claim \eqref{wrt:eq:convergence height gamma p} of Theorem~\ref{wrt:thm:profile and height of wrt}. Here we suppose that the weight sequence $\boldsymbol{w}$ satisfies \eqref{wrt:assum: wrt alpha p} for some $\gamma>0$ and some $p\in\intervalleof{1}{2}$.
\begin{proof}[Proof of \eqref{wrt:eq:convergence height gamma p}]
	Recall the asymptotics \eqref{wrt:eq:radish asymptotics} in Theorem~\ref{wrt:thm:profile and height of wrt}. It ensures that there almost surely exist vertices at height $\lfloor\gamma e^z\log n\rfloor$, for any fixed $z\in\intervalleoo{z_{-}}{z_{+}}$ and $n$ large enough. Hence the height of the tree $\ttT_n$ satisfies
	\begin{align*}
	\liminf_{n\rightarrow\infty}\frac{\haut(\ttT_n)}{\log n}\geq \gamma e^{z_{+}}.
	\end{align*}
	For the limsup, we use Lemma~\ref{wrt:lem:upperbound height in log} with $u=\gamma$ (this is justified by Lemma~\ref{wrt:lem:behaviour sum w/W}), which yields $\limsup_{n\rightarrow\infty}\frac{\haut(\ttT_n)}{\log n}\leq \gamma e^{z_{+}}$. 
\end{proof}
To finish the section, we state a proposition.
\begin{proposition}\label{wrt:prop:upper-bound height}
Let $f(n):= \sum_{i=2}^{n-1}\frac{w_i}{W_i}$. If $\log n = \petito{f(n)}$
then we have the almost sure convergences
\begin{align*}
\lim_{n\rightarrow\infty} \frac{\haut(\ttT_n)}{f(n)}=\lim_{n\rightarrow\infty} \frac{\haut(u_n)}{f(n)}=1.
\end{align*}
\end{proposition}
\begin{proof}
	As we can check from its moment generating function \eqref{wrt:eq:distribution of the height of unif vertex}, the random variable $\haut(u_n)-1$ is a sum of independent Bernoulli random variables, with expectation $f(n)$. Using standard bounds for $\haut(u_n)-1$ yields
\begin{align*}
\Pp{\abs{\haut(u_n)-1-f(n)}\geq \epsilon f(n)}\leq 2 \exp\left(-\epsilon^2f(n)/3\right),
\end{align*}
which is summable in $n$ for any $\epsilon>0$. The result of the proposition is then obtained using the Borel-Cantelli lemma.
\end{proof}

\section{Preferential attachment trees are weighted recursive trees}\label{wrt:sec:exchangeability}
In this section, we study preferential attachment trees with fitnesses $\mathbf a$ as defined in the introduction. First, in Section~\ref{wrt:subsec:coupling with a sequence of polya urns}, we prove Theorem~\ref{wrt:thm:connection PA WRT} which allows us to see them as weighted random trees $\wrt(\boldsymbol{\mathsf w}^\mathbf a)$ for some random weight sequence $\boldsymbol{\mathsf w}^\mathbf a$. Then in Section~\ref{wrt:subsec:proof of proposition 2} we prove Proposition~\ref{wrt:prop:assum pa implies assum wrt} which relates the asymptotic behaviour of $\boldsymbol{\mathsf w}^\mathbf a$ to the behaviour of $\mathbf a$.
Finally, in Section~\ref{wrt:subsec:distribution of the limiting sequence} we prove Proposition~\ref{wrt:prop:Mk markov chain}, which ensures that the sequence $\boldsymbol{\mathsf{m}}^\mathbf a$ obtained as the scaling limit of the degrees can be expressed as the increments of a Markov chain.

\subsection{Coupling with a sequence of P\'olya urns: proof of Theorem~\ref{wrt:thm:connection PA WRT}}\label{wrt:subsec:coupling with a sequence of polya urns}
Here we fix an arbitrary sequence $\mathbf{a}$ such that $a_1>-1$ and $\forall n\geq 2,\ a_n\geq 0$. Let us recall the notation, for $n\geq 0$,
\begin{align*}
A_n:=\sum_{i=1}^{n}a_i,
\end{align*}
with the convention that $A_0=0$.
We consider a sequence of trees $(\ttP_n)_{n\geq 1}$ evolving according to the distribution $\pa(\mathbf{a})$ and we want to prove Theorem~\ref{wrt:thm:connection PA WRT}, namely that there exists a random sequence of weights $\boldsymbol{\mathsf{w}}^\mathbf a$ for which the sequence evolves as a $\wrt(\boldsymbol{\mathsf{w}}^\mathbf a)$.
The proof uses a decomposition of this process into an infinite number of P\'olya urns. This is very close to what is used in the proofs of \cite[Theorem~2.1]{berger_asymptotic_2014} or \cite[Section 1.2]{bloem-reddy_preferential_2017_online} in similar settings. The novelty of our approach is to express this result using weighted random trees, since it allows us to apply all the results developed in the preceding section.
\paragraph{P\'olya urns.}For us, a P\'olya urn process $(\mathsf{Urn}(n))_{n\geq 0}=(X(n),\mathrm{Total}(n))_{n\geq 0}$ is a Markov chain on $E:=\enstq{(x,z)\in\R_+\times \R_+^*}{x\leq z}$ with transition probabilities given by the matrix $P$ where for all $(x,z)\in E$,
\begin{equation}
P\left((x,z),(x+1,z+1)\right)=\frac{x}{z} \qquad \text{and} \qquad P\left((x,z),(x,z+1)\right)=\frac{z-x}{z}.
\end{equation}
The quantities $X(n)$ and $\mathrm{Total}(n)$ represent respectively the number of red balls and the total number of balls at time $n$ in a urn containing red and blacks balls, in which we add a ball at each time, the colour of which is chosen at random proportionally to the current proportion in the urn. 
Starting at time $0$ from the state $(a,a+b)$, i.e. with $a$ red balls and $b$ black balls, it is well-known that the sequence $(\Delta X(n))_{n\geq1}=(X(n)-X(n-1))_{n\geq1}$ of random variables is \emph{exchangeable}, and an application of de Finetti's representation theorem ensures that it has the same distribution as i.i.d.\ samples of Bernoulli random variables with a \emph{random} parameter $\beta$, which has distribution $\mathrm{Beta}(a,b)$, where we use the convention that $\mathrm{Beta}(a,b)=\delta_1$ if $b=0$.

Note that the process $(\mathsf{Urn}(n))_{n\geq 0}$ is entirely determined from $(\Delta X(n))_{n\geq1}$ and that the random variable $\beta$ is a measurable function of the sequence $(\mathsf{Urn}(n))_{n\geq 0}$ because it can almost surely be obtained as $\beta=\lim_{n\rightarrow\infty}\frac{X(n)}{n}$.

\paragraph{Nested structure of urns in the tree.}
For all $k\ge 1$ we define the following process in $n\geq k$
\[W_k(n):=A_k+\sum_{i=1}^{k}\deg^+_{\ttP_{n}}(u_i),\]
the "total fitness" of the vertices $\{u_1,u_2,\dots,u_k\}$, for which we remark that for any $k\geq 1$ we have
\begin{equation}
	W_k(k)=A_k+k-1 \qquad \text{and} \qquad W_{k}(k+1)=A_k+k. 
\end{equation}
Imagine that $\ttP_n$ is constructed and we add a new vertex $u_{n+1}$ to the tree. We choose its parent in a downward sequential way: 
\begin{itemize}
	\item we first determine whether the parent is $u_n$, this happens with probability \[\frac{a_n+\deg^+_{\ttP_{n}}(u_n)}{W_n(n)}=1-\frac{W_{n-1}(n)}{W_n(n)},\]
	\item then with the complementary probability $\frac{W_{n-1}(n)}{W_n(n)}$ it is not, so conditionally on this we determine whether it is $u_{n-1}$, this happens with (conditional) probability
	\[\frac{a_{n-1}+\deg^+_{\ttP_{n}}(u_{n-1})}{W_{n-1}(n)}=1-\frac{W_{n-2}(n)}{W_{n-1}(n)}.\] 
	\item then with the complementary probability $\frac{W_{n-2}(n)}{W_{n-1}(n)}$ it is not, etc... We continue this process until we stop at some $u_i$.
\end{itemize}
Now let us fix $k\geq 1$ and introduce the following time-change: for all $N\geq 0$, we let
\begin{equation}
	\theta_{k}(N):=\inf\enstq{n\geq k+1}{W_{k+1}(n)=A_{k+1}+k+N},
\end{equation}
be the $N$-th time that a vertex in attached on one of the vertices $\{u_1,\dots,u_{k+1}\}$ after time $k+1$, where by definition, we have $\theta_{k}(0)=k+1$. 
Remark that it can be the case that $\theta_{k}(N)$ is not defined for large $N$, if there is only a finite number of vertices attaching to $\{u_1,\dots,u_{k+1}\}$. Let us ignore this possible problem for the moment, and only consider sequences $\mathbf a$ for which $A_n=\grandO{n}$, for which this will almost surely not happen. In this case for all $N\geq 0$ we set
\begin{equation}
\mathsf{Urn}_{k}(N):=(W_{k}(\theta_{k}(N)),W_{k+1}(\theta_{k}(N)))=(W_{k}(\theta_{k}(N)),A_{k+1}+k+N).
\end{equation}
Now, the three following facts are the key observations in order to prove Theorem~\ref{wrt:thm:connection PA WRT}:
\begin{enumerate}
	\item\label{wrt:it:urns are polya urns} for all $k\geq 1$, the process $\mathsf{Urn}_{k}=\left(\mathsf{Urn}_{k}(N)\right)_{N\geq 0}$ has the distribution of a P\'olya urn starting from the state $(A_{k}+k,A_{k+1}+k)$,
	\item\label{wrt:it:urns are independent} those process are jointly independent for $k\geq 1$,
	\item \label{wrt:it:urns determine the process}the whole sequence $(\ttP_n)_{n\geq 1}$ is a function the collection of processes $\left(\mathsf{Urn}_{k},\ k\geq1\right)$.
\end{enumerate}
Point \ref{wrt:it:urns are polya urns} already follows from the discussion above. A moment of thought shows that \ref{wrt:it:urns are independent} holds as well: of course the processes $(W_k(n),W_{k+1}(n))_{n\geq k+1}$ for different $k$ are not independent at all but the point is that they only interact through the time-changes $(\theta_k(\cdot),k\geq 1)$. 
Last, for \ref{wrt:it:urns determine the process}, let us note that we can reconstruct the tree $\ttP_n$ at time $n$ from the random variables $(W_i(k))_{1\leq i,k \leq n}$ and that these random variables can be entirely determined using
\begin{equation*}
\left\lbrace
\begin{aligned}
	W_k(n)&=\mathsf{Urn}_{k}(W_{k+1}(n)-(A_{k+1}+k)), \quad \text{for } 1\leq k\leq n-1,\\
	W_n(n)&=A_n+n-1.
	\end{aligned}
\right.
\end{equation*}

\paragraph{Reversing the construction and using the exchangeability.}
Let us now reverse the construction and start with an independent family $(\mathsf{Urn}_{k},\ k\geq 1)$ of processes which have for each $k\geq 1$ the distribution of a P\'olya urn starting from the state $(A_{k}+k,A_{k+1}+k)$, so that they have the joint same distribution as the ones described in \ref{wrt:it:urns are polya urns} and \ref{wrt:it:urns are independent}.
From what we did above, the sequence $(\ttP_n)_{n\geq 1}$ that they determine through \ref{wrt:it:urns determine the process} has distribution $\pa(\mathbf{a})$. 
A moment of thought shows that this argument actually still holds for a completely arbitrary sequence of fitnesses $\mathbf{a}$. 

Now, using de Finetti's theorem, each of the processes $\mathsf{Urn}_k$ can be produced by sampling $\beta_k\sim \mathrm{Beta}(A_k+k,a_{k+1})$ and adding a red ball at each step independently with probability $\beta_k$ and a black ball with probability $1-\beta_k$. This is of course done independently for different $k\geq1$.

In terms of our downward sequential procedure defined above for finding the parent of each newcomer, it amounts to saying that each time that we have to choose between attaching to $u_{k+1}$ or attach to a vertex among $\{u_1,\dots,u_k\}$, the former is chosen with probability $1-\beta_k$ and the latter with probability $\beta_k$.
Let us verify that the law of $(\ttP_n)_{n\geq 1}$ conditionally on the sequence $(\beta_k)_{k\geq 1}$ can indeed be expressed as WRT with the random sequence of weights $\boldsymbol{\mathsf w}^\mathbf{a}$ defined in Theorem~\ref{wrt:thm:connection PA WRT}, which is defined from the sequence $(\beta_k)_{k\geq 1}$ as,
\begin{align*}
\forall n\geq 1, \qquad \mathsf W^\mathbf a_n=\prod_{i=1}^{n-1}\beta_i ^{-1} \qquad \text{and}  \qquad \mathsf w^\mathbf a_n= \mathsf W^\mathbf a_{n}- \mathsf W^\mathbf a_{n-1},
\end{align*}
with the convention that $\mathsf W^\mathbf a_1=1$ and $\mathsf W^\mathbf a_0=0$.
Let us reason conditionally on the sequence $(\beta_k)_{k\geq 1}$ (or equivalently the sequence $(\mathsf w^\mathbf a_n)_{n\geq 1}$). 
When determining the parent of $u_{n+1}$, whose label we denote $J_{n+1}$ as in \eqref{wrt:eq:def Jn},
we successively try to attach to $u_n,u_{n-1},\dots$ until we stop at $u_{J_{n+1}}$.
Using the independence, we get that for every $k\in\{1,2,\dots,n\}$,
\begin{align*}
\Ppsq{J_{n+1}=k}{\ttP_1,\dots,\ttP_n,(\beta_i)_{i\geq 1}}
=\beta_{n-1}\beta_{n-2}\dots \beta_k (1-\beta_{k-1})
=\frac{\mathsf W^\mathbf a_k-\mathsf W^\mathbf a_{k-1}}{\mathsf W^\mathbf a_n}
=\frac{\mathsf w^\mathbf a_k}{\mathsf W^\mathbf a_n}.
\end{align*}

This proves Theorem~\ref{wrt:thm:connection PA WRT}. 
Let us explain how Corollary~\ref{wrt:cor: pa conditional distribution limiting degrees} follows from the proof that we developed here. From the discussion in the previous paragraph, in the case of a sequence $\mathbf a$ for which $A_n=\grandO{n}$, each of the processes $(\mathsf{Urn}_k(N))_{N\geq 0}$ for $k\geq 1$ is a measurable function of $(\ttP_n)_{n\geq1}$, and hence the associated $\beta_k$ also is. 
In the end, the sequence $(\mathsf w^\mathbf a_n)_{n\geq 1}$ is a measurable function of $(\ttP_n)_{n\geq1}$ and it is easy to check that it corresponds to the one described in the statement of Corollary~\ref{wrt:cor: pa conditional distribution limiting degrees}.

\subsection{Proof of Proposition~\ref{wrt:prop:assum pa implies assum wrt}}\label{wrt:subsec:proof of proposition 2}
Let $(\mathsf W^\mathbf a_n)_{n\geq 1}$ be the random sequence of cumulated weights defined Theorem~\ref{wrt:thm:connection PA WRT}, whose distribution depends on a sequence $\mathbf a$ of fitnesses, and is expressed using a sequence of independent Beta-distributed random variables $(\beta_k)_{k\geq1}$. We are going to prove Proposition~\ref{wrt:prop:assum pa implies assum wrt}, which relates the growth of $(\mathsf{W}^\mathbf a_n)_{n\geq1}$ to the one of $(A_n)_{n\geq 1}$. 
\begin{proof}[Proof of Proposition~\ref{wrt:prop:assum pa implies assum wrt}]
	As in \cite[Proof of Lemma~1.1]{goldschmidt_line_2015}, we introduce 
	\begin{equation}\label{wrt:eq:def Xn}
	X_n:=\prod_{i=1}^{n-1}\frac{\beta_i}{\Ec{\beta_i}}.
	\end{equation}
	It is easy to see that $X_n$ is a positive martingale, hence it almost surely converges to a limit $X_\infty$ as $n\rightarrow\infty$. 
	Now, using the fact that the $(\beta_n)_{n\geq 1}$ are independent and that for any integer $q \geq 0$, the $q$-th moment of a random variable with $\mathrm{Beta}(a,b)$ distribution is given by
	\begin{equation}\label{wrt:eq:moments loi beta}
	\frac{ \Gamma(a+q)\Gamma(a+b)}{\Gamma(a) \Gamma(a+b+q)}=\prod_{k=0}^{q-1}\frac{a+k}{a+b+k},
	\end{equation}
	we can compute
	\begin{align*}
	\prod_{i=1}^{n-1}\Ec{\beta_i^p}=\prod_{i=1}^{n-1}\left(\prod_{k=0}^{p-1}\frac{i+A_i+k}{i+A_{i+1}+k}\right)
	&=\prod_{k=0}^{p-1}\left(\frac{1+A_1+k}{n+A_{n}+k-1}\prod_{i=2}^{n-1}\frac{i+A_i+k}{i+A_{i}+k-1}\right)\\
	&=\left(\prod_{k=0}^{p-1}\frac{1+A_1+k}{n+A_{n}+k-1}\right)\cdot \prod_{k=0}^{p-1}\prod_{i=2}^{n-1}\left(1+\frac{1}{i+A_{i}+k-1}\right)
	\end{align*}
	Now from \eqref{wrt:assum: pa An=cn+rn}, there exists $\epsilon>0$ such that $A_n=c\cdot n +\grandO{n^{1-\epsilon}}$ and without loss of generality we can assume that $\epsilon<1$.
	For all $k\in\intervalleentier{0}{p-1}$ we can write
	\[
	n+A_n+k-1\underset{n\rightarrow\infty}{=} (c+1)n + \grandO{n^{1-\epsilon}} \quad \text{and so} \quad \frac{1}{n+A_{n}+k-1}\underset{n\rightarrow\infty}{=}\frac{1}{(c+1)n}+\grandO{n^{-1- \epsilon}}.
	\]
	Hence
	\begin{align*}
		\prod_{k=0}^{p-1}\prod_{i=2}^{n-1}\left(1+\frac{1}{i+A_{i}+k-1}\right)&= \prod_{k=0}^{p-1}\prod_{i=2}^{n-1}\left(1+\frac{1}{(c+1)i}+\grandO{i^{-1- \epsilon}}\right)\\
		&=\exp\left(\sum_{k=0}^{p-1} \sum_{i=2}^n\left(\frac{1}{(c+1)i}+\grandO{i^{-1- \epsilon}}\right)\right)\\
		&=\exp\left(\frac{p}{c+1}\log n + \cst + \grandO{n^{-\epsilon}}\right)\\
		&=\cst\cdot n^{\frac{p}{c+1}}\left(1 + \grandO{n^{-\epsilon}}\right).
	\end{align*}
In the end, since $\left(\prod_{k=0}^{p-1}\frac{1+A_1+k}{n+A_{n}+k-1}\right)=\cst\cdot \prod_{k=0}^{p-1}\frac{1}{(c+1)n+\grandO{n^{1-\epsilon}}}=\cst\cdot n^{-p}\cdot (1+\grandO{n^{-\epsilon}})$, we get
	\begin{align}\label{wrt:eq:definition Cp}
	\prod_{i=1}^{n-1}\Ec{\beta_i^p}
	&= C_p\cdot n^{-p+p/(c+1)}\cdot (1+\grandO{n^{-\epsilon}})
	\end{align}
	where $C_p$ is a positive constant which depends on the sequence $\mathbf a$ and $p$.
	This entails that, under our assumptions, for any $p\geq1$, we have
	\begin{align*}
	\Ec{X_n^p} = \frac{\prod_{i=1}^{n-1}\Ec{\beta_i^p}}{\prod_{i=1}^{n-1}\Ec{\beta_i}^p}\underset{n\rightarrow\infty}{\rightarrow} \frac{C_p}{C_1^p},
	\end{align*}
 which shows that this martingale is bounded in $L^p$ for all $p\geq1$ and hence it is uniformly integrable. Consequently, it converges a.s.\ and in $L^p$ to a limit random variable $X_\infty$, with moments determined by
	\begin{equation}\label{wrt:eq:moments Xinfty}
	\forall p \geq 1, \qquad \Ec{X_\infty^p}=\frac{C_p}{C_1^p}.
	\end{equation}
	Furthermore, we have 
	\begin{align}\label{wrt:eq:upper-bound increments squared Xn}
	\Ec{(X_{n+1}-X_n)^2}=\Ec{X_{n}^2\left(\frac{\beta_n}{\Ec{\beta_n}}-1\right)^2}\leq \Ec{X_n^2} \cdot \frac{\Var{\beta_n}}{\Ec{\beta_n}^2}.
	\end{align}
	Since $\beta_n\sim \mathrm{Beta}(n+A_n,a_{n+1})$, we get
	\begin{align}\label{wrt:eq:expectation and variance beta}
		\Ec{\beta_n}=\frac{n+A_{n}}{n+A_{n+1}}\rightarrow1 \qquad \text{and} \qquad \Var{\beta_n}=\frac{a_{n+1}(n+A_n)}{(n+A_{n+1})^2(n+A_{n+1}+1)}=\grandO{\frac{a_{n+1}}{n^2}}.
	\end{align}
	Using \eqref{wrt:eq:upper-bound increments squared Xn}, \eqref{wrt:eq:expectation and variance beta}, Lemma~\ref{wrt:lem:biggins lemma} and then summing over $n\leq k\leq 2n-1$ and using the fact that $\mathbf{a}$ satisfies \eqref{wrt:assum: pa An=cn+rn} we get that 
	\begin{align*}
	\Ec{(X_{2n}-X_n)^2}= \grandO{\frac{\sum_{k=n}^{2n}a_k}{n^2}}=\grandO{n^{-1}}.
	\end{align*} 
	Using Lemma~\ref{wrt:lem:convergence analytic martingale}, we get that almost surely, for any $\epsilon'<\frac{1}{2}$,
	\begin{align*}
		\abs{X_n-X_\infty}=\grandO{n^{-\epsilon'}}.
	\end{align*}
	
	Since $\beta_i>0$ almost surely for every $i\geq1$, the event $\{X_\infty=0\}$ is a tail event for the filtration generated by the $\beta_i$ and has probability $0$ or $1$. In the end, it has probability $0$ because $\Ec{X_\infty}=1$.
	We deduce that
	\begin{align}\label{wrt:eq:asymptotic behaviour Wn with quantified error}
	(\mathsf W_n^\mathbf{a})^{-1}=\prod_{i=1}^{n-1}\beta_i&= X_n\cdot \prod_{i=1}^{n-1}\Ec{\beta_i}\notag\\
	&=X_\infty\cdot \left(1 + \grandO{n^{-\epsilon'}}\right)\cdot C_1\cdot n^{-1+\frac{1}{c+1}}\cdot \left(1+\grandO{n^{-\epsilon}}\right)\notag \\
	&=C_1\cdot X_\infty\cdot n^{-1+\frac{1}{c+1}} \cdot \left(1+\grandO{n^{-\epsilon}+n^{-\epsilon'}}\right).
	\end{align}
	Hence, we have, 
	\begin{equation}\label{wrt:eq: asymptotic for Wk}
	\mathsf W_n^\mathbf{a}\underset{n\rightarrow\infty}{\bowtie} Z\cdot n^{\frac{c}{(c+1)}} \qquad  \text{with} \qquad  Z:=\frac{1}{X_\infty\cdot C_1}.
	\end{equation}
	Whenever $a_n\leq n^{c'+\petito{1}}$ as $n\rightarrow\infty$, we can show the following (we postpone the proof to the end of the section) 
	\begin{lemma}\label{wrt:eq: proba beta exponentiellement petite}
	 For any $\delta>0$ small enough, we have 
\begin{align*}
\Pp{1-\beta_k>k^{-1+c'+\delta}}\leq \exp\left(-(k+1)^{c'+\delta+\petito{1}}\right).
\end{align*}
	\end{lemma}

	Since the last quantity is summable in $k$ we can use the Borel-Cantelli lemma (and a sequence of $\delta$ going to $0$) to show that almost surely $1-\beta_k\leq k^{-1+c'+o_\omega(1)}$ as $k\rightarrow\infty$, where the term $o_\omega(1)$ denotes a random function of $k$ that tends to $0$ when $k\rightarrow\infty$.
	Combining this with \eqref{wrt:eq: asymptotic for Wk}, we finish proving the proposition by writing
	\[
	\mathsf w_{k}^\mathbf{a}=\mathsf W_{k}^\mathbf{a}-\mathsf W_{k-1}^\mathbf{a}=\mathsf W_k^\mathbf{a}\cdot (1-\beta_{k-1})\leq (k+1)^{c'-1/(c+1)+o_\omega(1)}.\qedhere
	\]\end{proof}

We finish by giving a proof of Lemma~\ref{wrt:eq: proba beta exponentiellement petite}.
\begin{proof}[Proof of Lemma~\ref{wrt:eq: proba beta exponentiellement petite}]
	Let $x\geq0$ and $y>1$ and let $X$ be a random variable with distribution $\mathrm{Beta}(x+1,y)$ and $Y$ with distribution $\mathrm{Beta}(x,1)$, independent of $X$. By standard results on Beta distributions, the product $Z=X\cdot Y$ has distribution $\mathrm{Beta}(x,y)$.  
	
	Then for any $z\in\intervalleff{0}{1}$ we have, using the explicit expression of the density of $X$,
	\begin{align*}
	\Pp{Z>z}\leq \Pp{X>z}&=\frac{\Gam{x+1+y}}{\Gam{x+1}\Gam{y}}\int_{z}^{1}u^{x}(1-u)^{y-1} \dd{u}\\
	&\leq \frac{\Gam{x+1+y}}{\Gam{x+1}\Gam{y}}\exp\left(-(y-1)z\right)\int_{z}^{1}u^{x}\dd{u} \\
	&\leq \frac{\Gam{x+1+y}}{\Gam{x+2}\Gam{y}}\cdot\exp\left(-(y-1)z\right),
	\end{align*}
and the last display in increasing in $x$.
We are going to use this inequality for well-chosen sequences $(x_n)$, $(y_n)$ and $(z_n)$ taking place of the values of $x,y,z$. 
Let us first remark that for any two non-negative sequences $(x_n)$ and $(y_n)$ with $(y_n)$ going to infinity and $x_n=\petito{y_n}$, we have the following estimate using Stirling's approximation:
\begin{align*}%
\log\left(\frac{\Gam{x_n+1+y_n}}{\Gam{x_n+2}\Gam{y_n}}\right)\underset{n\rightarrow\infty}{=} (x_n+1)\log (y_n)\cdot (1+\petito{1}).
\end{align*}
Now let us apply the above computations for every $n\geq 1$ with $z_n:= n^{-1+c'+\delta}$ to the random variables $(1-\beta_n)$ which have distribution $\mathrm{Beta}\left(x_n,y_n\right)$, with $x_n:=a_{n+1}$ and $y_n:=A_n+n$. 
In particular, in this context we have $x_n=a_n\leq (n+1)^{c'+\petito{1}}$ and $y_n=A_n+n=(n+1)^{1+\petito1}$, so that the all of the above applies and
\begin{align*}\log \Pp{(1-\beta_n)>n^{-1+c'+\delta}}&\leq (x_n+1)\log(y_n) (1+\petito{1}) - (y_n-1)z_n \\
&\leq (n+1)^{c'+\petito{1}} - (n+1)^{c'+\delta+\petito{1}} \\
&\leq - (n+1)^{c'+\delta+\petito{1}},
\end{align*}
which is what we wanted.
\end{proof}

\subsection{The distribution of the limiting sequence}\label{wrt:subsec:distribution of the limiting sequence}
Let us stay in the setting of Section~\ref{wrt:subsec:proof of proposition 2}. Suppose that we are working with a sequence of fitnesses $\mathbf{a}$ that satisfies \eqref{wrt:assum: pa An=cn+rn} for some $c>0$. 
The sequence $\left(\mathsf M_n^{\textbf{a}}\right)_{n\geq1}$ is defined in \eqref{wrt:eq:def mna} as some random multiple of the sequence $\left(\mathsf W_n^{\textbf{a}}\right)_{n\geq1}$, whose distribution is described in Theorem~\ref{wrt:thm:connection PA WRT} from a sequence $(\beta_n)_{n\geq 1}$ of independent random variables with $\beta_n\sim \mathrm{Beta}(A_n+n,a_{n+1})$, so that for all $n\geq 1$,
\begin{align*}
\mathsf M^\mathbf{a}_n=\frac{c+1}{Z}\cdot \prod_{k=1}^{n-1}\beta_{k}^{-1},
\end{align*}
where the random variable $Z$ is the one that appears in \eqref{wrt:eq: asymptotic for Wk}, and depends on the whole sequence $(\beta_n)_{n\geq 1}$.
\begin{proposition}\label{wrt:prop:Mk markov chain}
	For any sequence $\textbf{a}$ that satisfies the condition \eqref{wrt:assum: pa An=cn+rn}, the sequence $\left(\mathsf M_k^{\textbf{a}}\right)_{k\geq1}$ is a (possibly time-inhomogeneous) Markov chain such that for all $k\geq 1$,
	 $\mathsf M_{k+1}^{\textbf{a}}$ is independent of $\beta_1,\beta_2,\dots,\beta_{k}$.
	 The fact that for all $k\geq 1$ we have $\mathsf{M}_{k}^{\textbf{a}}= \beta_k\cdot \mathsf M_{k+1}^{\textbf{a}}$ with $\beta_k\sim\mathrm{Beta}(A_k+k,a_{k+1})$ independent of $\mathsf M_{k+1}^{\textbf{a}}$ characterises the backward transitions of the chain.
\end{proposition}

\begin{proof}
	We follow the same steps as \cite[Lemma~1.1]{goldschmidt_line_2015}. 
Recall the definition of the random variable $X_\infty$ as the limit of the sequence $(X_n)_{n\geq 1}$ defined in \eqref{wrt:eq:def Xn}, the definition \eqref{wrt:eq:definition Cp} of the constant $C_1$ and their relation to the random variable $Z$.  
	We have
	\begin{equation}\label{wrt:eq:def of Mk}
	\mathsf M_1^\mathbf{a}=\frac{c+1}{Z}=\left(C_1\cdot (c+1)\cdot X_\infty\right) \quad \text{and for $k\geq 2$,} \quad \mathsf M_k^\mathbf{a}=\mathsf M_1^\mathbf{a}\cdot \prod_{i=1}^{k-1}\beta_i^{-1}.
	\end{equation}
	It then follows that we can write, for $k\geq 1$,
	\begin{align*}
	\mathsf M_{k+1}^\mathbf{a}=C_1\cdot (c+1)\cdot X_\infty\cdot \prod_{i=1}^{k}\beta_i^{-1}=C_1\cdot (c+1)\cdot \lim_{n\rightarrow\infty}\frac{\prod_{i=k+1}^{n-1}\beta_i}{\prod_{i=1}^{n-1}\Ec{\beta_i}},
	\end{align*}
	which ensures that $\mathsf M_{k+1}^\mathbf{a}$ is independent of $\beta_1,\beta_2,...,\beta_{k}$. The limit in the last equality exists almost surely thanks to the results of the preceding section.
	
	Now we prove the Markov property of the chain. Let $k\geq 1$. Because of the definition of the chain as a product, the distribution of $\mathsf M_{k+1}^\mathbf{a}$ conditional on the past trajectory $\mathsf M_1^\mathbf{a},\mathsf M_2^\mathbf{a},\dots,\mathsf M_k^\mathbf{a}$ is the same as the distribution of $\mathsf M_{k+1}^\mathbf{a}$ conditional on $\mathsf M_k^\mathbf{a},\beta_1,\dots,\beta_{k-1}$. Since $\mathsf M_{k+1}^\mathbf{a}=\beta_k^{-1}\cdot \mathsf M_k^\mathbf{a}$ and that $\beta_k$ and $\mathsf M_k^\mathbf{a}$ are both independent of $\beta_1,\dots,\beta_{k-1}$, this conditional distribution corresponds to the one of $\mathsf M_{k+1}^\mathbf{a}$ conditional on the present state of the chain $\mathsf M_k^\mathbf{a}$.
\end{proof}
\paragraph{Computing the moments.}
In some cases where the sequence $\mathbf{a}$ is sufficiently regular, we can compute explicitly every moment of the random variable $\mathsf M^\mathbf{a}_k$ for every $k\geq 1$. Indeed, using \eqref{wrt:eq:moments Xinfty} and \eqref{wrt:eq:def of Mk} and the independence, we get 
\begin{align}\label{wrt:eq:moments Mk}
\Ec{(\mathsf M_k^\mathbf{a})^p}=\Ec{\left(C_1\cdot (c+1) \cdot \lim_{n\rightarrow\infty}\frac{\prod_{i=k}^{n-1}\beta_i}{\prod_{i=1}^{n-1}\Ec{\beta_i}}\right)^p}
&= C_1^p\cdot(c+1)^p\cdot \lim_{n\rightarrow\infty}\frac{\prod_{i=k}^{n-1}\Ec{\beta_i^p}}{\left(\prod_{i=1}^{n-1}\Ec{\beta_i}\right)^p}\nonumber \\
&= \frac{(c+1)^p\cdot C_p}{\prod_{i=1}^{k-1}\Ec{\beta_i^p}}.
\end{align}
In general, if the collection $(\mu_p)_{p\geq 1}$ of $p$-th moments of some random variable satisfies the so-called Carleman's condition: $\sum_{p=1}^\infty \mu_{2p}^{-1/(2p)}=\infty$, then its distribution is uniquely determined from those moments.
\section{Examples and applications}\label{wrt:sec:examples and applications}
In this section, we compute the explicit distribution of $(\mathsf M_n^{\mathbf{a}})$ for some particular sequences $\mathbf a$. 
We then describe some applications of our results to a model of P\'olya urn with immigration and then to a model of preferential attachment graphs.
\subsection{The limit chain for particular sequences $\mathbf{a}$}
As stated in the preceding section, we can compute the distribution of $\mathsf{M}^\mathbf{a}_k$ for some fixed $k$ by the expression of its moments \eqref{wrt:eq:moments Mk}, provided that they satisfy Carleman's condition. Knowing these distributions and the backward transitions given in Proposition~\ref{wrt:prop:Mk markov chain} then characterises the law of the whole process.
For two particular examples, this law has a nice expression.
\begin{proposition}\label{wrt:prop:distribution limiting chains}
	In the two following cases, the distribution of the chain $(\mathsf M_n^{\mathbf{a}})$ is explicit.
	\begin{enumerate}
\item\label{wrt:it:limit chain is ML} 	If $\mathbf{a}$ is of the form $\mathbf{a}=(a,b,b,b,\dots)$ with $a>-1$ and $b>0$, then the limiting sequence $(\mathsf M_n^\mathbf{a})_{n\geq1}$ is a Mittag-Leffler Markov chain $\MLMC\left(\frac{1}{b+1},\frac{a}{b+1}\right)$. 
\item\label{wrt:it:limit chain is PGG} If $\mathbf{a}$ is of the form $\mathbf{a}=(a,b_1,b_2,\dots,b_\ell,b_1,b_2,\dots b_\ell,b_1,\dots)$, periodic of period $\ell$ starting from the second term with $a>-1$ and $\ell$ integers $b_1,b_2,\dots,b_\ell$ with at least one being non-zero, then, letting $S=b_1+b_2+\dots+b_\ell$, the sequence $\frac{\ell^{\frac{-\ell}{S+\ell}}}{S+\ell} \cdot (\mathsf M_n^\mathbf{a})_{n\geq1}$ has the distribution of an Intertwined Product of Generalised Gamma Processes with parameters $(a,b_1,b_2,\dots,b_\ell)$, which we denote $\mathrm{IPGGP}(a,b_1,b_2,\dots,b_\ell)$.
	\end{enumerate}
\end{proposition}
 Note that the two cases \ref{wrt:it:limit chain is ML} and \ref{wrt:it:limit chain is PGG} are not mutually exclusive. We will prove the two points of this proposition in separate subsections. The proper definitions of the distributions to which we refer in the statement are given along the proof.

\subsubsection{Mittag-Leffler Markov chains}\label{wrt:subsubsec:mlmc}
Let us study the case where the underlying preferential attachment tree has a sequence of fitnesses $\mathbf{a}$ that are of the form $(a,b,b,b,\dots)$. 
We start by recalling the definitions of Mittag-Leffler distributions and Mittag-Leffler Markov chains and introduced in \cite{goldschmidt_line_2015}, and also studied in \cite{james_generalized_2015_online}.
\paragraph{Mittag-Leffler distributions.}
Let $0<\alpha<1$ and $\theta>-\alpha$. The generalized Mittag-Leffler $\mathrm{ML}(\alpha, \theta)$ distribution has $p$th moment
\begin{align}\label{wrt:eq:moments mittag-leffler}
\frac{\Gamma(\theta) \Gamma(\theta/\alpha + p)}{\Gamma(\theta/\alpha) \Gamma(\theta + p \alpha)}=\frac{\Gamma(\theta+1) \Gamma(\theta/\alpha + p+1)}{\Gamma(\theta/\alpha+1) \Gamma(\theta + p \alpha+1)}
\end{align}
and the collection of $p$-th moments for $p \in \N$ uniquely characterizes this distribution thanks to Carleman's criterion.
\paragraph*{Mittag-Leffler Markov Chains.}
For any $0<\alpha<1$ and $\theta>-\alpha$, we introduce the (a priori) inhomogenous Markov chain $(\mathsf M^{\alpha,\theta}_n)_{n\geq 1}$, the distribution of which we call the Mittag-Leffler Markov chain of parameters $(\alpha,\theta)$, or $\MLMC(\alpha,\theta)$. This type of Markov chain was already defined in \cite{goldschmidt_line_2015}, for some choice of parameters $\alpha$ and $\theta$. 
It is a Markov chain such that for any $n\geq 1$,
\begin{equation*}%
\mathsf M_n^{\alpha,\theta}\sim \mathrm{ML}\left(\alpha,\theta+n-1\right),
\end{equation*}
and the transition probabilities are characterised by the following equality in law:
\begin{equation*}%
\left(\mathsf M_n^{\alpha,\theta},\mathsf M_{n+1}^{\alpha,\theta}\right)=\left(B_n\cdot \mathsf M_{n+1}^{\alpha,\theta},\mathsf M_{n+1}^{\alpha,\theta}\right),
\end{equation*}
with $B_n\sim \mathrm{Beta}\left(\frac{\theta+n-1}{\alpha}+1,\frac{1}{\alpha}-1\right)$, independent of $\mathsf M_{n+1}^{\alpha,\theta}$. These chains are constructed (for some values of $\theta$ depending on $\alpha$) in \cite{goldschmidt_line_2015}. In fact, our proof of Proposition~\ref{wrt:prop:distribution limiting chains}\ref{wrt:it:limit chain is ML} ensures that these chains exist for any choice of parameters $0<\alpha<1$ and $\theta>-\alpha$.
Let us mention that the proof of \cite[Lemma~1.1]{goldschmidt_line_2015} is still valid for the whole range of parameters $0<\alpha<1$ and $\theta>-\alpha$, which proves that these Markov chains are in fact time-homogeneous. 
We provide, in a later paragraph, another proof of this time-homogeneity using an argument that relies on preferential attachment trees.

%
%
%
%
%
%
%
%
\begin{comment}

\begin{proof}
Denote $f(t,m')$ the joint density of $(B_n,M_{n+1})$ at point $(t,m')$. Then using \eqref{wrt:eq:egalité en loi transition} we let $m=tm'$ and change variables from $(t,m')$ to $(m,m')$. The density of $\left(M_{n},M_{n+1}\right)$ is then given by $(m')^{-1} f\left(\frac{m}{m'},m'\right)$. Using the explicit expressions for the density of $B_n$ and $M_{n+1}$ and equation \eqref{wrt:eq:egalité en loi transition}, we get the following expression for the density of $\left(M_{n},M_{n+1}\right)$ at $(m,m')$
\begin{equation*}
\frac{\Gam{\theta+n+1}\Gam{\frac{\theta+n}{\alpha}}}{\Gam{\frac{\theta+n}{\alpha}+1}\Gam{\frac{\theta+n-1}{\alpha}+1}\Gam{\frac{1}{\alpha}-1}}m'm^{\frac{\theta+n-1}{\alpha}}(m'-m)^{\frac{1}{\alpha}-2}g_\alpha(m')\ind{m'\geq m}.
\end{equation*}
Dividing the last expression by the density of $M_n$ at point $m$, we get the following transition density of the chain from $m$ to $m'$,
\begin{equation}
p(m,m')=\frac{\alpha}{\Gam{\frac{1}{\alpha}-1}}m' (m'-m)^{\frac{1}{\alpha}-2}\frac{g_\alpha(m')}{g_\alpha(m)}\ind{m'\geq m}.
\end{equation} 
Since this density does not depend on $n$, the chain is time-homogeneous.
\end{proof}

\end{comment}
%
\paragraph{The limiting Markov chain is a Mittag-Leffler.}
Recall the definition of the sequence $(\beta_k)_{k\geq 1}$ and their respective distributions $\beta_k\sim \mathrm{Beta}(A_k+k,a_{k+1})$. From our assumption that $\mathbf a=a,b,b,b\dots$ we have for all $k\geq 1$, \[(A_k+k,a_{k+1})=(1+a+(k-1)b,b).\]
\begin{proof}[Proof of Proposition~\ref{wrt:prop:distribution limiting chains}~\ref{wrt:it:limit chain is ML}]
For $p\geq 1$, we can make the following computation, using \eqref{wrt:eq:moments loi beta}, one change of indices and several times the property of the Gamma function that for any $z>0$ we have $\Gam{z+1}=z\Gam{z}$:
\begin{align}\label{wrt:eq:prod betai MLMC}
\prod_{i=1}^{n-1}\Ec{\beta_i^p}&=\prod_{i=1}^{n-1}\frac{\Gam{1+a+p+(b+1)(i-1)}\Gam{a+(b+1)i}}{\Gam{1+a+(b+1)(i-1)}\Gam{a+(b+1)i+p}}\notag\\
&=\frac{\Gam{1+a+p}}{\Gam{1+a}}\cdot \frac{\Gam{a+(b+1)(n-1)}}{\Gam{a+(b+1)(n-1)+p}}\cdot\frac{\Gam{\frac{a+p}{b+1}+n-1}}{\Gam{\frac{a}{b+1}+n-1}}\cdot\frac{\Gam{1+\frac{a}{b+1}}}{\Gam{1+\frac{a+p}{b+1}}}.
\end{align}
Using Stirling formula, we can then compute the numbers $C_p$ introduced in \eqref{wrt:eq:definition Cp},
\begin{align}\label{wrt:eq:Cp MLMC}
C_p=(b+1)^{-p}\cdot \frac{\Gam{1+a+p}\Gam{1+\frac{a}{b+1}}}{\Gam{1+a}\Gam{1+\frac{a+p}{b+1}}}.
\end{align}
Using \eqref{wrt:eq:moments Mk}, the moments of $\mathsf M_k$ are given, for any $p\in\N$ by the formula:
\begin{align*}
\Ec{(\mathsf M_k^\mathbf{a})^p}= \frac{(b+1)^p\cdot C_p}{\prod_{i=1}^{k-1}\Ec{\beta_i^p}} \underset{\eqref{wrt:eq:Cp MLMC},\eqref{wrt:eq:prod betai MLMC}}{=}\frac{\Gam{\frac{a}{b+1}+k-1}\Gam{a+(b+1)(k-1)+p}}{\Gam{a+(b+1)(k-1)}\Gam{\frac{a+p}{b+1}+k-1}}
\end{align*}
These moments identify using \eqref{wrt:eq:moments mittag-leffler} the distribution of $\mathsf M_k^\mathbf{a}$ for all $k\geq 1$,
\begin{align*}
\mathsf M_k^\mathbf{a}\sim \mathrm{ML}\left(\frac{1}{b+1},\frac{a}{b+1}+k-1\right).
\end{align*}
From this, and the form of the backward transitions, we can identify $(\mathsf M_k^\mathbf{a})_{k\geq 1}$ as having a distribution $\mathrm{MLMC}\left(\frac{1}{b+1},\frac{a}{b+1}\right)$. 
\end{proof}

\paragraph{Time-homogeneity of MLMC.}
Let us keep the notation from the previous paragraph with a sequence $\mathbf{a}=a,b,b,b\dots$ and let us show the time-homogeneity of the corresponding Mittag-Leffler Markov chain $(\mathsf{M}^\mathbf{a}_k)_{k\geq 1}\sim \mathrm{MLMC}\left(\frac{1}{b+1},\frac{a}{b+1}\right)$ using its connection with preferential attachment trees.

For any $x>-1$, consider the sequence $\boldsymbol{x}=x,b,b,b\dots$ and $(\ttP^x_n)_{n\geq 1}\sim \pa(\boldsymbol{x})$ in such a way that, using Theorem~\ref{wrt:thm:convergence degrees pa},
\begin{align*}
	\mathsf{M}_1^{\boldsymbol{x}}=\lim_{n\rightarrow\infty} n^{-1/(b+1)}\cdot \deg_{\ttP_n^x}^+(u_1) \quad \text{and}\quad 	\mathsf{M}_2^{\boldsymbol{x}}=\lim_{n\rightarrow\infty} n^{-1/(b+1)}\cdot (\deg_{\ttP_n^x}^+(u_1)+ \deg_{\ttP_n^x}^+(u_2)).
\end{align*}
By choosing $x$ appropriately, we can make $(\mathsf{M}_1^{\boldsymbol{x}},\mathsf{M}_2^{\boldsymbol{x}})$ have the distribution of any of the couples $(\mathsf{M}_k,\mathsf{M}_{k+1})$ for $k\geq 1$.
Thus, in order to prove the time-homogeneity of the transitions, it suffices to prove that the conditional distribution of $\mathsf{M}_2^{\boldsymbol{x}}$ with respect to $\mathsf{M}_1^{\boldsymbol{x}}$ does not depend on $x$.

Recall from Section~\ref{wrt:subsubsec:convergence of the weight measure intro} in the introduction that we see $(\ttP_n^x)_{n\geq 1}$ as an increasing sequence of plane trees, defined as subsets of $\bU$. Also recall that for any $u\in\bU$, we denote $T(u)$ the subtree descending from $u$.
 At every time $n\geq 1$, we can consider the sequence $(\#(\ttP_n^x\cap T(1)),\#(\ttP_n^x\cap T(2)),\dots)$, which counts the number of vertices in the subtrees descending from the children of $u_1=\emptyset$ in $\ttP_n^x$, in order of creation (completed by an sequence of zeros).
We can check that this sequence evolves as $n$ grows with the same distribution as the number of customers seating at different tables in a Chinese Restaurant Process with seating plan $(\frac{1}{b+1},\frac{x}{b+1})$, see \cite[Section~3.2]{pitman_combinatorial_2006} for a definition.

Then, conditionally on the evolution of this sequence, every time that a vertex is added to one of those subtrees, it is attached to any vertex already present in the subtree with probability proportional to its out-degree plus $b$ (and in particular this does not depend on the value of $x$).

Thanks to \cite[Corollary~3.9]{pitman_combinatorial_2006}, two Chinese Restaurant Processes with respective seating plan $(\frac{1}{b+1},\frac{x}{b+1})$ and $(\frac{1}{b+1},\frac{x'}{b+1})$ with $x,x'>-1$ have a density with respect to each other and this density is a function of the scaling limit of the number of tables created in the process, which corresponds in our case to $\mathsf{M}_1^{\boldsymbol{x}}$.

These observations allow us to conclude that the distribution of $(\ttP^x_n)_{n\geq 1}$ for any $x>-1$ has a positive density with respect to $(\ttP^0_n)_{n\geq 1}$, and this density is a function of $\mathsf{M}_1^x$. 
From here, it is clear that conditionally on $\mathsf{M}_1^{\boldsymbol{x}}$, the distribution of the quantity $(\mathsf{M}_2^{\boldsymbol{x}}-\mathsf{M}_1^{\boldsymbol{x}})=\lim_{n\rightarrow\infty}\deg_{\ttP_n^x}^+(u_2)$ does not depend on $x$, which concludes the argument.

\subsubsection{Products of generalised Gamma.}\label{wrt:subsec:product generalised gamma}
The following paragraphs aim at proving Proposition~\ref{wrt:prop:distribution limiting chains}\ref{wrt:it:limit chain is PGG}. 
In the first paragraph and second paragraph we define the families of distributions of $\mathrm{GGP}$ and $\mathrm{IPGGP}$-processes. Some special cases of these processes already appeared in \cite{pekoez_generalized_2016, pekoez_joint_2017}.
In the third one we prove that the distribution of $(\mathsf M_k^\mathbf{a})_{k\geq 1}$ belongs to this family whenever the sequence $\mathbf a$ is of the form assumed in Proposition~\ref{wrt:prop:distribution limiting chains}\ref{wrt:it:limit chain is PGG}.

\paragraph{Construction of a $\mathrm{GPP}(z,r)$-process.} 
For $z,r>0$ real numbers, let $(Z_i)_{i\geq 1}$ be a family of independent variables with the following distribution:
\begin{equation*}
Z_1\sim\mathrm{Gamma}\left(\frac{z}{r}\right) \quad \text{and for } i\geq 2, \quad Z_i\sim \mathrm{Exp}(1),
\end{equation*}
where, for any $k>0$, the distribution $\mathrm{Gamma}(u)$ has density $x\mapsto \frac{x^{u-1}e^{-x}}{\Gamma(u)}\ind{x>0}$ with respect to the Lebesgue measure.  
Then for all $k\geq 1$ we define $\mathsf{G}_k$ as,
\begin{equation*}
\mathsf{G}_k:=\left(\sum_{i=1}^{k}Z_i\right)^{\frac{1}{r}}.
\end{equation*}
We say that the process $(\mathsf{G}_k)_{k\geq 1}$ has the distribution of a \emph{Generalised Gamma process} with parameters $(z,r)$ which we denote $\mathrm{GPP}(z,r)$.

Let us note that,using standard distributional equalities with Gamma and Beta distributions, for every $k\geq 1$, we have $(\mathsf{G}_k)^r\sim \mathrm{Gamma}\left(k-1+\frac{z}{r}\right)$ and
\begin{align}\label{wrt:eq:finite dimensional gamma process}
V_k:=\left(\frac{\mathsf{G}_{k}}{\mathsf{G}_{k+1}}\right)^r\sim \mathrm{Beta}\left(k-1+\frac{z}{r},1\right), \quad \text{so that} \quad V_k^{1/r}=\frac{\mathsf{G}_{k}}{\mathsf{G}_{k+1}}\sim \mathrm{Beta}\left(\frac{k-1}{r}+z,1\right),
\end{align}
and $V_k^{1/r}$ is independent of $\mathsf{G}_{k+1}$. 
In fact, we can further show that $V_1,V_2,\dots ,V_k, \mathsf{G}_{k+1}$ are jointly independent with the corresponding distribution and that this characterizes the finite dimensional marginals of this process.

\begin{remark}\label{wrt:rem:poisson point process}
		For $z=r$, the process $(\mathsf{G}_k)_{k\geq 1}$ has exactly the distribution of the points of a Poisson process on $\intervalleoo{0}{\infty}$ with intensity $r\cdot t^{r-1}\dd t$, listed in increasing order.
\end{remark}

\paragraph{Intertwined Products of $\mathrm{GGP}$-processes.}
Let $a>-1$ and $b_1,b_2,\dots,b_\ell$ be positive integers with at least one being non-zero.
We let $B_r:=\sum_{s=1}^{r}b_s$ for all $0\leq r \leq \ell$, with the convention that $B_0:=0$. We also let $S=B_\ell$. Then we define the set
\[\mathcal{S}:=\{1,2,\dots,S+\ell-1\}\setminus \enstq{B_{r}+r}{1\leq r \leq \ell-1}.\]
Start with independent $\mathrm{GPP}$ processes $\enstq{\mathsf{G}^{(q)}}{q\in \mathcal{S}}$ indexed by $\mathcal S$ such that for all $q\in \mathcal S$,
\[\mathsf{G}^{(q)}\sim\mathrm{GPP}(a+q,\ell+S).\]
Now $\mathsf{G}=(\mathsf{G}_k)_{k\geq 1}$ is defined in such a way that for all $n\geq 1$ and $1\leq r\leq \ell$ we have
\begin{align}\label{wrt:eq:description IPGG process}
\mathsf{G}_{\ell \cdot (n-1)+r}=\left(\underset{1\leq q\leq r-1+B_{r-1}}{\prod_{q\in \mathcal S}}\mathsf{G}^{(q)}_{n+1} \right) \cdot \left( \underset{r+B_{r-1}\leq q \leq S+\ell-1}{\prod_{q\in \mathcal S}}\mathsf{G}^{(q)}_{n} \right).
\end{align} 
The process $(\mathsf{G}_k)_{k\geq 1}$ defined above is said to have distribution of an \emph{Intertwined Product of Generalized Gamma Processes} with parameters $(a,b_1,b_2,\dots,b_\ell)$, denoted $\mathrm{IPGGP}(a,b_1,b_2,\dots,b_\ell)$.
Its finite dimensional marginals can be obtained in the same way as it was done in the preceding paragraph for Generalized Gamma processes. 

\paragraph{Identification of the limiting chain.}
Fix $\ell\geq1$ and $b_1,b_2,\dots,b_\ell\geq 0$ some integers (where at least one is non-zero) and suppose that the sequence $\mathbf{a}$ has the following form,
\begin{align*}
\mathbf{a}=(a,b_1,b_2,\dots,b_\ell,b_1,b_2,\dots b_\ell,b_1,\dots)
\end{align*}
meaning that the sequence is periodic with period $\ell$ starting from the second term, with $a>-1$. 

For any $j\geq 0$ and $1\leq r\leq \ell$ we have
\[\beta_{j\ell+r}\sim \mathrm{Beta}\left(a+B_{r-1}+r+j(\ell+S),b_r\right),\]
for the $(\beta_k)_{k\geq 1}$ as defined in Theorem~\ref{wrt:thm:connection PA WRT}.
For any $j\geq 1,p\geq 1$, we use the moments \eqref{wrt:eq:moments loi beta} of a Beta random variable and a telescoping argument to write
\begin{align*}
\Ec{(\beta_{j\ell+r})^p}&=\prod_{q=0}^{p-1}\frac{a + B_{r-1}+r +j(\ell+S)+q}{a + B_{r-1}+r +j(\ell+S)+q+b_r}=\prod_{q=0}^{b_r-1}\frac{a + B_{r-1}+r +j(\ell+S)+q}{a + B_{r-1}+r +j(\ell+S)+q+p}.
\end{align*}
Using the last display, we get that for any $n\geq 1$,
\begin{align*}
\prod_{j=0}^{n}\Ec{(\beta_{j\ell+r})^p}
&=\prod_{q=0}^{b_r-1}\frac{\Gam{\frac{a + B_{r-1}+r+q}{\ell+S}+n+1}\Gam{\frac{a + B_{r-1}+r+q+p}{\ell+S}}}{\Gam{\frac{a + B_{r-1}+r+q}{\ell+S}}\Gam{\frac{a + B_{r-1}+r+q+p}{\ell+S}+n+1}}.
\end{align*}
Using Stirling's approximation we get
\begin{align*}
\prod_{j=0}^{n}\Ec{(\beta_{j\ell+r})^p}\underset{n\rightarrow\infty}{\sim} n^{-\frac{p\cdot b_r}{\ell+S}}\cdot \prod_{q=0}^{b_r-1}\frac{\Gam{\frac{a + B_{r-1}+r+q+p}{\ell+S}}}{\Gam{\frac{a + B_{r-1}+r+q}{\ell+S}}}.
\end{align*}
Hence, recalling the definition of $C_p$ in \eqref{wrt:eq:definition Cp}, we get
\begin{align*}
C_p=\ell^{\frac{pS}{\ell+S}}\cdot \prod_{r=1}^\ell \prod_{q=0}^{b_r-1}\frac{\Gam{\frac{a + B_{r-1}+r+q+p}{\ell+S}}}{\Gam{\frac{a + B_{r-1}+r+q}{\ell+S}}}=\ell^{\frac{pS}{\ell+S}}\cdot\prod_{i\in \mathcal S}\frac{\Gam{\frac{a + i+p}{\ell+S}}}{\Gam{\frac{a +i}{\ell+S}}}.
\end{align*}

Then using \eqref{wrt:eq:moments Mk} with $c=S/\ell$,
\begin{align*}
&\Ec{(\mathsf M_{\ell\cdot(n-1)+r}^\mathbf{a})^p}\\
&=\frac{(c+1)^p\cdot C_p}{\prod_{i=1}^{\ell\cdot(n-1)+r-1}\Ec{\beta_i^p}}\\
&=\left(\frac{S+\ell}{\ell}\right)^p\cdot \ell^{\frac{pS}{S+\ell}}\cdot \prod_{s=1}^{r-1} \prod_{j=0}^{b_r-1}\frac{\Gam{n+\frac{a + B_{r-1}+r+j+p}{\ell+S}}}{\Gam{n+\frac{a + B_{r-1}+r+j}{\ell+S}}}\cdot \prod_{s=r}^\ell \prod_{j=0}^{b_r-1}\frac{\Gam{n-1+\frac{a + B_{r-1}+r+j+p}{\ell+S}}}{\Gam{n-1+\frac{a + B_{r-1}+r+j}{\ell+S}}}\\
&=\left((S+\ell)\cdot (\ell^{\frac{-\ell}{S+\ell}})\right)^p \cdot \underset{1\leq q\leq r-1+B_{r-1}}{\prod_{q\in \mathcal S}}\frac{\Gam{n+\frac{a + q+p}{\ell+S}}}{\Gam{n+\frac{a+q}{\ell+S}}}\cdot\underset{r+B_{r-1}\leq q \leq S+\ell-1}{\prod_{q\in \mathcal S}}\frac{\Gam{n-1+\frac{a + q+p}{\ell+S}}}{\Gam{n-1+\frac{a +q}{\ell+S}}}.
\end{align*}
Using the last display and the fact that random variable with distribution $\mathrm{Gamma}(u)$ has $p$-th moment equal to $\frac{\Gam{u+p}}{\Gam{u}}$,
we can identify the distribution of the one-dimensional marginals $\frac{\ell^{\frac{\ell}{S+\ell}}}{S+\ell} \cdot \mathsf M_k^\mathbf{a}$ for any $k\geq1$ with the ones of the process described in \eqref{wrt:eq:description IPGG process}. 
The identification of the distribution of the process $\frac{\ell^{\frac{\ell}{S+\ell}}}{S+\ell}\cdot(  \mathsf M_k^\mathbf{a})_{k\geq 1}$ as  $\mathrm{IPGGP}(a,b_1,b_2,\dots,b_\ell)$ is then obtained by comparing their finite dimensional distribution which are characterized by Proposition~\ref{wrt:prop:Mk markov chain} and, respectively, \eqref{wrt:eq:description IPGG process} together with the discussion below \eqref{wrt:eq:finite dimensional gamma process}.
\paragraph{Sparse sequences.}
Let us treat a particular example of parameters $a,b_1,b_2,\dots b_\ell$ for which the distribution $\mathrm{IPGGP}(a,b_1,b_2,\dots,b_\ell)$ has a simpler description than the general case. 
Suppose that only one of the parameters $b_1,b_2,\dots,b_\ell$ is non-zero, say $b_\ell$ for example. 
Keeping the notation introduced above, the corresponding set $\mathcal{S}$ contains only $b_\ell$ elements $\mathcal{S}=\{\ell, \ell+1,\dots,\ell+b_\ell-1\}$.
Following the definition \eqref{wrt:eq:description IPGG process}, the process $(\mathsf{G}_k)_{k\geq 1}$ with distribution $\mathrm{IPGGP}(a,0,0,\dots,0,b_\ell)$ is constant on every interval $\intervalleentier{(n-1)\ell+1}{n\ell}$ for any integer $n\geq 1$ and , the process $(\mathsf{G}_{(n-1)\ell+1})_{n\geq 1}$ is just given by a product of $b_\ell$ independent $\mathrm{GGP}$-processes
\[\mathsf{G}_{(k-1)\ell+1}=\prod_{\ell \leq q\leq \ell+ b_\ell-1}\mathsf{G}^{(q)}_{k},\]
where for all $\ell\leq q \leq \ell +b_\ell-1$, the process $(\mathsf{G}^{(q)}_k)_{k\geq 1}$ has distribution $\mathrm{GPP}(a+q,\ell+b_\ell)$.

In the particular case where $a=1$ and $(b_1,b_2,\dots,b_{\ell-1},b_\ell)=(0,0,\dots ,0,1)$, the picture is even simpler because the last display becomes a product over only one term.
We can check using Remark~\ref{wrt:rem:poisson point process} that the process $(\mathsf{G}_{(k-1)\ell +1})_{k\geq 1}$ has then exactly the distribution of the points of a Poisson process on $\intervalleoo{0}{\infty}$ with intensity $(\ell+1)t^\ell\dd t$, listed in increasing order, which was already noted in \cite[Remark~2]{pekoez_joint_2017}.

\subsection{Application to P\'olya urns with immigration}\label{wrt:subsec:application to polya urn with immigration}  
Define the following generalisation of P\'olya's urn, which depends on a sequence of numbers $(a_n)_{n\geq1}$: start at time $1$ with an urn containing $a_1$ red balls. At every time $n\geq 2$, we sample a ball uniformly at random from the urn, return it to the urn with $1$ additional ball of the same colour, plus an immigration of $a_n$ additional white balls. 
The outcome of the first step being deterministic, it is equivalent to consider that we start at time $2$ with $a_1+1$ red balls and $a_2$ white balls in the urn, so that we allow ourselves to consider any (possibly negative) value $a_1>-1$.  
This model was studied in the sequence of paper \cite{pekoez_degree_2013,pekoez_generalized_2016,pekoez_joint_2017} in specific cases of periodic immigration and also studied in \cite{banderier_periodic_2019_online} with a larger class of periodic immigration. 

Denote $R_n$ the number of red balls in the urn at time $n$ and let us state a scaling limit result for $R_n$ when $n\rightarrow\infty$. We also identify the speed of convergence and the Gaussian fluctuations around the limit, provided that the immigration is sufficiently regular.

Recall from the introduction the assumption \eqref{wrt:assum: pa An=cn+rn} defined for a real number $c>0$. We introduce the following more precise assumption of the same type, for any $c>0$ and $\delta>0$.
\begin{equation}\label{wrt:assum:ad hoc pa assum for gaussian limit}
A_n=c\cdot n \cdot \left(1+\grandO{n^{-\delta}}\right). \tag{$H_c^\delta$}
\end{equation}
Remark that for any $0<\delta<\frac{1}{2}$, this assumption is satisfied for periodic sequences $\mathbf{a}$, and almost surely satisfied by sequences of i.i.d. non-negative random variables with a second moment. 
\begin{proposition}
Assume that the sequence $\mathbf{a}=(a_n)_{n\geq 1}$ satisfies \eqref{wrt:assum: pa An=cn+rn} for some $c>0$. Then for $D_n:=n^{-\frac{1}{c+1}}\cdot R_n$,
\begin{enumerate}
	\item \label{wrt:it:convergence nombre de boules normalise} we have the following almost sure convergence,
	\begin{align*}
	D_n\underset{n\rightarrow\infty}{\longrightarrow} D_\infty,
	\end{align*}
	where $D_\infty$ has the same law as $\mathsf{M}_1^\mathbf{a}$, defined in \eqref{wrt:eq:def mna}.
	\item \label{wrt:it:gaussian fluctuations} If $\delta>\frac{1}{2(c+1)}$ then we have
	\begin{align*}
	n^{\frac{1}{2(c+1)}}\cdot \frac{D_\infty-D_n}{\sqrt{D_n}}\overset{\mathrm{(d)}}{\underset{n\rightarrow\infty}{\longrightarrow}} \mathcal{N}(0,1).
	\end{align*}
\end{enumerate}
\end{proposition}
\begin{remark}
If the sequence $(a_n)_{n\geq 1}$ has one of the particular forms treated in Proposition~\ref{wrt:prop:distribution limiting chains} of the previous section, we can identify the distribution of the limiting random variable as being Mittag-Leffler or a product of independent generalised Gamma random variables. This gives us an alternative proof for the similar statement \cite[Theorem~3.8]{banderier_periodic_2018}. 
\end{remark}
\begin{proof}Let $(\ttP_n)_{n\geq 1}$ be a sequence of trees with distribution $\pa(\mathbf{a})$ and let $R_n:=a_1+\deg^+_{\ttP_n}(u_1)$. With this definition, the sequence $(R_n)_{n\geq 1}$ has exactly the same distribution as the number of red balls in a P\'olya urn with immigration with immigration sequence $\mathbf{a}$.
	
If the sequence $\mathbf{a}$ satisfies our assumption \eqref{wrt:assum: pa An=cn+rn} for some $c>0$ then using \eqref{wrt:eq:convergence degrees pa} we can write the following almost sure convergence
\begin{equation*}
	n^{-\frac{1}{c+1}}\cdot\deg^+_{\ttP_n}(u_1)\underset{n\rightarrow\infty}{\rightarrow}\mathsf{M}_1^\mathbf{a},
\end{equation*}
where the sequence $(\mathsf{M}_n^\mathbf{a})_{n\geq 1}$ is defined in \eqref{wrt:eq:def mna}, so this proves \ref{wrt:it:convergence nombre de boules normalise}.

Let us turn to the proof of \ref{wrt:it:gaussian fluctuations}. We will prove this convergence in two steps, by first proving some corresponding result for the degree of the first vertex in a $\wrt$, and then using Theorem~\ref{wrt:thm:connection PA WRT} and Proposition~\ref{wrt:prop:assum pa implies assum wrt} to transfer the result to the corresponding $\pa$ distribution.
Indeed, let $(\ttT_n)_{n\geq 1}$ be a sequence of trees with distribution $\wrt(\boldsymbol{w})$ with a sequence $\boldsymbol{w}$ satisfying the following assumption
\begin{equation}\label{wrt:assum:wrt gaussian limit} W_n\underset{n\rightarrow\infty}{=}\frac{n^{\gamma}}{1-\gamma}\cdot \left(1+\petito{n^{-\frac{1-\gamma}{2}}}\right),
\end{equation}
for some $\gamma\in \intervalleoo{0}{1}$.
In this context, recalling \eqref{wrt:eq:equality distribution degree}, the degree of the first vertex can be written as
\begin{align*}
\deg^+_{\ttT_n}(u_1)&=\sum_{i=2}^{n}\ind{U_i\leq \frac{w_1}{W_n}}\\
&=w_1\cdot n^{1-\gamma}+w_1\cdot \left(\sum_{i=2}^{n}\frac{1}{W_n}-n^{1-\gamma}\right) + \sum_{i=2}^{n}\left(\ind{U_i\leq \frac{w_1}{W_n}}-\frac{w_1}{W_n}\right).
\end{align*}
Now, using our assumption on the sequence $(W_n)_{n\geq 1}$ we get
$\frac{1}{W_n}=\frac{1-\gamma}{n^\gamma} + \petito{n^{-\frac{\gamma}{2}-\frac{1}{2}}}$,
so that
\begin{align*}
\left(\sum_{i=2}^{n}\frac{1}{W_n}-n^{1-\gamma}\right)=\petito{n^{\frac{1-\gamma}{2}}}.
\end{align*}
Rearranging the terms, we get
\begin{align*}
	n^{\frac{1-\gamma}{2}}\cdot (n^{-(1-\gamma)}\cdot \deg^+_{\ttT_n}(u_1)-w_1)=n^{-\frac{1-\gamma}{2}}\cdot \sum_{i=2}^{n}\left(\ind{U_i\leq \frac{w_1}{W_n}}-\frac{w_1}{W_n}\right) + \petito{1},
\end{align*}
and using the Lindeberg-Feller theorem (see \cite[Theorem~3.4.5]{durrett_probability_2010} for example), we get that the latter expression converges in distribution when $n\rightarrow\infty$ to a Gaussian distribution $\mathcal{N}(0,w_1)$. Recalling that $n^{-(1-\gamma)}\cdot \deg^+_{\ttT_n}(u_1)\rightarrow w_1$ a.s. as $n\rightarrow\infty$,  we can also write using Slutsky's lemma
\begin{align}\label{wrt:eq:fluctuations gaussiennes wrt}
	n^{\frac{1-\gamma}{2}}\cdot \frac{(n^{-(1-\gamma)}\cdot \deg^+_{\ttT_n}(u_1)-w_1)}{\sqrt{n^{-(1-\gamma)}\cdot \deg^+_{\ttT_n}(u_1)}} \overset{\mathrm{(d)}}{\underset{n\rightarrow\infty}{\longrightarrow}}\mathcal{N}(0,1).
\end{align}

Now let us transfer this result to the case of preferential attachment trees. For this, it suffices to prove that $\mathbf{a}$ satisfies the condition \eqref{wrt:assum:ad hoc pa assum for gaussian limit} with $\delta>\frac{1}{2(c+1)}$ then the corresponding sequence $(\mathsf{M}^\mathbf{a}_n)_{n\geq 1}$ defined \eqref{wrt:eq:def mna} almost surely satisfies \eqref{wrt:assum:wrt gaussian limit} for $\gamma=\frac{c}{c+1}$. 
From Proposition~\ref{wrt:prop:assum pa implies assum wrt} and the definition of   $(\mathsf{M}_n^\mathbf{a})_{n\geq 1}$ as a scaled version of $(\mathsf{W}_n^\mathbf{a})_{n\geq 1}$, we know that we have  $\mathsf{M}_n^\mathbf{a}\underset{n\rightarrow\infty}{=}\frac{1}{1-\gamma}\cdot n^{\gamma}\cdot (1+ \grandO{n^{-\epsilon}})$ almost surely, for $\gamma=\frac{c}{c+1}$ and some $\epsilon>0$. 
Going along the proof of Proposition~\ref{wrt:prop:assum pa implies assum wrt}, we get from \eqref{wrt:eq:asymptotic behaviour Wn with quantified error} that 
\begin{align*}
	\mathsf{M}^\mathbf{a}_n\underset{n\rightarrow\infty}{=}\frac{1}{1-\gamma}\cdot n^{\gamma}\cdot (1+\grandO{n^{-\zeta}})
\end{align*}
for any $\zeta<\delta\wedge \frac{1}{2}$, so that \eqref{wrt:assum:wrt gaussian limit} is almost surely satisfied by $(\mathsf{M}_n^\mathbf{a})_{n\geq 1}$ if 
\begin{align*}
\delta> \frac{1-\gamma}{2}=\frac{1}{2(c+1)}.
\end{align*}
Now, thanks to Theorem~\ref{wrt:thm:connection PA WRT}, conditionally on the sequence $(\mathsf{M}_n^\mathbf{a})_{n\geq 1}$ the distribution of $(\ttP_n)_{n\geq 1}$ is $\wrt((\mathsf{m}^\mathbf{a}_n)_{n\geq 1})$. Applying \eqref{wrt:eq:fluctuations gaussiennes wrt} in this case finishes to prove \ref{wrt:it:gaussian fluctuations}. 

\end{proof}
\subsection{Applications to some other models of preferential attachment}\label{wrt:sec:other model of PA}
Let us present here another model of preferential attachment which appears in the literature, for example in \cite{pekoez_joint_2017}. This model does not produce a tree as ours does, but we can couple them in such a way that some of their features coincide. We only focus on one particular model of graph here but the method presented here can adapt to other similar models.

\paragraph{A model of $(m,\alpha)$-preferential attachment}Let $\mathtt S$ be a non-empty graph, with vertex-set $\{v_1^{(1)},\dots,v_1^{(k)}\}$ which have degrees $(d_1, \dots d_k)$, and $m\geq 2$ an integer and $\alpha>-m$ a real number such that $\alpha+d_i>0$ for all $1\leq i \leq k$. 
The model is then the following: we let $\mathtt G_1=\mathtt S$. Then, at any time $n\geq 1$, the graph $\mathtt G_{n+1}$ is constructed from the graph $\mathtt G_n$ by:
\begin{itemize}
	\item adding a new vertex labelled $v_{n+1}$ with $m$ outgoing edges,
	\item choosing sequentially to which other vertex each of these edges are pointed, each vertex being chosen with probability proportional to $\alpha$ plus its degree (the degree of the vertices are updated after each edge-creation).
\end{itemize}
The degree of a vertex in a graph refers in this section to the number of edges incident to it. Here the growth procedure in fact produces \emph{multigraphs}, in which it is possible for two vertices to be connected to each other by more than one edge. In this case, all those edges contribute in the count of their degree.

We can couple this model to a preferential attachment tree with sequence of fitnesses $\mathbf{a}$ defined as:
\begin{align*}
\mathbf{a}= (w(\mathtt S),\underset{m-1}{\underbrace{0,0,\dots,0}}, m+\alpha,\underset{m-1}{\underbrace{0,0,\dots,0}},m+\alpha,0,0 \dots),
\end{align*}
where $w(\mathtt S):=d_1+d_2+\dots+d_k+k\alpha$.

Indeed, we can construct $(\ttP_n)$ with distribution $\pa(\mathbf{a})$. Then, for any $n\geq 1$, consider the tree $\ttP_{1+m(n-1)}$ and for all $2\leq i\leq n$, merge together each vertex with fitness $m+\alpha$ together with the $m-1$ vertices with fitness $0$ that arrived just before it. If $\mathtt G_1$ only contains one vertex, it is immediate that the obtained sequence of graphs has exactly the same distribution as $(\mathtt G_n)_{n\geq 1}$. For general seed graphs $\mathtt S$, we can still use the same construction and the obtained sequence of graphs has the same evolution as some sequence $(\widetilde{\mathtt G}_n)_{n\geq 1}$ which would be obtained from $(\mathtt G_n)_{n\geq 1}$ by merging all the vertices $\{v_1^{(1)},\dots,v_1^{(k)}\}$ into a unique vertex $v_1$.

Note that a similar construction would also be possible if the degrees of the vertices $v_2,v_3,\dots$ were given by a sequence of integers $(m_2,m_3,\dots)$ instead of all being equal to some constant value $m$. This is for example the case in the model studied in \cite{deijfen_preferential_2009}, where the degrees are random. 

We have the following convergence for degrees of vertices in the graph, as $n\rightarrow\infty$.
\begin{proposition}\label{wrt:cor:convergence degree sequence m,delta preferential attachment}
	The following convergence holds almost surely in any $\ell^p$ with $p>2+\frac{\alpha}{m}$:
	\begin{multline*}
	n^{-\frac{1}{2+\alpha/m}}(\deg_{\ttG_n}(v_1^{(1)}),\deg_{\ttG_n}(v_1^{(2)}),\dots,\deg_{\ttG_n}(v_1^{(k)}),\deg_{\ttG_n}(v_2),\deg_{\ttG_n}(v_3),\dots)\\
	\underset{n\rightarrow\infty}{\longrightarrow} (\mathsf N_1\cdot B^{(1)},\mathsf N_1\cdot B^{(2)},\dots \mathsf N_1\cdot B^{(k)},\mathsf N_2-\mathsf N_1,\mathsf N_3-\mathsf N_2,\dots),
	\end{multline*}
	where \begin{align*}
	(B^{(1)},B^{(2)},\dots B^{(k)})\sim \mathrm{Dir}(d_1+\alpha,d_2+\alpha,\dots,d_k+\alpha),
	\end{align*}
	and the process $(\mathsf N_n)_{n\geq 1}$ is independent of $(B^{(1)},B^{(2)},\dots, B^{(k)})$. 
	
	Furthermore, whenever $\alpha\in \mathbb{Z}$ with $\alpha>-m$ or $m=1$ then the distribution of $(\mathsf N_n)_{n\geq 1}$ is explicit and given by: 
	\begin{itemize}
		\item if $\alpha \in \Z$ with $\alpha>-m$, then
		\[\frac{m^{\frac{-2m}{2m+\alpha}}}{2m+\alpha}\cdot (\mathsf N_n)_{n\geq 1} \sim \mathrm{IPGGP}(w(\mathtt{S}),\underset{m-1}{\underbrace{0,0,\dots,0}},m+\alpha)\]
		\item if $m=1$, then \[(\mathsf N_n)_{n\geq 1}\sim \MLMC\left(\frac{1}{2+\alpha},\frac{w(\mathtt S)}{2+\alpha}\right).\]
	\end{itemize}
\end{proposition}
This result strengthens the one of \cite[Theorem~1, Theorem~2 and Proposition~1]{pekoez_joint_2017} which corresponds (up to some definition convention) to the case $\alpha=1-m$. We emphasize that the convergence here  is almost sure in an $\ell^p$ space.

\begin{proof}[Proof of Proposition~\ref{wrt:cor:convergence degree sequence m,delta preferential attachment}]

Using the coupling argument, we know that we can construct jointly the sequence of graphs $(\mathtt{G}_n)_{n\geq 1}$ and a sequence of trees $(\ttP_n)_{n\geq 1}\sim\pa(\mathbf{a})$ with fitness sequence \[\mathbf{a}=(	w(S),\underset{m-1}{\underbrace{0,0,\dots,0}},m+\alpha,\underset{m-1}{\underbrace{0,0,\dots,0}},m+\alpha,\dots),\]
in such a way that for every $n\geq1$, the sequence
\begin{align*}
	\left(\sum_{j=1}^k\left(\deg_{\ttG_n}(v^{(j)}_1)-d_j\right),(\deg_{\ttG_n}(v_2)-m),\dots,(\deg_{\ttG_n}(v_n)-m),0,0\dots\right)
\end{align*}
coincides with
\begin{align*}
	(\deg_{\ttP_{1+(n-1)m}}^+(u_1),\deg_{\ttP_{1+(n-1)m}}^+(u_{1+m}),\deg_{\ttP_{1+(n-1)m}}^+(u_{1+2m}),\dots).
\end{align*}
Using this connection and Theorem~\ref{wrt:thm:connection PA WRT}, Proposition~\ref{wrt:prop:assum pa implies assum wrt} and Proposition~\ref{wrt:prop:control wkn} we get
	\begin{multline*}
n^{-\frac{1}{2+\alpha/m}}(\deg_{\ttG_n}(v_1)+\deg_{\ttG_n}(v_2)+\dots+\deg_{\ttG_n}(v_k),\deg_{\ttG_n}(u_2),\deg_{\ttG_n}(u_3),\dots)\\
\underset{n\rightarrow\infty}{\longrightarrow} (\mathsf N_1,\mathsf N_2-\mathsf N_1,\mathsf N_3-\mathsf N_2,\dots),
\end{multline*}
almost surely in $\ell^p$ for all $p>2+\frac{\alpha}{m}$, for some random sequence $(\mathsf N_n)_{n\geq 1}$. Note that the time-change between $(\mathtt{G}_n)_{n\geq 1}$ and $(\mathtt{P}_n)_{n\geq1}$ is responsible for an extra factor in the scaling, so that the sequence $(\mathsf N_n)_{n\geq 1}$ has the distribution of $m^{\frac{m}{2m+\alpha}}\cdot (\mathsf{M}_n^\mathbf{a})_{n\geq 1}$. 
In the case $\alpha\in\Z$ or $m=1$, Proposition~\ref{wrt:prop:distribution limiting chains} identifies the distribution of the limiting sequence. 

Last, the convergence of $\frac{1}{\sum_{j=1}^{k}\deg_{\ttG_n}(v^{(j)}_1)}\cdot (\deg_{\ttG_n}(v^{(1)}_1),\deg_{\ttG_n}(v^{(2)}_1),\dots,\deg_{\ttG_n}(v^{(k)}_1))$ just follows from the classical result of convergence for the proportion of balls in a P\'olya urn.
\end{proof}
\appendix
\section{Technical proofs and results}\label{wrt:app:computations}
This appendix contains the proofs of technical results that are used throughout this paper.
Let start by stating a useful conditional version of the Borel-Cantelli lemma. 
\begin{lemma}\label{wrt:lem:sum bernoulli}Let $(\cF_n)$ be a filtration and let $(B_n)_{n\geq 1}$ be a sequence of events adapted to this filtration. For all $n\geq 1$, let $\mathsf p_n:=\Ppsq{B_n}{\cF_{n-1}}$. We have
	\begin{align*}
	\frac{\sum_{i=1}^{n}\mathbf{1}_{B_i}}{\sum_{i=1}^{n}\mathsf p_i}\underset{n\rightarrow\infty}{\rightarrow} 1 \qquad \text{a.s.\ on the event} \quad \left\lbrace\sum_{i=1}^{\infty}\mathsf p_i=\infty\right\rbrace
	\end{align*}
	and also
	\begin{align*}
	\sum_{i=1}^{n}\mathbf{1}_{B_i} \quad \text{converges a.s.\ on the event} \quad   \left\lbrace\sum_{i=1}^{\infty}\mathsf p_i<\infty\right\rbrace.
	\end{align*}
\end{lemma}
\begin{proof}
	The first convergence is the content of Theorem~5.4.11 and the second one is an application of Theorem~5.4.9 to the martingale $\left(\sum_{i=1}^{n}(\mathbf{1}_{B_i}-\mathsf p_i)\right)_{n\geq 1}$, both taken from \cite{durrett_probability_2010}.
\end{proof}

The following lemma is a rewriting of \cite[Lemma~1]{biggins_uniform_1992}. We provide the proof for completeness. 
\begin{lemma}\label{wrt:lem:biggins lemma}
Let $(M_n)_{n\geq 1}$ be a complex-valued martingale with finite $q$-th moment for some $q\in\intervalleff{1}{2}$. Then for every $n\geq 1$ we have
\begin{align*}
	\Ec{\abs{M_{n+1}}^q}\leq \Ec{\abs{M_{n}}^q}+ 2^q \cdot \Ec{\abs{M_{n+1}-M_n}^q}.
\end{align*}
\begin{proof}
	Let $X_{n+1}:= M_{n+1}-M_n$ and let $X_{n+1}'$ be a random variable such that conditionally on $(M_1,\dots,M_n)$ the random variable $X_{n+1}'$ is independent of, and has the same distribution as $X_{n+1}$.
	Then 
	\begin{align*}
\Ec{\abs{M_{n+1}}^q}&= \Ec{\bigabs{\Ecsq{M_{n+1}-X_{n+1}'}{M_1,\dots M_{n+1}}}^q}\\
&\leq \Ec{\abs{M_{n+1}-X_{n+1}'}^q}\\
&= \Ec{\abs{M_{n}+X_{n+1}-X_{n+1}'}^q}\\
&\leq \Ec{\abs{M_{n}}^q}+\Ec{\abs{X_{n+1}-X_{n+1}'}^q}\\
&\leq \Ec{\abs{M_{n}}^q}+2^q\cdot \Ec{\abs{X_{n+1}}^q},
\end{align*}
where the first equality comes from the fact that $\Ecsq{X_{n+1}'}{M_1,\dots M_{n+1}}=0$. The first inequality is the one of Jensen for conditional expectation, applied to the convex function $z\mapsto \abs{z}^q$. The second inequality is due to Clarkson, see \cite[Lemma~1]{bahr_inequality_1964}, and can be applied because the distribution of $X_{n+1}-X_{n+1}'$ conditional on $M_n$ is symmetric and $1\leq q \leq 2$. The last inequality comes from the triangle inequality for the $L^q$-norm.  
\end{proof}
\end{lemma}
 Let us state another result about martingales, which we use numerous times throughout the paper. Recall our uniform big-$O$ and small-$o$ notation, introduced in \eqref{wrt:eq:big and small o notation}. 
\begin{lemma}\label{wrt:lem:convergence analytic martingale}
	Suppose that $(z \mapsto Z_n(z))_{n\geq 1}$ is a sequence of analytic functions on some open set $\mathscr O\subset \mathbb C$, adapted to some filtration $(\cG_n)$. Suppose that for every $z\in \mathscr O$, the sequence $(Z_n(z))_{n\geq 1}$ is a martingale with respect to the filtration $(\cG_n)$. If there exists a parameters $q>1$ and continuous functions $\alpha:\mathscr O\rightarrow\R$ and $\delta:\mathscr O\rightarrow \intervalleoo{0}{\infty}$ such that for all $n\geq 1$ we have
	\begin{align*}
	\Ec{\abs{Z_{2n}(z)-Z_n(z)}^q}=\grandOdom{\mathscr O}{n^{\alpha(z) q - \delta(z)+\petitodom{\mathscr O}{1}}},
	\end{align*}
	then for any compact subset $K\subset \mathscr O$, there exists $\epsilon(K)>0$ such that
	\begin{enumerate}
		\item\label{wrt:it:convergence gamma>0} if $\alpha>0$ on $\mathscr O$ we have $n^{-\alpha(z)}\cdot\abs{Z_n(z)-Z_1(z)}=\grandOdom{K}{n^{-\epsilon(K)}}$ almost surely and also in expectation,
		\item\label{wrt:it:convergence gamma<0} if $\alpha\leq0$ on $\mathscr O$, the almost sure limit $Z_\infty(z)$ exists for $z\in \mathscr O$ and we have $n^{-\alpha(z)}\cdot\abs{Z_n(z)-Z_\infty(z)}=\grandOdom{K}{n^{-\epsilon(K)}}$ almost surely and also in expectation.
	\end{enumerate}
\end{lemma}
\begin{proof}[Proof of Lemma~\ref{wrt:lem:convergence analytic martingale}]
	First, without loss of generality, we can consider that the term $\petitodom{\mathscr O}{1}$ is identically equal to $0$, otherwise we just replace the function $z\mapsto\delta(z)$ by $z\mapsto\frac{1}{2}\cdot \delta(z)$. 
	Second, by compactness, it is sufficient to prove the result for a small disk around each $x\in K$. Since $\mathscr O$ is an open set, let $\rho>0$ be such that $\mathrm D(x,2\rho)\subset \mathscr O$, where $\mathrm D(x,2\rho)$ is the closed disk in the complex plane with centre $x$ and radius $2\rho$. We denote
	\begin{align*}
		\underline\alpha=\inf_{\mathrm D(x,2\rho)}\alpha, \quad \overline \alpha=\sup_{\mathrm D(x,2\rho)}\alpha, \quad \underline\delta=\inf_{\mathrm D(x,2\rho)}\delta,
	\end{align*}	
	and choose $\rho$ small enough so that $\underline\alpha-\overline \alpha+\frac{1}{q}\underline\delta >0$.
	Then if we let $\xi:\intervalleff{0}{2\pi}\rightarrow \mathbb C$ such that $\xi(t)=x+2\rho e^{it}$, we have for any $n$ and $m$, using the Cauchy formula
	\begin{align*}
	\sup_{z\in \mathrm D(x,\rho)}\abs{Z_n(z)-Z_m(z)}\leq \pi^{-1} \int_{0}^{2\pi}\abs{Z_n(\xi(t))-Z_m(\xi(t))}\dd t.
	\end{align*}
	Now,
	\begin{align}\label{wrt:eq:upper bound using cauchy formula}
	\sup_{2^s\leq n\leq 2^{s+1}}\sup_{z\in \mathrm D(x,\rho)}\abs{Z_n(z)-Z_{2^s}(z)}&\leq \pi^{-1}\sup_{2^s\leq n\leq 2^{s+1}} \int_{0}^{2\pi}\abs{Z_n(\xi(t))-Z_{2^s}(\xi(t))}\dd t \notag \\
	&\leq \pi^{-1} \int_{0}^{2\pi} \sup_{2^s\leq n\leq 2^{s+1}} \abs{Z_n(\xi(t))-Z_{2^s}(\xi(t))}\dd t.
	\end{align}
	Using sequentially Jensen's inequality and Doob's maximal inequality in $L^q$, gives us for every $z\in \mathrm D(x,\rho)$:
	\begin{align}\label{wrt:eq:expectation Tn}
	\Ec{\sup_{2^s\leq n\leq 2^{s+1}}\abs{Z_n(z)-Z_{2^s}(z)}}&\leq \Ec{\sup_{2^s\leq n\leq 2^{s+1}}\abs{Z_n(z)-Z_{2^s}(z)}^q}^{\frac{1}{q}} \notag \\
	&\leq \frac{q}{q-1}\cdot \Ec{\abs{Z_{2^{s+1}}(z)-Z_{2^s}(z)}^q}^{\frac{1}{q}}\notag \\
	&\underset{s\rightarrow\infty}=\grandOdom{\mathrm D(x,2\rho)}{2^{\left(\overline\alpha-\frac{1}{q}\cdot\underline\delta\right) s}}.
	\end{align}
	So using \eqref{wrt:eq:upper bound using cauchy formula}, Fubini's theorem and \eqref{wrt:eq:expectation Tn}, we get
	\begin{align}\label{growing:eq:control expectation sup sup}
	\Ec{\sup_{2^s\leq n\leq 2^{s+1}}\sup_{z\in \mathrm D(x,\rho)}\abs{Z_n(z)-Z_{2^s}(z)}}&\leq \pi^{-1} \int_{0}^{2\pi} \Ec{\sup_{2^s\leq n\leq 2^{s+1}} \abs{Z_n(\xi(t))-Z_{2^s}(\xi(t))}}\dd t \notag \\
		&\underset{s\rightarrow\infty}=\grandO{2^{\left(\overline\alpha-\frac{1}{q}\cdot\underline\delta\right) s}}.
	\end{align}
	Now let us treat the two cases $\alpha>0$ and $\alpha\leq0$ separately. Remark that the quantity $\left(\overline\alpha-\frac{1}{q}\cdot\underline\delta\right)$ is negative when $\alpha\leq 0$, but can be of any sign in the case $\alpha>0$. 
	
	$\bullet$ For $\alpha>0$ and $r\geq 0$, we let
	\begin{align*}
	A_r:= 2^{-\underline\alpha r}\cdot \sum_{s=0}^{r}\sup_{2^s\leq n\leq 2^{s+1}}\sup_{z\in \mathrm D(x,\rho)}\abs{Z_n(z)-Z_{2^s}(z)}.
	\end{align*}
	Using \eqref{growing:eq:control expectation sup sup}, we have
	\begin{align*}
	\Ec{A_r}
	&\leq \cst \cdot 2^{-\underline\alpha r}\cdot \sum_{s=0}^{r}2^{\left(\overline\alpha-\frac{1}{q}\cdot\underline\delta\right) s}\leq \cst \cdot 2^{-(\underline\alpha -0\vee (\overline\alpha-\frac{1}{q}\cdot\underline\delta))r}.
	\end{align*}
	Thanks to our assumptions, the number $\beta:=\underline\alpha -0\vee (\overline\alpha-\frac{1}{q}\cdot\underline\delta)$ is positive. Using Markov's inequality  and the last display yields
	\begin{align*}
	\Pp{ A_r> 2^{-\frac{\beta}{2}r}}\leq 2^{\frac{\beta}{2}r} \cdot \Ec{A_r}
	\leq \cst\cdot  2^{-\frac{\beta}{2}r},	\end{align*}
	which is summable, so the Borel-Cantelli lemma ensures that $A_r=\grandO{2^{-\frac{\beta}{2}r}}$ almost surely as $r\rightarrow\infty$. 
	Now for any $n\geq 1$, there is a unique integer $r_n$ such that $2^{r_n}\leq n < 2^{r_n+1}$, namely $r_n=\lfloor\log_2(n)\rfloor$, and we write
	\begin{align*}
		n^{-\alpha(z)}\cdot\abs{Z_n(z)-Z_1(z)}\leq n^{-\alpha(z)}\sup_{1\leq k \leq n}\sup_{z\in \mathrm D(x,\rho)}\abs{Z_n(z)-Z_1(z)} \leq A_{r_n},
	\end{align*}
	which proves point \ref{wrt:it:convergence gamma>0}, because almost surely $A_{r_n}=\grandO{n^{-\frac{\beta}{2}}}$ as $n\rightarrow\infty$.
	
	$\bullet$ For $\alpha\leq 0$, the reasoning is similar so we use the same notation for slightly different quantities. For any integer $r\geq 0$ we let
	 \begin{align*}
A_r:= 2^{-\underline\alpha (r+1)} \cdot\sum_{s=r}^{\infty}\sup_{2^s\leq k\leq 2^{s+1}}\sup_{z\in \mathrm D(x,\rho)}\abs{Z_k(z)-Z_{2^s}(z)}
	 \end{align*}
	 Then, thanks to \eqref{growing:eq:control expectation sup sup}, we have
	 \begin{align*}
	 \Ec{A_r}&\leq \cst \cdot 2^{-\underline\alpha (r+1)} \cdot\sum_{s=r}^{\infty}{2^{\left(\overline\alpha-\frac{1}{q}\cdot\underline\delta\right) s}}\leq \cst \cdot 2^{-(\underline\alpha - \overline\alpha+\frac{1}{q}\cdot\underline\delta)r},
	 \end{align*}
	 and thanks to our assumption the number $\beta:= \underline\alpha - \overline\alpha+\frac{1}{q}\cdot\underline\delta$ is positive. Using the same arguments as in the case $\alpha>0$ we have $A_r=\grandO{2^{-\frac{\beta}{2}r}}$ almost surely as $r\rightarrow\infty$ and taking $r_n:= \lfloor \log_2(n)\rfloor$ yields
	\begin{align*}
	n^{-\alpha(z)}\sup_{k \geq n}\sup_{z\in \mathrm D(x,\rho)}\abs{Z_k(z)-Z_n(z)}& \leq A_{r_n}.
	\end{align*}
	Again we almost surely have $A_{r_n}=\grandO{n^{-\frac{\beta}{2}}}$ as $n\rightarrow\infty$.
	This ensures that the sequence of functions $(z\mapsto Z_n(z))_{n\geq 1}$ is almost surely a Cauchy sequence for the uniform convergence on the disc $\mathrm D(x,\rho)$ (so that its limit $z\mapsto Z_\infty(z)$ is well-defined on the disk) and that \ref{wrt:it:convergence gamma<0} is satisfied.
\end{proof}

Finally, let us give a proof of Lemma~\ref{wrt:lem:behaviour sum w/W}.
\begin{proof}[Proof of Lemma~\ref{wrt:lem:behaviour sum w/W}]
From the assumption, we know that there exists $\epsilon>0$ such that $W_n\underset{}{=} \cst \cdot n^{\gamma}+\grandO{n^{\gamma-\epsilon}}$ as $n\rightarrow\infty$. Without loss of generality, we can assume that $\epsilon<1$. 
Then it is immediate that $w_n=W_{n+1}-W_n=\grandO{n^{\gamma-\epsilon}}$. 
Then 
\begin{align*}
\sum_{i=n}^{2n}\left(\frac{w_i}{W_i}\right)^2\leq \frac{1}{W_n^2}\cdot \max_{n\leq i \leq 2n} w_i \cdot \sum_{i=n}^{2n}w_i \leq \frac{W_{2n}}{W_n^2}\cdot \max_{n\leq i \leq 2n} w_i = \grandO{n^{-\epsilon}},
\end{align*}
and the first point follows by summing over intervals of the type $\intervalleentier{n2^k}{n2^{k+1}}$.

Now write
\begin{align*}
\frac{W_1}{W_n}=\prod_{i=2}^{n}\frac{W_{i-1}}{W_i}=\prod_{i=2}^{n}\left(1-\frac{w_i}{W_i}\right)=\exp\left(\sum_{i=2}^n\log\left(1-\frac{w_i}{W_i}\right)\right).
\end{align*}
Since $\frac{w_i}{W_i}\rightarrow 0$ as $i\rightarrow \infty$, we get
\begin{align*}
\log\left(1-\frac{w_i}{W_i}\right)=-\frac{w_i}{W_i}+\grandO{\left(\frac{w_i}{W_i}\right)^2}
\end{align*}
Putting everything together, we get
\begin{align*}
	\sum_{i=2}^{n}\frac{w_i}{W_i}&=-\sum_{i=2}^{n}\log\left(1-\frac{w_i}{W_i}\right) + \sum_{i=2}^{n}\grandO{\left(\frac{w_i}{W_i}\right)^2}\\
	&=\log W_n-\log W_1+\sum_{i=2}^{\infty}\grandO{\left(\frac{w_i}{W_i}\right)^2}-\grandO{\sum_{i=n+1}^{\infty}\left(\frac{w_i}{W_i}\right)^2}\\
	&=\log W_n + \cst + \grandO{n^{-\epsilon}}.
\end{align*}
Last, just remark that $\log W_n= \log(\cst \cdot n^{\gamma}\cdot (1+\grandO{n^{-\epsilon}}))=\gamma \log n + \cst + \grandO{n^{-\epsilon}},$ which finishes the proof.
\end{proof}
\printbibliography
\end{document}